\documentclass[numbook, envcountsect, envcountsame, envcountreset, runningheads, smallextended]{svjour3}

\usepackage[authoryear,round]{natbib}                   
\usepackage{amsmath}                                                        
\numberwithin{equation}{section}
\usepackage{graphicx}                                                       
\usepackage{subfigure}
\usepackage{enumitem}                                                    
\usepackage{hyperref}
\usepackage{amssymb}                                                        
\usepackage[mathscr]{eucal}                                             
\usepackage{cancel}                                                             
\usepackage[normalem]{ulem}                                                                 
\usepackage{pstricks}
\usepackage{rotating}
\usepackage{lscape}
\usepackage[paperwidth=8.5in,paperheight=11in,top=1.00in, bottom=1.00in, left=1.00in, right=1.00in]{geometry}
\usepackage{mathtools}                                                      
\mathtoolsset{showonlyrefs=true}                                    
\usepackage{fixltx2e,amsmath}                                           
\MakeRobust{\eqref}

\usepackage{mathdots}
\allowdisplaybreaks                                                             
\newtheorem{notation}[theorem]{Notation}
\newtheorem{assumption}[theorem]{Assumption}

\def \y {{\eta}}
\def \s {{\sigma}}
\def \g {{\gamma}}
\def \a {{\alpha}}
\def \b {{\beta}}

\def \R {\mathbb{R}}
\def \p {\partial}
\def \t {\tau}

\newcommand{\<}{\langle}
\renewcommand{\>}{\rangle}
\renewcommand{\(}{\left(}
\renewcommand{\)}{\right)}

\newcommand\Rb{\mathbb{R}}

\newcommand\Ac{\mathscr{A}}
\newcommand\Acc{\bar{\mathscr{A}}}
\newcommand\ac{\bar{a}}

\newcommand\Bc{\mathscr{B}}

\newcommand\Gc{\mathscr{G}}

\newcommand\Lc{\mathscr{L}}

\newcommand\Nc{\mathscr{N}}

\newcommand\eps{\varepsilon}
\newcommand\e{\varepsilon}

\newcommand\sig{\sigma}

\newcommand\lam{\lambda}
\newcommand\del{\delta}

\newcommand\Cv{\mathbf{C}}
\newcommand\mv{\mathbf{m}}

\renewcommand\d{\delta}

\newcommand\ii{\mathtt{i}}
\newcommand\dd{d}
\newcommand\ee{\mathrm{e}}
\newcommand\BS{\mathrm{BS}}

\def \ub {u^{\text{\rm BS}}}

\def \tens {\Lambda}
\def \tensbis {\Theta}

\newcommand\FF{u^{\BS}}
\newcommand\GG{\s^{\BS}}

\newcommand\DomY{\R^{d-1}}

\newcommand\DomXY{\R\times\DomY}

\def \ce {{M}}
\def \cem {{M_{0}}}
\def \Gbs {{\Gamma_{0}}}
\def \Gb {\bar{\Gamma}}

\def \R  {{\mathbb {R}}}
\def \x {{\xi}}
\def \g {{\gamma}}
\def \e {{\varepsilon}}
\def \eps {{\varepsilon}}
\def \t {{\tau}}
\def \n {{\nu}}
\def \m {{\mu}}
\def \y {{\eta}}
\def \th {{\theta}}
\def \z {{\zeta}}
\def \p {{\partial}}
\def \a {{\alpha}}
\def \O {{\Omega}}

\def \d {{\delta}}

\def \TT {\mathbf{T}}

\def \k {{\kappa}}

\def \a {{\alpha}}
\def \b {{\beta}}
\def \d {{\delta}}

\def \G {\Ga}

\def \Ga {{\Gamma}}

\def \s {{\sigma}}

\def \R {{\mathbb {R}}}
\def \N {{\mathbb {N}}}

\def \x {{\xi}}
\def \e {{\varepsilon}}
\def \eps {{\varepsilon}}
\def \r {{\varrho}}

\def \t {{\tau}}
\def \t {{\tau}}
\def \n {{\nu}}

\def \m {{\mu}}
\def \y {{\eta}}
\def \th {{\theta}}
\def \z {{\zeta}}

\def \g {{\gamma}}
\def \O {{\Omega}}
\def \phi {{\varphi}}

\def \tilde {\widetilde}

\def\l {\lambda}

\def \F {\mathcal{F}}

\def \à {{\`a }}
\def \è {{\`e }}
\def \ò {{\`o }}
\def \ù {{\`u }}

\def \cy {{\y}}

\newcommand\Hb{\bar{H}}
\newcommand\Ho{H}

\newcommand\Cdt{\tilde{C}}
\newcommand\Cdb{\bar{C}}

\newcommand\xfixed{x_0}
\newcommand\yfixed{y_0}

\begin{document}

\title{The exact Taylor formula of the implied volatility}

\author{
Stefano Pagliarani
\thanks{DEAMS, Universit\`a di Trieste, Trieste, Italy.
\textbf{e-mail}: pagliarani@cmap.polytechnique.fr. Work supported by the Chair {\it Financial
Risks} of the {\it Risk Foundation}.} \and Andrea Pascucci
\thanks{Dipartimento di Matematica, Universit\`a di Bologna, Bologna, Italy.
\textbf{e-mail}: andrea.pascucci@unibo.it}
}

\date{This version: \today}

\maketitle

\begin{abstract}
In a model driven by a multi-dimensional local diffusion, we study the behavior of implied
volatility $\s$ and its derivatives with respect to log-strike $k$ and maturity $T$ near expiry
and at the money. We recover explicit limits of the derivatives ${\partial_T^q}
\partial_k^m \sigma$ for $(T,x-k)$ approaching the origin within the parabolic region $|x-k|\leq
\lam \sqrt{T}$, with $x$ denoting the spot log-price of the underlying asset and where $\l$ is a
positive and arbitrarily large constant. Such limits yield the exact Taylor formula for
implied volatility
within the parabola $|x-k|\leq \lam \sqrt{T}$. {In order to include important models of interest
in mathematical finance, e.g. Heston, CEV, SABR, the analysis is carried out under the weak
assumption that the infinitesimal generator of the diffusion is only locally elliptic.}
\end{abstract}




%
%

\section{Introduction}
\label{sec:intro}
\par

This paper 
{deviates} from the mainstream literature on asymptotic methods in finance; in fact, our main
result does not add another formula to the plethora of approximation formulas for the implied
volatility (IV) already available in the literature.
Rather, we prove an {\it exact result:} 
a rigorous derivation of the {exact Taylor formula of IV}, as a \emph{function of both strike and
maturity}, in a parabolic region close to {expiry} and {at-the-money} {(ATM)}.

This is done under general assumptions that allow to include popular models, such as the CEV and
the Heston models, as very particular cases: indeed, we consider a multivariate model driven by a
stochastic process that is a {\it local diffusion} in a sense that suitably generalizes the
classical notion of diffusion as given by \cite{StroockVaradhan} and
\cite{FriedmanSDE1,FriedmanSDE2}.

The literature on IV asymptotics is extensive and {exploits} a diverse range of mathematical
techniques. Focusing on short-time asymptotics, well-known results were obtained by
\cite{berestycki2002asymptotics}, \cite{berestycki-busca-florent} and \cite{Durrleman2010}.
Deferring precise definitions until the body of this paper, we denote by $\s(t,x;T,k)$ the IV
related to a Call option with log-strike $k$ and maturity $T$, where $x$ is the spot log-price of
the underlying asset at time $t$. \cite{berestycki-busca-florent} {uses PDEs techniques} to prove
the existence of the limits {$\lim\limits_{T\rightarrow t^+} \s(t,x;T,k)$} in a generic stochastic
volatility model and to characterize such limits in terms of Varadhan's geodesic distance (see
also to \cite{GavalasYortsos} for related results). More recently, {\cite{Durrleman2010} {gives
conditions} under which it is possible to recover the ATM-limits {$\lim\limits_{T\rightarrow t^+}{
\partial_T^q}
\partial_k^m \sigma(t,k;T,k)$} using a semi-martingale decomposition of implied volatilities;
although this approach performs also in non-Markovian settings, the validity of the conditions for
the existence of the limits is verified only under Markovian assumptions and employing
the results in \cite{berestycki-busca-florent}}.

While it is common practice to consider the IV as a function of maturity and strike $(T,k)$, the
aforementioned papers examine only the {\it vertical limits}, as $T\rightarrow t^+$, of
{$\sigma(t,x;T,k)$}. The aim of this paper is to give conditions for the existence and an explicit
representation of the limits of ${\partial_T^q}
\partial_k^m \sigma(t,x;T,k)$, {at any order $m,{q}$}, as $(T-t,x-k)$ approaches the origin within the parabolic region
$\mathcal{P}_{\l}:=\{|x-k|\leq \lam \sqrt{T-t}\}$; here $\l$ is an arbitrarily large positive
parameter. From a practical perspective, $\mathcal{P}_{\l}$ is the region of interest where
implied volatility data are typically observed in the market. As a by-product, we also provide a
rigorous and explicit derivation of the exact Taylor formula (see formula
\eqref{eq:Tay_form_intro2} below) for the implied volatility
$\sigma(t,x;\cdot,\cdot)$ in $\mathcal{P}_{\l}$, around $(T,k)=(t,x)$.

\begin{figure}[htb]
\centering
\includegraphics[width=1\textwidth,height=0.12\textheight]{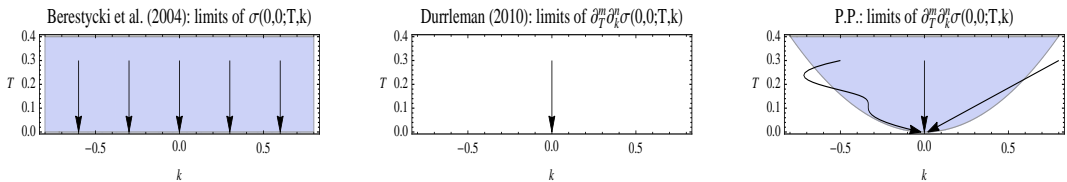}
\caption{Directions along which the limits are computed in \cite{berestycki-busca-florent}, in
\cite{Durrleman2010} and in this paper,  respectively.} \label{fig:limits}
\end{figure}

{The starting point is the analysis {of the transition density} first developed in a scalar
setting in \cite{pagliarani2011analytical} and later extended to asymptotic IV expansions in
multiple dimensions in \cite{LPP2},} {where the authors} derived a fully explicit approximation,
hereafter denoted by $\bar{\sigma}_N$, for the IV at any given order $N\in\N$. Our main result,
Theorem \ref{th:IV_error} {below}, gives a sharp error bound on {${\partial_T^q}\partial_k^m
(\sigma-\bar{\sigma}_N)$} and leads to the existence of the limits
\begin{equation}\label{eq:IV_limit}
 \lim_{(T,k)\to (t,x)\atop |x-k|\leq \lam \sqrt{T-t}} {\partial_T^q} \partial_k^m  \big(\sigma-\bar{\sigma}_N\big)(t,x;T,k)  = 0,\qquad 2{q}+m\leq N.
\end{equation}
{In the one-dimensional case and for derivatives of order less than or equal to two, similar
results were proved in \cite{BompisGobet2012} by using Malliavin calculus techniques. Our results
are proved under mild conditions on the driving stochastic process, which is assumed to be a
Feller process and an inhomogeneous local diffusion. Loosely speaking, we assume that the
infinitesimal generator of the diffusion is only \emph{locally elliptic} (i.e. elliptic on a
certain domain $D\subseteq \R^d$) and its coefficients satisfy suitable regularity conditions;
note that no ellipticity condition is imposed on the complementary set $\R^d\setminus D$. Results
under such general hypotheses appear to be novel compared to the existing literature. In
particular, our analysis includes processes with killing and/or degenerate processes: {\it our
assumptions do not even imply that the law of the underlying process has a density} and therefore
our results apply to many degenerate cases of interest, such as the well-known CEV, Heston and
SABR models, among others.}

Formula \eqref{eq:IV_limit} implies that the limits of the derivatives ${\partial_T^q}
\partial_k^m \sigma$ exist if and only if the limits of ${\partial_T^q} \partial_k^m
\bar{\sigma}_N$ do exist, and in that case we have
\begin{equation}\label{eq:IV_limit_bis2}
 \lim_{(T,k)\to (t,x)\atop |x-k|\leq \lam \sqrt{T-t}} {\partial_T^q} \partial_k^m  \sigma(t,x;T,k)  =  \lim_{(T,k)\to (t,x)\atop |x-k|\leq \lam \sqrt{T-t}} {\partial_T^q} \partial_k^m
 \bar{\sigma}_N(t,x;T,k).
\end{equation}
Note that, in general, the limits in \eqref{eq:IV_limit_bis2} do not exist: a simple example is
given in \cite{roper2009relationship}, Section 6, {who} exhibit a log-normal model with
oscillating time-dependent volatility. In that case the results by
\cite{berestycki2002asymptotics}, \cite{berestycki-busca-florent} and \cite{Durrleman2010} do not
apply, while the approximation $\bar{\sigma}_N$ in \cite{LPP4} turns out to be exact at order
$N=0$. More generally, we shall provide simple and explicit conditions ensuring the existence of
the limits of ${\partial_T^q} \partial_k^m \bar{\sigma}_N$, and consequently {the existence of
those of} ${\partial_T^q} \partial_k^m\s$ in \eqref{eq:IV_limit_bis2}. {A particular case is when
the underlying diffusion is time-homogeneous: in that case, $\bar{\sigma}_N$ is polynomial in time
and thus smooth up to $T=t$.}

Denoting by ${\partial_T^q} \partial_k^m \bar{\sigma}_N(t,x)$ the limits in
\eqref{eq:IV_limit_bis2}, {whose explicit expression is known at any order,} we get the following
{exact \emph{parabolic} Taylor formula for $\s$}:
\begin{equation}\label{eq:Tay_form_intro2}
 \s(t,x;T,k) = \sum_{2{q}+m\leq N} \frac{{\partial_T^q} \partial_k^m  {\bar{\sigma}_{N}(t,x)}}{{q}!m!}(T-t)^{q} (k-x)^m  + \text{\rm o}\left((T-t)^{\frac{N}{2}}+|k-x|^N\right),
\end{equation}
as $(T,k)\to(t,x)$ in $\mathcal{P}_{\l}$.
{Here, the meaning of the adjective \emph{parabolic} is twofold. On the one hand it refers to the
parabolic domain $\mathcal{P}_{\l}$ on which the Taylor formula is proved; on the other hand, it
refers to nature of the reminder, which is expressed in terms of  the homogeneous norm typically
used to describe the geometry induced by a parabolic differential operator.} Note that this
formula describes the behavior of
$\s$
in a joint regime of small log-moneyness and/or small maturity. This result appears to be novel
compared to the existing literature and complementary to \cite{lee2011asymptotics},
\cite{MijatovicTankov} and \cite{CaravennaCorbetta}. In \cite{lee2011asymptotics} the asymptotic
behavior of $\s$ in joint regime of extreme strikes and short/long time-to-maturity is studied; \cite{MijatovicTankov} studied,  in an exponential L\'evy model, the small-time asymptotic behavior of
$\s$ along relevant curves lying outside the parabolic region $\mathcal{P}_{\l}$ for any $\lam>0$; eventually, in
a very general setting, \cite{CaravennaCorbetta} studied the asymptotics of $\s$ for different
regimes of log-strikes and maturities, including the region $\mathcal{P}_{\l}$ where their result
coincides with ours at order zero.

A part from the mere interest of having at hand a Taylor formula like \eqref{eq:Tay_form_intro2},
additional advantages of having two-dimensional limits, as opposed to vertical ones, might come
from applications such as the asymptotic study of the IV generated by VIX options (see
\cite{barletta-pagliarani-nicolato}). In this case, the underlying value, given by the price of
the future-VIX, is not fixed but varies in time, meaning that the log-moneyness of an ATM VIX-Call
is not constantly zero, but approaches zero for small time-to-maturities along a curve which is
not a straight line.

The proof of our result proceeds in several steps. {We first introduce a notion of local diffusion
(Assumption \ref{assum1i}): we study its basic properties and the existence of a local transition
density. 
We provide a double characterization of the local density in terms of the forward and the backward
Kolmogorov equations (Theorem \ref{la1}): the forward representation follows from H\"ormander's
theorem and is coherent with the classical results by \cite{kusuoka-stroock}. On the other hand,
the backward representation appears to be novel at this level of generality. Indeed, its proof is
more delicate and requires the use the Feller property combined with the classical pointwise
estimates by \cite{moser71} for weak solutions of parabolic PDEs.} Then we derive sharp asymptotic
estimates for the derivatives ${\partial^q_T}\partial^m_k u(t,x;T,k)$, with $u$ representing the
pricing function of a Call option with maturity $T$ and log-strike $k$. This will be done first in
a uniformly parabolic framework and then will be extended to a \emph{locally} parabolic setting to
include the majority of the models used in mathematical finance. The second step is particularly
interesting due to the very loose assumptions imposed on the generator $\Ac_{t}$ of the underlying
diffusion. The main idea is to prolong $\Ac_{t}$ with an operator $\tilde{\Ac}_{t}$ which is
globally parabolic and then to prove that locally in space 
the difference between the fundamental solution of $\tilde{\Ac}_{t}$ and the local density of the
underlying process decays exponentially as the time-to-maturity approaches zero. This last step
requires an articulated use of some techniques first introduced by \cite{Safonov1998}. Eventually,
the estimates on the derivatives ${\p_T^q}\partial_k^m u$ are combined with some sharp estimates
on the inverse of the B\&S pricing function and on its sensitivities to obtain the main results,
Theorem \ref{th:IV_error} and the Taylor formula \eqref{eq:Tay_form_intro2}.

\smallskip The paper is organized as follows.
In Section \ref{sec:model} we describe the general setting and show some illustrative examples of
popular models satisfying our standing assumptions. In Section \ref{asec3} we briefly recall the
{asymptotic expansion procedure} proposed by \cite{LPP2}. In Section \ref{sec:error} we derive
error estimates for prices and sensitivities, first under the strong assumption of uniform
parabolicity (Subsection \ref{sec:price.error}) and then in the general case (Subsection
\ref{sec:price.error.local}). In Section \ref{sec:error.impvol} we prove our main result (Theorem
\ref{th:IV_error}) on the error estimates of the IV and its derivatives, and the consequent
parabolic Taylor formula. 
Finally, the Appendix contains the proof of Theorem \ref{th:error_estimates_taylor} and other
auxiliary results, namely: some short-time/small-volatility asymptotic estimates for the
Black-Scholes sensitivities (Appendix \ref{app:BS_greeks}), an explicit representation formula for
the terms appearing in the proxy $\bar{\s}_N$ (Appendix \ref{app:un_represent_theorem}), and a
multi-variate version of the Fa\`a di Bruno's formula (Appendix \ref{append:faa_bell}).

{\bf Acknowledgments.} The authors are grateful to Enrico Priola, Jian Wang and an anonymous
referee for their valuable comments and suggestions to improve the quality of the paper.

\section{{Local diffusions and local transition densities}
}\label{sec:model} In this section we describe the general setting and state the standing
assumptions under which the main results of the paper are carried out. We also show some examples
and prove some conditions under which such assumptions are satisfied. Generally we adopt
definitions and notations from \cite{FriedmanSDE1,FriedmanSDE2}.

We fix $T_{0}>0$ and consider a continuous $\R^d$-valued
Markov process $Z=(Z_{t})_{t\in[0,T_{0}]}$ with transition probability function $\bar{p}=\bar{p}(t,z;T,d\z)$,
defined on the space $\left(\O,\F,(\F_{T}^{t})_{0\le t\le T\le T_{0}},(P_{t,z})_{0\le t\le
T_{0}}\right)$. For any bounded Borel measurable function $\phi$, we denote by
\begin{equation}\label{ae67}
   E_{t,z}\left[\phi(Z_{T})\right]:=(\TT_{t,T}\phi)(z):=\int_{\R^{d}}\bar{p}(t,z;T,d\z)\phi(\z),\qquad 0\le t<T\le
   T_{0},\ z\in \R^{d},
\end{equation}
the $P_{t,z}$-expectation and the semigroup associated with the transition probability function
$\bar{p}$, respectively (cf. Chapter 2.1 in \cite{FriedmanSDE1}).

We assume that $Z=(S,Y)$ where $S$ is a non-negative martingale\footnote{We assume that $S$
is a martingale in order to ensure that the financial model is well posed: however this assumption
will not be used in the proof of our main results.} and $Y$ takes values in $\R^{d-1}$: here $S$ represents the risk-neutral price of a financial asset and $Y$ models a
number of stochastic factors in the market. For simplicity, we assume zero interest rates and no
dividends\footnote{The case of deterministic interest rates and/or dividends can be easily
included by performing the analysis on the forward prices.}.

Throughout the paper we assume the existence of a
domain\footnote{Connected and open set.} $D\subseteq \R_{>0}\times \DomY$ on which the following
three standing assumptions hold. We would like to emphasize that in the following
assumptions, we impose only {\it local conditions,} satisfied by all the most popular financial
models.
\begin{assumption}\label{assum1i}
The process $Z$ is a \emph{local diffusion on $D$}, meaning that for any
  $t\in[0,T_{0}[$, $\d>0$, $1\le i,j\le d$ and $H$, compact subset of $D$, there exist
  the limits
\begin{align}\label{ae51b}
  &\lim_{h\rightarrow 0^+}\!\!\!\!\int\limits_{\{|z-\z|>\d\}\cap H}\frac{\bar{p}(t,z;t+h,d\z)}{h}=\lim_{h\rightarrow 0^+}\!\!\!\!
  \int\limits_{\{|z-\z|>\d\}\cap H}\frac{\bar{p}(t-h,z;t,d\z)}{h}=0,
\intertext{uniformly w.r.t. $z\in \R_{\ge 0}\times\DomY$, and the limits}
   \label{ae51}
  &\lim_{h\rightarrow 0^+}\!\!\!\!\int\limits_{|z-\z|>\d}\frac{\bar{p}(t,z;t+h,d\z)}{h}=\lim_{h\rightarrow 0^+}\!\!\!\!\int\limits_{|z-\z|>\d}\frac{\bar{p}(t-h,z;t,d\z)}{h}=0,\\  \label{ae51c}
  &\lim_{h\rightarrow 0^+}\!\!\!\!\int\limits_{|z-\z|<\d}\!\!\!\!(\z_{i}-z_{i})\frac{\bar{p}(t,z;t+h,d\z)}{h}=\lim_{h\rightarrow
  0^+}\!\!\!\!\int\limits_{|z-\z|<\d}\!\!\!\!(\z_{i}-z_{i})\frac{\bar{p}(t-h,z;t,d\z)}{h}=:\ac_{i}(t,z),\\ \label{ae51d}
  &\lim_{h\rightarrow 0^+}\!\!\!\!\int\limits_{|z-\z|<\d}\!\!\!\!(\z_{i}-z_{i})(\z_{j}-z_{j})\frac{\bar{p}(t,z;t+h,d\z)}{h}
  =\lim_{h\rightarrow
  0^+}\!\!\!\!\int\limits_{{|z-\z|<\d}}\!\!\!\!(\z_{i}-z_{i})(\z_{j}-z_{j})\frac{\bar{p}(t-h,z;t,d\z)}{h}=:\ac_{ij}(t,z),\qquad
\end{align}
uniformly w.r.t. $z\in H$.
\end{assumption}

The following lemma, whose proof is deferred to
Subsection \ref{subsec:proof_theorem_density}, collects some useful
consequences of Assumption \ref{assum1i}.
\begin{lemma}\label{la5}
Under Assumption \ref{assum1i},
for any {$\phi \in C_{0}([0,T_{0}]\times D)$
} and $f\in C_{0}^{2}\left([0,T_{0}]\times D\right)$
we have
\begin{align}
 \label{ae53}
 &\lim_{T-t\rightarrow 0^+} \left\| {\TT_{t,T}\phi(T,\cdot)-\phi(t,\cdot)}\right\|_{L^{\infty}(\R_{\ge 0}\times
 \DomY)}=0,\\
\label{ae71}
 &\lim_{T-t\rightarrow 0^+}\left\|\frac{\TT_{t,T}f(T,\cdot)-f(t,\cdot)}{T-t}-\left(\p_{t}+\Acc_{t}\right)
 f(t,\cdot)\right\|_{L^{\infty}(\R_{\ge 0}\times \DomY)}=0,
\end{align}
where
\begin{align}\label{ae75}
 \Acc_{t}:= \frac{1}{2}\sum_{i,j=1}^{d}\ac_{ij}(t,z)\p_{z_{i}z_{j}}+\sum_{i=1}^{d}\ac_{i}(t,z)\p_{z_{i}}\qquad t\in [0,T_{0}[,\ z\in D.
\end{align}
Moreover, for any $0\le t< T< T_{0}$ and $z\in \R_{\ge 0}\times \DomY$, we have
\begin{equation}\label{ae62}
 \frac{d}{dT}\left(\TT_{t,T}f(T,\cdot)\right)(z)=\TT_{t,T}\left(\left(\p_{T}+\Acc_{T}
 \right)f(T,\cdot)\right)(z).
\end{equation}
\end{lemma}
Many financial models are defined in terms of (stopped) solutions of stochastic differential
equations. We refer to Section 2.2 in \cite{FriedmanSDE1} for the definition and basic results
about $t$-stopping times with respect to a given Markov process. The following result shows that
stopped solutions of SDEs satisfy Assumption \ref{assum1i}.
\begin{lemma}\label{la6}
Let $\left(Z_{t}\right)_{t\in[0,T_{0}]}$ be a continuous Markov process defined as $Z_{t}=\hat{Z}_{t\wedge \t}$, where:
\begin{itemize}
  \item[i)] $\hat{Z}$ is a solution of the SDE
\begin{align}\label{eq:sde}
 d\hat{Z}_{t}=\m(t,\hat{Z}_{t})dt+\s(t,\hat{Z}_{t})dW_{t}
 \end{align}
where $W$ is a multi-dimensional Brownian motion and the coefficients of the SDE are continuous and bounded on $[0,T_0]\times D$, with $D$ a domain of $\R^d$;
  \item[ii)] $\t$ is the first exit time of $\hat{Z}$ from a domain $D'\subseteq \R_{\geq 0}\times \R^{d-1}$ containing $D$.
\end{itemize}
Then $Z$ is a local diffusion on $D$ in the sense of Assumption \ref{assum1i}, with
\begin{align}\label{ae76}
 \ac_{i}=\m_i,\qquad \ac_{ij}=\big(\s\s^{\ast}\big)_{ij},\qquad 1\leq i,j \leq d.
\end{align}
\end{lemma}
The proof of Lemma \ref{la6} is deferred to Subsection \ref{subsec:proof_theorem_density}.

We refer to the operator $\Acc_{t}$ in \eqref{ae75} as the \emph{infinitesimal generator of $Z$ on $D$}. In the
second standing assumption we require that $\Acc_{t}$ is a non-degenerate operator. Notice that
$\Acc_{t}$ is defined only locally, on the domain $D$. In the following assumption and
throughout the paper $N\ge 2$ is a fixed integer\footnote{To simplify the presentation, we assume
$N\ge 2$. However, the proofs of neither the results in dimension one (i.e. $d=1$), nor the results for the derivatives of order one
or two in a generic dimension, do require this condition.}.
\begin{assumption}\label{assum1and}
The operator $\Acc_{t}$
satisfies the following conditions:
\begin{itemize}
  \item[(i)] the coefficients $\ac_{ij}
  ,\ac_{i}
  \in C_{P}^{N,1}([0,T_0[\times D)$, where $C_{P}^{N,\a}$ denotes the usual parabolic
H\"older space (see, for instance, Chapter 10.1 in \cite{FriedmanSDE2});
  \item[(ii)] $\Acc_{t}$ is elliptic on $D$, i.e. there exist $M>0$ and $\eps\in ]0,1[$ such that
\begin{align}\label{cond:ell-loc}
  \eps M |\z|^2 \leq \sum_{i,j=1}^{d}\ac_{ij}(t,z)\z_{i}\z_{j}\leq M  |\z|^2 ,\qquad t\in\left[0,T_{0}\right[,\ z\in D,\
 \z\in\mathbb{R}^d.
\end{align}
\end{itemize}
\end{assumption}

Finally, we state the third standing assumption.
\begin{assumption}\label{assum1ii}
$Z$ is a \emph{Feller process on $D$}, i.e. for any $T\in ]0,T_{0}[$ and {$\phi\in C_0(\R^{d})$}
the function $(t,z)\mapsto (\TT_{t,T}\phi)(z)$
is continuous on $[0,T[\times D$.
\end{assumption}

The following result summarizes some properties of the
law of $Z$. In particular
it states the existence of a {\it local transition density} for $Z$ on $D$, which is a non-negative
measurable function $\Gb=\Gb(t,z;T,\z)$, defined for $0\le t<T< T_{0}$ and $z,\z\in D$, such that,
for any $H\in \Bc(D)$ (Borel subset of $D$),
\begin{equation}\label{ae36_bis}
  \bar{p}(t,z;T,H)=\int_{H}\Gb(t,z;T,\z) d\z.
  \end{equation}
Moreover, it provides a double characterization of such local density, first as a solution to a
\emph{forward} Kolmogorov equation (w.r.t. the \emph{ending point} $(T,\z)$) and then as a
solution to a \emph{backward} Kolmogorov equation (w.r.t. the \emph{initial point} $(t,z)$). The
existence and the \emph{forward} representation follow from H\"ormander's theorem,
\cite{Hormander}, after proving that the law is a local solution, in the distributional sense, of
the adjoint of the infinitesimal generator of $Z$. This result is rather classical and is coherent
with the well-known results by \cite{kusuoka-stroock} (see also the more recent paper by
\cite{DeMarco2011}). In order to prove the \emph{backward} formulation we still employ
H\"ormander's theorem, but in this case the proof is more delicate and technically involved. In
fact, to prove that the law is a distributional solution of the generator of $Z$, it will be
crucial to use the Feller property combined with the classical pointwise estimates by
\cite{moser71} for weak solutions of parabolic PDEs. At this level of generality, the resulting
\emph{backward} representation for the transition local density appears to be novel and of
independent interest.
\begin{theorem}\label{la1}
Let Assumptions \ref{assum1i} and \ref{assum1and} be in force.
Then $Z$ has a local transition density $\Gb$ on $D$ such that, for any $(t,z)\in [0,T_{0}[\times D$, $\Gb(t,z;\cdot,\cdot)\in C^{N,1}_{P}\left(]t,T_{0}[\times D\right)$ and solves the forward
Kolmogorov equation
\begin{equation}\label{ae44}
  \left(\p_{T}-\Acc_{T}^{*}\right)
  f=0\qquad \text{on }]t,T_{0}[\times D.
\end{equation}
Here $\Acc_{T}^{*}$ denotes the formal adjoint of $\Acc_{T}$, acting as
\begin{align}
 \Acc^{*}_{T}f=\frac{1}{2}\sum_{i,j=1}^{d} \p_{z_{i}z_{j}}\left(\ac_{ij}(T,\cdot)f\right)-
 \sum_{i=1}^{d}\p_{z_{i}}\left(\ac_{i}(T,\cdot)f\right).
 \end{align}
If in addition also Assumption \ref{assum1ii} is satisfied, then $\Gb(\cdot,\cdot;T,\z)\in {C_{P}^{N+2,1}([0,T[\times D)}$ for any $(T,\z)\in ]0,T_{0}[\times D$, and solves the backward Kolmogorov equation
   \begin{equation}\label{eq:backward_Kolm}
   \left(\p_{t}+\Acc_{t}\right)f=0\qquad \text{on }{[0,T[\times D}.
   \end{equation}
\end{theorem}
We will give a detailed proof of Theorem \ref{la1} in Subsection \ref{subsec:proof_theorem_density}. Before, in Subsections \ref{subsec:CEV} and \ref{subsec:SDEs}, we provide illustrative examples of popular models that satisfy
Assumptions \ref{assum1i}, \ref{assum1and} and  \ref{assum1ii}, and to which our analysis applies.
Only in order to deal with the {derivatives}
of a Call option price w.r.t. the strike, in
Section \ref{sec:price.error.local} we will introduce additional assumptions to ensure existence
and local boundedness of such derivatives.

\subsection{The CEV model}\label{subsec:CEV}
Consider the SDE
\begin{equation}\label{25}
  d\tilde{S}_{t}=\s \tilde{S}_{t}^{\b}dW_{t},
\end{equation}
where $\s>0$ and $0<\b<1$. It is well-known (cf. \cite{Ikedabook}, p. 221, or \cite{yorbook2},
Chapter 11) that \eqref{25} has a {\it unique strong solution} that can be represented, through the
transformation $X_{t}=\frac{\tilde{S}_{t}^{2(1-\b)}}{\s^{2}(1-\b)^{2}}$, in terms of the squared
Bessel process
  $$dX_{t}=\d dt+ 2\sqrt{X_{t}}dW_{t},$$
with $\d=\frac{1-2\b}{1-\b}$. The process $\tilde{S}$ has distinct properties according to the
parameter regimes $\b<\frac{1}{2}$ and $\b\ge \frac{1}{2}$. To describe these properties, first we
introduce the functions
\begin{equation}\label{28}
  \Gb_{\pm}(t,s;T,S)=\frac{s^{\frac{1}{2}-2\b}\sqrt{S}e^{-\frac{s^{2(1-\b)}+S^{2(1-\b)}}{2(1-\b)^{2}\s^{2}(T-t)}}}{(1-\b)\s^{2}(T-t)}
  I_{\pm\frac{ 1}{2(1-\b)}}\left(\frac{(s S)^{1-\b}}{(1-\b)^{2}\s^{2}(T-t)}\right),
\end{equation}
where $I_{\n}(x)$ is the modified Bessel function of the first kind defined by
  $$I_{\n}(x)=\left(\frac{x}{2}\right)^{\n}\sum_{k=0}^{\infty}\frac{x^{2k}}{2^{2k}k!\Gamma_E(\n+k+1)},$$
and $\Gamma_{E}$ represents the Euler Gamma function. Both $\Gb_{+}$ and $\Gb_{-}$ are fundamental
solutions of $(\p_{t}+\Acc)$ where $\Acc$ is the infinitesimal generator of $\hat{S}$:
\begin{equation}\label{ae50}
  \Acc=\frac{\s^{2}s^{2\b}}{2}\p_{ss}.
\end{equation}
Precisely, we have
  $$(\p_{t}+\Acc)\Gb_{\pm}(\cdot,\cdot;T,S)=0,\qquad \text{on }[0,T[\times\R_{>0},$$
and
  $$ \lim_{(t,s)\to(T,\bar{s})\atop t<T}\int_{\R_{>0}}\Gb_{\pm}(t,s;T,S)\phi(S)dS=\phi(\bar{s}), \qquad \bar{s}\in\R_{>0},$$
for any continuous and bounded function $\phi$.

The point $0$ is an attainable state for $\tilde{S}$. In particular, if $\b\ge \frac{1}{2}$ then
$0$ is absorbing: if we denote by $\t_{s}:=\inf\{\t\mid\tilde{S}_{\t}=0\}$ the first time
$\tilde{S}$ hits $0$ starting from $\tilde{S}_{0}=s\ge 0$, then we have $\tilde{S}_{t}=0$ for
$t\ge \t_{s}$. The law of $\tilde{S}$ has a Dirac delta component at the origin and the function
$\Gb_{+}$ in \eqref{28} is the transition {\it semi-density} of $\tilde{S}$ on $\R_{>0}$: more
precisely, denoting by $\tilde{p}$ the transition probability function of $\tilde{S}$, we have
  $$\tilde{p}(t,s;T,H)=\int_{H}\Gb_{+}(t,s;T,S)dS$$
for any Borel subset $H$ of $\R_{>0}$ and
  $$\int_{0}^{+\infty}\Gb_{+}(t,s;T,S)dS<1.$$
On the other hand, if $\b< \frac{1}{2}$ then $\tilde{S}$ reaches $0$ but it is reflected: in this
case $\Gb_{-}$, which integrates to one on $\R_{>0}$, is the transition density of $\tilde{S}$.
Moreover, $\tilde{S}$ is a strict local martingale (cf. \cite{DelbaenShirakawa2002} or
\cite{HestonLoewensteinWillard2007}) that ``cannot'' represent the risk-neutral price of an asset:
the intuitive idea is that arbitrage opportunities would arise investing in an asset whose price
is zero at the stopping time $\t_{s}$ but later becomes positive.

For this reason, in the CEV model introduced by \cite{Cox1975} the asset price is defined as the
process obtained by stopping the unique strong solution $\hat{S}$, starting from
$\tilde{S}_{0}=s$, of the SDE \eqref{25} at $\t_{s}$, that is
  $$S_{t}:=\tilde{S}_{t\wedge\t_{s}},\qquad t\ge 0.$$
For any $0<\b<1$, the transition semi-density of $S$ is $\Gb_{+}$ in \eqref{28}.
For this model, \cite{DelbaenShirakawa2002} show that, for any $0<\b<1$, the process is a non-negative martingale.

Now let $D$ be any domain compactly contained in $\R_{>0}$. By Lemma \ref{la6}, the stopped process
$S$ is a local diffusion on $D$ and satisfies Assumption \ref{assum1i}. The infinitesimal
generator $\Acc$ is the operator in \eqref{ae50}, has smooth coefficients and is uniformly
elliptic on $D$: thus Assumption \ref{assum1and} is satisfied {for any $N\in\N$}. Moreover, the
Feller property on $D$ (Assumption \ref{assum1ii}) follows from the explicit expression of the
transition semi-density or from the general results in \cite{ethier1986markov}, Chapter 8 (see
Problem 3 p.382 and Thm. 2.1 p.371).

The CEV model (and also its stochastic volatility counterpart, the popular SABR model used in
interest rates modeling) is an interesting example of degenerate model because the infinitesimal
generator is {\it not globally uniformly elliptic and the law of the price process is not
absolutely continuous w.r.t the Lebesgue measure}.

\begin{remark}
\cite{Durrleman2010}{, p. 175,  provided} formulas for the implied volatility in a local
volatility (LV) model with LV-function $\s=\s(s)$. His expression for the time-derivative of the
ATM implied volatility, denoted by $\Sigma$, is equal to
\begin{equation}\label{aed1}
  \p_{t}\Sigma(t,s)\vert_{t=0}=\frac{1}{12} s^2(s)^2 \sigma ''(s)-\frac{4}{3} s^2 \sigma (s) \sigma '(s)^2+\frac{1}{12} s \sigma (s)^2 \sigma '(s).
\end{equation}
{The latter} is slightly different from the expression we get from our Taylor expansion that, in
this particular case, can be computed as in Section \ref{sec:impvol} and reads as
\begin{equation}\label{aenoi1}
 \p_{t}\Sigma(t,s)\vert_{t=0}=\frac{1}{12} s^2 \sigma (s)^2 \sigma ''(s)-\frac{1}{24} s^2 \sigma (s) \sigma'(s)^2+\frac{1}{12} s \sigma (s)^2 \sigma '(s).
\end{equation}
Actually, {simple numerical tests performed in the CEV model confirm that formula \eqref{aenoi1}
is correct.} As a matter of example, in Table \ref{tab1} we show the values of
$\p_{t}\Sigma(t,1)\vert_{t=0}$ in the CEV model with $\s=S_{0}=1$ (cf. \eqref{25}) and
$\b=0.1,\dots,0.9$.
\begin{table}[ht]
 \caption{ATM IV time-derivative}
 \centering
 \label{tab1}
\begin{tabular}{l l l l}
 \hline\hline
 $\b$ & Numerical approx. & Taylor expansion & Durrleman\\
 \hline
0.1 & 0.0337524 & 0.03375 & -1.0125 \\
 0.2 & 0.0266639 & 0.0266667 & -0.8 \\
 0.3 & 0.0204115 & 0.0204167 & -0.6125 \\
 0.4 & 0.0149955 & 0.015 & -0.45 \\
 0.5 & 0.0104115 & 0.0104167 & -0.3125 \\
 0.6 & 0.00666029 & 0.00666667 & -0.2 \\
 0.7 & 0.00374753 & 0.00375 & -0.1125 \\
 0.8 & 0.00136839 & 0.00166667 & -0.05 \\
 0.9 & 0.000415421 & 0.000416667 & -0.0125
\end{tabular}
\end{table}
\end{remark}

\subsection{Multi-factor local-stochastic volatility models}\label{subsec:SDEs}
We consider a pricing model defined as the solution of a
system of SDEs of the form
\begin{align}\label{eq:StochVol}
&\begin{cases}
 d S_t =  \cy_1(t,S_t,Y_t)S_t \,d W^{(1)}_t,\\
  S_0 = s \in \Rb_{> 0},
\end{cases}
\\ \label{eq:StochVol_bis}
&\begin{cases}
 d Y^{(i)}_t =\mu_i(t,S_t,Y_t) d t + \cy_{i}(t,S_t,Y_t) d W^{(i)}_t, \qquad i=2,\dots,d,\\
 Y_0 = y \in \DomY,
\end{cases}
\end{align}
where
$W$ is a $d$-dimensional correlated Brownian motion with
\begin{equation}\label{eq:correlation}
d \< W^{(i)}, W^{(j)} \>_t = \rho_{ij}(t,S_t,Y_t) d t,\qquad
\ i,j=1,\dots,d.
\end{equation}

In the most classical setting, one assumes that the coefficients of the SDEs are measurable
functions, locally Lipschitz continuous in the spatial variables
$(s,y)
$ uniformly w.r.t. $t\in[0,T_{0}]$, and have sub-linear growth in $(s,y)$; for more details we
refer, for instance, to condition (A$'$) p.113 of Chapter 5.3 in \cite{FriedmanSDE1}. In this
case, a unique global-in-time solution $(S,Y)$ exists, which is a Feller process\footnote{The
definition of Feller process given in \cite{FriedmanSDE1}, Chapter 2.2, is slightly different from
ours. However the Feller property for solutions of SDEs is proved in \cite{FriedmanSDE1} as a
consequence of Lemma 5.3.3: this lemma also implies the Feller property as given in Assumption
\ref{assum1ii}.} and a diffusion (see Theorems 5.3.4 and 5.4.2 in \cite{FriedmanSDE1}).

Usually, however, the above conditions are considered too restrictive and of limited practical
use. Actually, we shall see that Assumptions \ref{assum1i}, \ref{assum1and} and \ref{assum1ii} are satisfied under much weaker conditions. To see this, we first note that the infinitesimal generator $\Acc$ of $(S,Y)$ is the operator
of the form \eqref{ae75} with coefficients given by
\begin{align} \bar{a}_1= 0, & &
\bar{a}_{i} = \mu_i,    &&   \bar{a}_{11}=\rho_{11}\eta^2_1 s^2,  & &
\bar{a}_{1i}=\bar{a}_{i1}=\rho_{1i}\eta_i \eta_1s,   & &
\bar{a}_{ij}=\bar{a}_{ji}=\rho_{ij}\eta_i\eta_j       & & i,j=2,\cdots,d.
\end{align}
Now, Assumption \ref{assum1and} is straightforward to verify and applies to the great majority of the models used in finance, and thus, by Lemma \ref{la6}, Assumption \ref{assum1i} is also satisfied provided that a solution to the system \eqref{eq:StochVol}-\eqref{eq:StochVol_bis} exists. The Feller property in Assumption \ref{assum1ii} has to be verified case by case. Results ensuring the Feller property for the solution of an SDE under weak regularity conditions on the coefficients (H\"older or local Lipschitz continuity) have been recently proved by \cite{Wang-Feller} (see Proposition 2.1) and by \cite{WangZhang2015}. Moreover, the results of Chapter 8 in \cite{ethier1986markov} cover several SDEs related to financial models.

\smallskip
As a matter of example, we analyze the classical model proposed by \cite{heston1993}. Set $d=2$ and
\begin{align}\label{eq:Heston_S}
&
 d S_t =   S_t \sqrt{Y_t} d W^{(1)}_t,&&  S_0 \in \Rb_{>0},
\\ \label{eq:Heston V}
&
 d Y_t =\kappa (\theta-Y_t) d t + \delta \sqrt{Y_t} d W^{(2)}_t,&& Y_0 \in \R_{> 0},
\end{align}
where $\delta$ is a positive constant (the so-called vol-of-vol parameter), $\kappa,\theta>0$ are
the drift-mean and the mean-reverting term of the variance process respectively, and $W$ is
a $2$-dimensional Brownian motion with correlation $\rho\in ]-1,1[$.
It is well known that the joint transition probability function $\bar{p}$ in \eqref{ae67} admits an explicit characterization in terms of its Fourier-Laplace transform.
Precisely, setting $X_{t}=\log S_{t}$, and assuming for simplicity $\d=1$, we have 
\begin{equation}\label{eq:Heston_char_exp}
\hat{p}(t,x,y;T,\x,\y):=E_{t,x,y}\left[e^{\ii \x X_{T}-\y Y_{T}}\right]=e^{\ii x \xi-yA(T-t,\xi,\eta)} B(T-t,\xi,\eta),
\end{equation}
where
\begin{align}
A(u,\xi,\eta)=  \frac{b(\xi)g(\xi,\eta) e^{-D(\xi)(u-s)}- a(\xi)}{g(\xi,\eta)
  e^{-D(\xi)(u-s)}-1}, && B(u,\xi,\eta)=  e^{-\k \theta a(\xi)u} \left(\frac{g(\xi,\eta)-1}{g(\xi,\eta)e^{-D(\xi)u}-1}\right)^{2\k\theta},
\end{align}
with
\begin{align}
g(\xi,\eta)=\frac{a(\xi)-\eta}{b(\xi)-\eta}, && a(\xi)=\ii \xi\rho - \k+D(\xi), &&   b(\xi)=\ii \xi\rho - \k-D(\xi), &&   D(\xi)=\sqrt{(\ii \xi\rho - \k)^{2}+\xi\left(\xi+\ii \right)}.
\end{align}
Using the explicit knowledge of the characteristic function of $S$, \cite{andersen2007moment},
Proposition 2.5, prove that $S$ is a martingale and can reach neither $\infty$ nor $0$ in finite
time (see also \cite{LionsMusiela} for related results in a more general setting). The variance
process $Y$ can reach the boundary with positive probability if the Feller condition $2\k \theta
\geq \delta^{2}$ is violated and in this case the origin is a reflecting boundary. In any case,
the distribution of $Y_t$ has no mass at $0$ for any positive $t$.

By Lemma \ref{la6}, Assumptions \ref{assum1i} is verified on any domain $D$ compactly contained in
$ \R_{>0}\times \R_{> 0}$ and the generator $\Acc$ of $(S,Y)$ reads as
\begin{equation}
\Acc= \frac{y s^{2}}{2}\p_{ss}+\frac{\delta^{2}y}{2}\p_{yy}+\rho\delta y s \,\p_{sy}
+\kappa(\theta-y)\partial_y,\qquad (s,y)\in \R_{>0}\times \R_{\geq 0}.
\end{equation}
It is also clear that Assumption \ref{assum1and} is satisfied on $D$ for any $N\in\N$.
Finally, the Feller property follows by the explicit expression of the characteristic function in \eqref{eq:Heston_char_exp}, and thus Assumption \ref{assum1ii} is also satisfied.
\begin{remark}\label{rem:Heston_density}
By Theorem \ref{la1}, the couple $(S,Y)$ in the Heston model has a smooth local transition density
on any domain $D$ compactly contained in $ \R_{>0}\times \R_{> 0}$. Therefore, since $p\big(t,z;T,
\R^2 \setminus ( \R_{>0}\times \R_{> 0}) \big)=0$, the process $(S,Y)$ has a  transition density
on $\R^2$, which is smooth on $\R_{>0}\times \R_{> 0}$. In particular, the marginal distribution
of $S_t$ has a smooth density on $\R_{>0}$, which is consistent with \cite{Rollin2010}.
\end{remark}

\subsection{Proofs of Lemmas \ref{la5},  \ref{la6} and Theorem \ref{la1}
}\label{subsec:proof_theorem_density}

\proof[of Lemma \ref{la5}] We first remark that in the statement of the lemma, the short notation
(see \eqref{ae53})
  $$\lim_{T-t\rightarrow 0^+} \left\| {\TT_{t,T}\phi(T,\cdot)-\phi(t,\cdot)}\right\|_{\infty}=0,$$
must be interpreted as
  $$\lim_{h\rightarrow 0^+} \left\| {\TT_{t,t+h}\phi(t+h,\cdot)-\phi(t,\cdot)}\right\|_{\infty}=
  \lim_{h\rightarrow 0^+} \left\| {\TT_{t-h,t}\phi(t,\cdot)-\phi(t-h,\cdot)}\right\|_{\infty}=0,$$
 and analogously for \eqref{ae71}.
  Hereafter, for greater convenience, we shall use this abbreviation systematically. Now let us prove
\eqref{ae53}. For a given $\phi \in C_{0}([0,T_0]\times D)$, we denote by $H_{\phi}$ the support
of $\phi$ and consider a compact subset $H$ of $D$ such that $H_{\phi}\subseteq [0,T_0]\times H$
and $\bar{\d}:=\text{dist}\big(H_{\phi},[0,T_{0}]\times (\R^{d}\setminus H)\big)>0$. Then we have
  $$\TT_{t,T}\phi(T,z)-\phi(t,z)=I_{t,T,1}(z)+I_{t,T,2}(z)+I_{t,T,3}(z)$$
where
\begin{align}
 I_{t,T,1}(z)&= \int_{H}\bar{p}(t,z;T,d\z)\left(\phi(T,\z)-\phi(T,z)\right),\\
 I_{t,T,2}(z)&= \left(\phi(T,z)-\phi(t,z)\right)\int_{H}\bar{p}(t,z;T,d\z),
 \\
 I_{t,T,3}(z)&=-\phi(t,z)\int_{(\R_{\ge 0}\times \DomY)\setminus H}\bar{p}(t,z;T,d\z).
\end{align}
Since $\phi$ is uniformly continuous, for any $\e>0$ there exists $\d_{\e}>0$ such that
\begin{align}
 \left|I_{t,T,1}(z)\right|\le
 \e\int_{|z-\z|\le\d_{\e}}\bar{p}(t,z;T,d\z)+2\|\phi\|_{\infty}\int_{H\cap\{|z-\z|>\d_{\e}\}}\bar{p}(t,z;T,d\z)
\end{align}
and therefore, by \eqref{ae51b},
  $$\limsup_{T-t\rightarrow 0^+}\left|I_{t,T,1}(z)\right|\le\e$$
uniformly w.r.t. $z\in \R_{\ge 0}\times \DomY$. Moreover  we have
 $$\left|I_{t,T,2}(z)\right|\le \left|\phi(T,z)-\phi(t,z)\right|\longrightarrow 0$$
as $T-t\to 0^+$, uniformly w.r.t. $z$.
On the other hand, by \eqref{ae51} we have
  $$\left|I_{t,T,3}(z)\right|\le\|\phi\|_{\infty}\int_{|z-\z|>\bar{\d}}\bar{p}(t,z;T,d\z)\longrightarrow 0$$
as $T-t\rightarrow 0^+$, uniformly w.r.t. $z\in H_{\phi}$, and $I_{t,T,3}(z)\equiv 0$ if $z\notin
H_{\phi}$. This concludes the proof of \eqref{ae53}. Notice that, for any $z\in D$ and $r>0$ such that $B(z,r):=\{\z\mid|z-\z|<r\}\subseteq D$, we have
\begin{equation}\label{ae73}
 \lim_{T-t\rightarrow 0^+}\int_{B(z,r)}\bar{p}(t,z;T,d\z)=1;
\end{equation}
indeed for any $\phi \in C_{0}(B(z,r))$ such that $|\phi|\le 1$ and $\phi(z)=1$, by \eqref{ae53}
we have
  $$1\ge \int_{B(z,r)}\bar{p}(t,z;T,d\z)\ge\TT_{t,T}\phi(z)\longrightarrow \phi(z)=1$$
as $T-t\rightarrow 0^+$.

The proof of \eqref{ae71} is similar: for any $f\in C_{0}^{2}\left([0,T_{0}]\times D\right)$ we
have
  $$\frac{\TT_{t,T}f(T,z)-f(t,z)}{T-t}=I_{t,T,1}(z)+I_{t,T,2}(z)$$
where
\begin{align}\label{ae74}
 I_{t,T,1}(z)= \int_{H}\bar{p}(t,z;T,d\z)\frac{f(T,\z)-f(t,z)}{T-t},\qquad
 I_{t,T,2}(z)=\frac{f(t,z)}{T-t}\int_{(\R_{\ge 0}\times \DomY)\setminus H}\bar{p}(t,z;T,d\z),
\end{align}
with $H$ defined analogously to how it was defined in the proof of \eqref{ae53}. Again, by \eqref{ae51} the term
$I_{t,T,2}(z)$ is negligible in the limit. As for $I_{t,T,1}(z)$, it suffices to plug the Taylor
formula
\begin{equation}\label{ae63}
\begin{split}
  f(T,\z)-f(t,z)=&\,(T-t)\p_{t}f(t,z)+\sum_{i=1}^{d}(\z_{i}-z_{i})\p_{z_{i}}f(t,z)\\
  &+\frac{1}{2}\sum_{i,j=1}^{d}(\z_{i}-z_{i})(\z_{j}-z_{j})\p_{z_{i}z_{j}}f(t,z)+\text{o}(|T-t|)+\text{o}(|z-\z|^{2}).
\end{split}
\end{equation}
into \eqref{ae74} and pass to the limit using \eqref{ae73}, \eqref{ae51c} and \eqref{ae51d}. This
proves \eqref{ae71}.

Finally, we have
\begin{align}
 &\left\| \frac{\TT_{t,T+h}f(T+h,\cdot)-\TT_{t,T}f(T,\cdot)}{h} - \TT_{t,T} ((\p_{T}+\Acc_{T}) f(T,\cdot))
 \right\|_{L^{\infty}(\R_{\geq 0}\times \DomY)}\\
 & = \left\| \TT_{t,T}\left(
 \frac{\TT_{T,T+h}f(T+h,\cdot)-f(T,\cdot)}{h}- (\p_{T}+\Acc_{T}) f(T,\cdot) \right) \right\|_{L^{\infty}(\R_{\geq 0}\times
 \DomY)} \\
 &\leq \left\| \frac{\TT_{T,T+h}f(T+h,\cdot)-f(T,\cdot)}{h}- (\p_{T}+\Acc_{T}) f(T,\cdot) \right\|_{L^{\infty}(\R_{\geq 0}\times \DomY)}
 \longrightarrow 0,\qquad \text{as }h\to 0^+,
\end{align}
where the last limit follows from \eqref{ae71}. This proves the existence of the right derivative.
For the left derivative it suffices to use the identity
\begin{align}
  &\frac{\TT_{t,T-h}f(T-h,\cdot)-\TT_{t,T}f(T,\cdot)}{-h}-\TT_{t,T}\left((\p_{T}+\Acc_{T})f(T,\cdot)\right)\\
  &=  \TT_{t,T-h}\left(\frac{\TT_{T-h,T}-I}{h}-(\p_{T}+\Acc_{T})\right)f(T,\cdot)+\left(\TT_{t,T-h}-\TT_{t,T}\right)\left((\p_{T}+\Acc_{T})f(T,\cdot)\right),
\end{align}
where $I$ is the identity operator. This concludes the proof.
\endproof

\proof[of Lemma \ref{la6}.] \emph{Step 1.} We prove \eqref{ae51b}. Fix $\d>0$ and $H$, compact
subset of $D$. {Consider a family of functions $(\phi_{z})_{z\in\R^d}$ such that $\phi_{z}(z)=0$,
$\phi_z(\z)\equiv 1$ for $\z\in H\cap\{|\z-z|>\d\}$ and $\phi_{z}\in C_{0}^{\infty}(D)$ with all
the derivatives bounded by a constant $C_{1}$ which depends on $D,H$ and $\d$ but not on $z$.} By
the It\^o formula we have
\begin{equation}\label{ae77}
  \phi_z(\hat{Z}_{T})= \phi_z(\hat{Z}_{t}) + \int_{t}^{T}\Acc_{s}\phi_z(\hat{Z}_{s})ds+\int_{t}^{T}\nabla \phi_z(\hat{Z}_{s})\s(s,\hat{Z}_{s})dW_{s},
\end{equation}
with $\Acc_{s}$ as defined in \eqref{ae75} and $\ac_{i},\ac_{ij}$ as in \eqref{ae76}. Notice that
  $$
  \left|\Acc_{s}\phi_z(\hat{Z}_{s})\right|+\left|\nabla \phi_z(\hat{Z}_{s})\s(s,\hat{Z}_{s})\right|\le C_{2}, \qquad s\in [0,T_0],\quad z\in \R^d,
  $$
with $C_{2}$ dependent only on $C_{1}$ and the $L^{\infty}([0,T_{0}]\times D)$-norm of the
coefficients of the SDE. Let $\bar{p}(t,z;T,d\z)$ denote the transition probability of the stopped
process $Z_{T}=\hat{Z}_{T\wedge \t}$. Then, by recalling the definition of $\t$ and since
$D\subseteq D'$ and $\phi_{z}$ has  compact support in $D$, we have
  $$\int_{\{|z-\z|>\d\}\cap H}\bar{p}(t,z;T,d\z)\le E_{t,z}\big[{\phi_z^{4}}(\hat{Z}_{T\wedge\t})\big]\le E_{t,z}\big[{\phi_z^{4}}(\hat{Z}_{T})\big],$$
and \eqref{ae51b} follows from \eqref{ae77}, the H\"older inequality and Doob's maximal inequality
(in the form of Corollary 6.4 p.87 in \cite{FriedmanSDE1} with $m=2$). The proof of \eqref{ae51}
is analogous and is omitted.

\medskip\noindent
\emph{Step 2.} We prove \eqref{ae51c}. Fix $1\leq i\leq d$ and $H$, compact subset of $D$. We
first remark that it is sufficient to prove the thesis for $\d<\bar{\d}:=\text{\rm dist}(H,\p D)$.
Indeed, we have
  $$\frac{1}{T-t}\int_{|z-\z|<\d}(\z_{i}-z_{i})\bar{p}(t,z;T,d\z)=\frac{1}{T-t}\int_{|z-\z|<\bar{\d}}(\z_{i}-z_{i})\bar{p}(t,z;T,d\z)+I_{t,T}$$
where, by \eqref{ae51},
  $$I_{t,T}=\frac{1}{T-t}\int_{\bar{\d}\le |z-\z|<\d}(\z_{i}-z_{i})\bar{p}(t,z;T,d\z)\longrightarrow 0$$
as $T-t\to 0^{+}$, uniformly w.r.t $z\in H$.

Next, we consider a family of functions $(\phi_{z})_{z\in H}$ such that
$\phi_{z}(\z)=\z_{i}-z_{i}$ for $|\z-z|<\d$ and $\phi_{z}\in C_{0}^{\infty}(D)$ with all the
derivatives bounded by a constant $C_{1}$ which depends on $D,H$ and $\d$ but not on $z$. Note
that
  \begin{equation}\label{eq:ste120}
  \left|\nabla \phi_{z}(Z_{s})\s(s,Z_{s})\right|\le C_{2}, \qquad s\in [0,T_0],\quad z\in H,
  \end{equation}
with $C_{2}$ dependent only on $C_{1}$ and the $L^{\infty}([0,T_{0}]\times D)$-norm of the
coefficients of the SDE. Now, we set $\Psi_{z}(t,\cdot)=\Acc_{t}\phi_{z}$ and note that $\Psi_{z}(t,\z)=a_{i}(t,\z)$ for
$|\z-z|<\d$. Denoting again by $\bar{p}(t,z;T,d\z)$ the transition probability of the stopped
process $(\hat{Z}_{T\wedge \t})$, we have
\begin{align}
 \frac{1}{T-t}\int_{|z-\z|<\d}(\z_{i}-z_{i})\bar{p}(t,z;T,d\z) - \ac_{i}(t,z) &= I_{1,t,T,z}+I_{2,t,T,z}
\end{align}
where, by \eqref{ae51},
  $$I_{1,t,T,z}:=-\frac{1}{T-t}\int_{|z-\z|\ge \d}\bar{p}(t,z;T,d\z)\phi_{z}(\z)\longrightarrow 0$$
as $T-t\to 0^{+}$, uniformly in $H$, and
\begin{align}
 I_{2,t,T,z}&:=E_{t,z}\left[\frac{{ \phi_{z}}(\hat{Z}_{T\wedge \t})}{T-t}-\Psi_{z}(t,z)\right]
\intertext{(since by assumption $D\subseteq D'$ and $\phi_{z}$ has compact support in $D$, and using \eqref{ae77} and the fact that, by
\eqref{eq:ste120}, the stochastic integral is a true martingale)}
 &=E_{t,z}\left[\frac{1}{T-t}\int_{t}^{T}\Acc_{s}\phi_{z}(\hat{Z}_{s\wedge \t})ds-\Psi_{z}(t,z)\right] \\
 &=E_{t,z}\left[\int_{0}^{1}\Psi_{z}(t+ \rho(T-t),\hat{Z}_{\left(t+ \rho(T-t)\right)\wedge\t})d\rho-\Psi_{z}(t,z)\right]
\intertext{(by Fubini's theorem)}
 &=\int_{0}^{1}\left(\left(\TT_{t,t+ \rho(T-t)}\Psi_{z}(t+ \rho(T-t),\cdot)\right)(z)
 -\Psi_{z}(t,z)\right)d\rho.
\end{align}
Thus, by \eqref{ae53} and the fact that $\Psi_{z}(t,\cdot)\in C_{0}([0,T_{0}]\times D)$ by
definition, we infer that $I_{2,t,T,z}$ converges to zero
as $T-t\to 0^{+}$, uniformly w.r.t. $z\in H$. We remark here explicitly that \eqref{ae53} in Lemma
\ref{la5} is proved using \eqref{ae51b} and \eqref{ae51} only, which in turn have already been
proved for the stopped process in the previous step; therefore, no circular argument has been used.
The proof of \eqref{ae51d} is based on analogous arguments; thus we leave the details to the
reader.
\endproof

\proof[of Theorem \ref{la1}] We fix $(t,z)\in [0,T_{0}[\times D$ and $f\in
C^{2}_{0}\left([0,T_{0}[\times D\right)$, and show that the process
\begin{equation}\label{ae41}
  M^{t}_{T}:=
  f(T,Z_{T})-f(t,Z_{t})-\int_{t}^{T}\left(\p_{u}+\Acc_{u}\right)f(u,Z_{u})du,\qquad
  {t\leq T<T_0},
\end{equation}
is a $\F^{t}$-martingale. First observe that, integrating \eqref{ae62}, we get the identity
\begin{equation}\label{ae68}
 \left(\TT_{t,T}f(T,\cdot)\right)(z)-f(t,z)=\int_{t}^{T}\TT_{t,\t}\left(\left(\p_{\t}+\Acc_{\t}\right)f(\t,\cdot)\right)(z)d\t, \qquad T\in ]t,T_0[.
\end{equation}
Note that the integrand in \eqref{ae68} is bounded, as a function of $\tau$, because of Assumption \ref{assum1and} and since $f\in C^{2}_{0}\left([0,T_{0}[\times D\right)$ and $\TT_{t,\tau}$ is a contraction. Now,
for $\t\in[t,T]$ we have
\begin{align}
  E_{t,z}\left[M^{t}_{T}\mid\F^{t}_{\t}\right]&=M^{t}_{\t}+E_{t,z}\left[f(T,Z_{T})-f(\t,Z_{\t})-\int_{\t}^{T}\left(\p_{u}+
  \Acc_{u}\right)f(u,Z_{u})du\mid\F^{t}_{\t}\right]\\
  &=M^{t}_{\t}+{\Phi(\t,Z_{\t})}
\end{align}
where, by the Markov property,
\begin{align}
{\Phi(\tau,z)}&=E_{\t,z}\left[f(T,Z_{T})-f(\t,z)-\int_{\t}^{T}\left(\p_{u}+
  \Acc_{u}\right)f(u,Z_{u})du\right]
\intertext{(by Fubini's theorem)}
  &=\left(\TT_{\t,T}f(T,\cdot)\right)(z)-f(\t,z)-\int_{\t}^{T}\TT_{\t,u}\left(\left(\p_{u}+
  \Acc_{u}\right)f(u,\cdot)\right)(z)du
\end{align}
which is $0$ by \eqref{ae68}.

Notice that $M^{t}_{t}=0$, thus for any $f\in C^{2}_{0}\left(]t,T_{0}[\times D\right)$ we have 
\begin{equation}\label{ae42}
  0=E_{t,z}\left[M^{t}_{T_{0}}\right]=\int_{t}^{T_{0}}\int_{{D}}\bar{p}(t,z;T,d\z)\left(\p_{T}+\Acc_{T}\right)f(T,\z)d T.
\end{equation}
Since $f$ is arbitrary, equation \eqref{ae42} means that $\bar{p}(t,z;\cdot,\cdot)$ satisfies
equation \eqref{ae44} on $]t,T_{0}[\times {D}$ in the sense of distributions. If the coefficients
of the generator are smooth functions, then from H\"ormander's theorem (see, for instance, Section
V.38 in \cite{RogersWilliams}) we infer that $\bar{p}(t,z;\cdot,\cdot)$ admits a local density
$\Gb(t,z;\cdot,\cdot)$ which is a smooth function and solves the forward Kolmogorov PDE on
$]t,T_{0}[\times {D}$. In the general case, it suffices to use a standard regularization argument
by smoothing the coefficients and then applying Schauder's interior estimates (cf.
\cite{FriedmanSDE2}, Chapter 10.1): in regard to this, we refer for instance to
\cite{Kusuoka2015}. The first part of the statement then follows since $z$ and $r$ are arbitrary.

Next, we use the classical Moser's pointwise estimates (see \cite{moser71} and the more
recent and general formulation in Corollary 1.4 in \cite{pascpol1}) to prove a
$L^{\infty}_{\text{\rm loc}}$-estimate of $\Gb$ that will be used in the second part of the proof.
More precisely, let us fix $(t,z)\in[0,T_{0}[\times D$, $T\in]t,T_{0}[$ and $H$, compact subset of
$D$, and set $r=\frac{1}{2}\min\{\sqrt{T_{0}-T},\sqrt{T-t},\text{\rm dist}(H,\p D)\}$. Since
$\Gb(t,z;\cdot,\cdot)$ solves the PDE $\left(\p_{T}-\Acc_{T}^{*}\right)\Gb(t,z;\cdot,\cdot)=0$ on
$]t,T_{0}[\times D$, by Moser's estimate we have that
\begin{equation}\label{eq:ste129}
 \Gb(t,z;T,\z) \leq \frac{c_{0}}{r^{d+2}}\int_{T-r^{2}}^{T+r^{2}}\int_{B(\z,r)} \Gb(t,z;\bar{T},\bar{\zeta}) d \bar{\z} d \bar{T} \leq 2c_{0}r^{-d},\qquad \z\in H,
\end{equation}
where {\it the constant $c_{0}$ depends only on the dimension $d$ and the local-ellipticity
constant $M$ of Assumption \ref{assum1and}-(ii)}. We notice explicitly that the constant $c_{0}$
in \eqref{eq:ste129} is independent of $z\in D$ and $\z\in H$.

\smallskip
To prove the second part of Theorem \ref{la1}, we adapt the argument of Theorem 2.7 in
\cite{JansonTysk2006}.
We fix $\phi\in C_{0}(D)$, $ T\in ]0,T_{0}[$, $z_{0}\in D$ and $r>0$ such that the closure of the
ball $B(z_{0},r)$ is contained in $D$. Then we denote by $f$ the smooth solution of
\begin{equation}\label{ae80}
  \begin{cases}
  \left(\p_{t}+{\Acc_{t}}\right)f=0\qquad & \text{on } [0,T[\times B(z_{0},r), \\
    f(t,z)=\left(\TT_{t,T}\phi\right)(z) &  (t,z)\in\p_{P}\left([0,T]\times B(z_{0},r)\right),
  \end{cases}
\end{equation}
where
  $$\p_{P}\left([0,T]\times B(z_{0},r)\right):=\left([0,T]\times \p B(z_{0},r)\cup\left(\{T\}\times B(z_{0},r)\right)\right)$$
is the parabolic boundary of the cylinder $[0,T]\times B(z_{0},r)$. Such a solution exists because
$\Acc_{t}$ is uniformly elliptic on $[0,T_{0}[\times D$ and $(t,z)\mapsto (\TT_{t,T}\phi)(z)$ is
continuous on $[0,T]\times D$ by the Feller property (cf. Assumption \ref{assum1ii})) and
\eqref{ae53}.

Now, we fix $t\in [0,T[$ and denote by $\t_{0}$ the $t$-stopping time defined as $\t_{0}=T \wedge
\t_{1}$ where $\t_{1}$ is the first exit time, after $t$, of
$Z$ from $B(z_{0},r)$.
By the $\F^{t}$-martingale property of the process $M^{t}$ in \eqref{ae41}, with $f$ as in
\eqref{ae80}, and the Optional sampling theorem, we have the stochastic representation
\begin{equation}\label{ae33}
  f(t,z)=E_{t,z}\left[(\TT_{\t_{0},T}\phi)(Z_{\t_{0}})\right].
\end{equation}
On the other hand, {for $(t,z)\in [0,T[\times B(z_{0},r)$} we have
\begin{align}
 (\TT_{t,T}\phi)(z)&=E_{t,z}\left[\phi(Z_{T})\right]  =E_{t,z}\left[E_{t,z}\left[\phi(Z_{T})\mid \F^{t}_{\t_{0}}\right]\right]=
\intertext{(by the strong Markov property)}
 &=E_{t,z}\left[(\TT_{\t_{0},T}\phi)(Z_{\t_{0}})\right]=f(t,z), \label{ae30}
\end{align}
and in particular $(t,z)\mapsto(\TT_{t,T}\phi)(z)$ solves the backward equation \eqref{eq:backward_Kolm}.

Finally, we consider a sequence $(\phi_{n})_{n\in\N}$ of functions in $C_{0}(D)$, approximating a
Dirac delta $\d_{\bar{z}}$ for a fixed $\bar{z}\in D$. We also fix a test function $\psi\in
C_{0}^{\infty}(]0,T[\times D)$ and integrate by parts to obtain
\begin{align}
 0&=\int_{0}^{T}\int_{D}\left(\p_{t}+\Acc_{t}\right)(\TT_{t,T}\phi_{n})(z) \psi(t,z)dtdz\\
 &=\int_{0}^{T}\int_{D}(\TT_{t,T}\phi_{n})(z)\left(-\p_{t}+\Acc_{t}^{*}\right)\psi(t,z)dtdz\\
 \label{ae88}
 &=\int_{0}^{T}\int_{D}\int_{D}\Gb(t,z;T,\z)\phi_{n}(\z)d\z
 \left(-\p_{t}+\Acc_{t}^{*}\right)\psi(t,z)dtdz.
\end{align}
Note that $\z\mapsto\Gb(t,z;T,\z)$ is a continuous function for $t<T$, and therefore
  $$\int_{D}\Gb(t,z;T,\z)\phi_{n}(\z)d\z\longrightarrow \Gb(t,z;T,\bar{\z})$$
pointwisely. On the other hand, the $L^{\infty}_{\text{\rm loc}}$-estimate \eqref{eq:ste129} of
$\Gb$ allows to pass to the limit as $n\to\infty$ in \eqref{ae88}, using the dominated convergence
theorem, to get
 $$\int_{0}^{T}\int_{D}\Gb(t,z;T,\bar{z})\left(-\p_{t}+\Acc_{t}^{*}\right)\psi(t,z)dtdz=0.$$
This shows that $\Gb(\cdot,\cdot;T,\z)$ is a distributional solution of \eqref{eq:backward_Kolm}
on $[0,T[\times D$ and we conclude using again H\"ormander's theorem.
\endproof
\begin{remark}\label{ra6}
The same argument used to prove \eqref{ae30} applies to the case of $\phi(s,y)=(s-K)^{+}$, and allows to prove that the expectation $E_{t,s,v}\big[(S_{T}-K)^{+}\big]$ solves the backward equation
\eqref{eq:backward_Kolm} as a function of $(t,s,v)$. Indeed, it suffices to use a standard localization technique and the
fact that the Call payoff $(S_{T}-K)^{+}$ is integrable because $S$ is a martingale by assumption.
\end{remark}

\section{Analytical approximations of prices and implied volatilities}\label{asec3}
Here we briefly recall the construction proposed in \cite{LPP2} of an
explicit approximating series for option prices, along with a consequent polynomial expansion for
the related implied volatility. Such construction relies on a singular perturbation technique that
allows, in its most general form, to carry out closed-form expansions for the local transition density; this leads to an approximation of the solution to the
related backward Cauchy problem with generic final datum $\varphi$. Such technique has been
recently fully described in \cite{LPP4} in the
uniformly parabolic setting,
and subsequently extended in \cite{PP_compte_rendu} to the case of
locally parabolic operators and in \cite{lorig-pagliarani-pascucci-1} to models with jumps.
{Moreover, a recent extension of this technique to \emph{utility
indifference pricing} was proposed by \cite{Lorig-indifference}.}

We consider a model $Z=(S,Y)$ that satisfies the Assumptions \ref{assum1i}, \ref{assum1and} and
\ref{assum1ii} in Section \ref{sec:model}. We denote by $C_{t,T,K}$ the time $t$ no-arbitrage
value of a European Call option
with positive strike $K$ and maturity $T\le T_{0}$, defined as $C_{t,T,K} = v(t,S_t,Y_t;T,K)$ where
\begin{align}\label{eq:v}
 v(t,s,y;T,K):= E_{t,s,y} [(  S_T-K )^{+}],\qquad (t,s,y)\in [0,T]\times \R_{\ge 0}\times
 \DomY.
\end{align}
Clearly\footnote{Simply note that $(S_T-K )^{+}\le S_{T}$ and $S$ is a martingale by assumption.}
we have $v(t,0,y;T,K)\equiv 0$ and therefore, to avoid trivial situations, we may assume a
positive initial price, i.e. $s>0$.
As a consequence of Theorem \ref{la1} (see also Remark
\ref{ra6}), for any positive $K$, the function $v$ in \eqref{eq:v} is such that
$v(\cdot,\cdot;T,K)\in C_{P}^{N+2,1}(]0,T[\times D)\cap C([0,T]\times {D})$ and solves the
backward Kolmogorov equation \eqref{eq:backward_Kolm}:
 $$\left(\p_{t}+\Acc_{t}\right)v(\cdot,\cdot;T,K)=0\qquad \text{on }]0,T[\times D.$$

As it will be shown in Section \ref{sec:impvol}, in order to obtain an explicit expansion of the implied volatility, it is crucial to expand the Call price {around} a Black\&Scholes price.
Since the perturbation technique that we employ naturally yields Gaussian approximations at the leading term, we shall work in logarithmic variables.
Therefore, for any $T\in]0,T_{0}]$ and $k\in\R$, we set
\begin{equation}\label{aeprice}
 u(t,x,y;T,k)= v\left(t,e^x,y;T,e^k\right),\qquad 0\leq t\leq T,\ (x,y)\in \R\times \DomY,
\end{equation}
where $v$ is the pricing function in \eqref{eq:v}. Here, $x$ and $k$ are meant to represent the
spot log-price of the underlying asset and the log-strike of the option, respectively. Note that, the
function $u$ is well defined regardless of the process $S$ hitting zero or not.

After switching to log-variables, the generator $\Acc_{t}$ in \eqref{ae75} is transformed
into the second order operator
\begin{align}\label{operator_AA}
 \Ac:= \frac{1}{2}\sum_{i,j=1}^{d}a_{ij}(t,z)\p_{z_{i}z_{j}}+\sum_{i=1}^{d}a_{i}(t,z)\p_{z_{i}},\qquad t\in [0,T_{0}],\ z=(x,y)\in \R\times
 \DomY,
\end{align}
with
\begin{align}
  a_{11}(t,x,y)&=e^{-2x}\bar{a}_{11}\big(t,e^x,y\big),\qquad a_1 (t,x,y) = - \frac{e^{-2x}}{2}\bar{a}_{11}
  \big(t,e^x,y\big),
\intertext{and, for $i,j=2,\dots,d$,}
  a_{1i}(t,x,y)&=e^{-x}\bar{a}_{1i}\big(t,e^x,y\big),\qquad
  a_{ij}(t,x,y)=\bar{a}_{ij}\big(t,e^x,y\big), \qquad a_i (t,x,y) = \bar{a}_{i}
  \big(t,e^x,y\big).
\end{align}
For the reader's convenience, we also recall the classical definitions of Black\&Scholes price and
implied volatility given in terms of the spot log-price and the log-strike.
\begin{definition}
We denote by $u^{\BS}$ the \emph{Black\&Scholes price} function defined as
\begin{align}
u^{\BS}(\s;\t,x,k)
    &:= \ee^x \Nc(d_{+}) - \ee^k \Nc(d_{-}), &
d_{\pm}
    &:= \frac{1}{\sig \sqrt{\tau}} \( x - k \pm \frac{\sig^2 \tau}{2}  \),\qquad  x,k\in\R,\
    \s,\tau>0,
\end{align}
where $\Nc$ is the CDF of a standard normal random variable.
\end{definition}
\begin{definition}\label{def:BS-IV}
The \emph{implied volatility} $\s=\s(t,x,y;T,k)$ of the price $u(t,x,y;T,k)$ as in \eqref{aeprice} is
the unique positive solution of the equation
\begin{align}
 u^{\BS}(\sig;T-t,x,k)    &= u(t,x,y;T,k).   \label{eq:imp.vol.def}
\end{align}
\end{definition}
Note that Definition \ref{def:BS-IV} is well-posed because $C_{t,T,K}$ is a no-arbitrage price and
thus $u(t,x,y;T,k)$ belongs to the no-arbitrage interval $](e^{x}-e^{k})^+,e^x[$.

The computations in the following two subsections are meant to be formal and not rigorous. They
only serve the purpose to lead us through the definition of an approximating expansion for prices
and implied volatilities. The well-posedness of such definitions will be clarified, under rigorous
assumptions
in Section \ref{sec:error}.

\subsection{Price expansion}\label{sseca1}

We fix $\bar{z}=(\bar{x},\bar{y})\in\R \times \R^{d-1}$, such that
$\left(e^{\bar{x}},\bar{y}\right)\in D$ with $D$ as in Assumption \ref{assum1and}, and expand the
operator $\Ac_{t}$ by replacing the functions $a_{ij}(t,\cdot)$, $a_{i}(t,\cdot)$ with their
Taylor series around $\bar{z}$. We formally obtain
\begin{equation}
  \Ac_{t} =  \sum_{n=0}^\infty  \Ac^{(\bar{z})}_{t,n}
\end{equation}
where
\begin{equation}
 \Ac^{(\bar{z})}_{t,n}
 =\sum_{|\b|=n} \bigg(\sum_{i,j=1}^d
 \frac{D^{\b}a_{ij}(t,\bar{z})}{\b!}(z-\bar{z})^{\b} \partial_{z_i z_j}   +   \sum_{i=1}^d
 \frac{D^{\b}a_{i}(t,\bar{z})}{\b!}(z-\bar{z})^{\b} \partial_{z_i}   \bigg). \label{eq:A.expand}
\end{equation}
The intuitive idea underlying the following procedure is inspired by the fact that, typically, the pricing function
$u(\cdot,\cdot;T,k)$ solves the backward Cauchy problem
\begin{equation}\label{eqpde}
\begin{cases}
 (\p_t + \Ac_{t}) u(\cdot,\cdot;T,k)=  0, \qquad& \text{on }[0, T[\times\R\times \DomY,\\
  u(T,x,y;T,k)= \left(e^{x}-e^k\right)^{+}, & (x,y)\in \R\times \DomY.
\end{cases}
\end{equation}
Actually, \eqref{eqpde} holds automatically true if the operator $(\partial_t + \Ac_{t})$ is
uniformly parabolic and can be also proved to be satisfied, case by case, in many degenerate cases
of interest in mathematical finance, such as the CEV model. Nevertheless, the validity of
\eqref{eqpde} is not necessary for our analysis and {\it it is not required as an assumption}.

Next we assume that the pricing function $u$ can be expanded as
\begin{align}
  u  &=  \sum_{n=0}^\infty u^{(\bar{z})}_n. \label{eq:v.expand}
\end{align}
Inserting \eqref{eq:A.expand} and \eqref{eq:v.expand} into \eqref{eqpde} we find that the
functions $(u_{n}(\cdot,\cdot;T,k))_{n\geq 0}$ satisfy the following sequence of nested Cauchy
problems
\begin{align}\label{eq:v.0.pide}
&\begin{cases}
 (\partial_t + \Ac_{t,0} ) u^{(\bar{z})}_0(\cdot,\cdot;T,k) =  0, \qquad& \text{on }[0,T[\times\R^{d}, \\
 u^{(\bar{z})}_0(T,x,y;T,k) = \left(e^{x}-e^k\right)^{+},& (x,y) \in\R\times\R^{d-1},
\end{cases}
\intertext{and} \label{eq:v.n.pide} &\begin{cases}
 (\partial_t + \Ac_{t,0} ) u^{(\bar{z})}_n(\cdot,\cdot;T,k) =  - \sum\limits_{h=1}^{n} \Ac^{(\bar{z})}_{t,h} u^{(\bar{z})}_{n-h}(\cdot,\cdot;T,k), \qquad& \text{on }
 [0,T[\times\R^{d}, \\
 u^{(\bar{z})}_n(T,z;T,k) =  0, &    z \in\R^{d}.
\end{cases}
\end{align}
Note that, by Assumption \ref{assum1and}, $\Ac_{t,0}$ is an elliptic operator with time-dependent
coefficients and therefore
problem \eqref{eq:v.0.pide} can be solved to obtain
\begin{align}
{u^{(\bar{z})}_0(t,x,y;T,k)}
    &= u^\BS\big(\sig^{(\bar{z})}_0;T-t,x,k\big), \qquad \sig_0^{(\bar{z})}\equiv \sig_0^{(\bar{z})}(t,T)
    =  \sqrt{{\frac{1}{T-t}} \int_t^T a_{11}(\t,\bar{z}) d \t},
 \label{e10}
\end{align}
for any $t\in [0,T]$ and $(x,y)\in\R\times\R^{d-1}$. As for the $n$-th order correcting term
$u^{(\bar{z})}_n$, an explicit representation in terms of differential operators acting on
$u^{(\bar{z})}_0$ is available (see Theorem \ref{th:un_general_repres}).
\begin{definition}
\label{def:taylor} {For fixed maturity date $T$ and log-strike $k$,} we define the \emph{$N$-th
order approximations} of $u(\cdot,\cdot;T,k)$ as
\begin{align}
 \bar{u}_{N}(t,z;T,k)= \sum_{n=0}^N u^{({z})}_n(t,z;T,k),\qquad t\in\left[0,T\right],\ z\in \R\times \DomY, \label{eq:u.approx}
\end{align}
where the functions {$u^{(z)}_{n}$} are explicitly defined as in \eqref{e10}-\eqref{eq:un}.
\end{definition}
We recall that similar price expansions have been developed by \cite{bengobet2010},
\cite{TakahashiYamada2015} using Malliavin calculus techniques and by \cite{BayerLaurence} using
heat kernel methods.

\subsection{Implied volatility expansion}\label{sec:impvol}

We briefly recall how to derive a formal polynomial IV expansion from the price expansion
\eqref{eq:v.expand}-\eqref{eq:v.0.pide}-\eqref{eq:v.n.pide}. To ease notation, we will sometimes
suppress the dependence on $(t,x,y;T,k)$.
Consider the family of approximate Call prices indexed by $\del$
\begin{align}
 u^{(\bar{z})}(\delta)
     =  u^{\BS}\left(\sigma^{(\bar{z})}_0\right)+\sum_{n=1}^N \delta^n u^{(\bar{z})}_{n} +
     \delta^{N+1} \left(u-\sum_{n=0}^N  u^{(\bar{z})}_{n}\right), \qquad \del
    \in    [0,1], \label{eq:u.expand.again}
\end{align}
with $\sigma^{(\bar{z})}_0$ as in \eqref{e10} and the functions {$u^{(\bar{z})}_n$} as
in Subsection \ref{sseca1}.
Note that setting $\del=1$ yields the true pricing function $u$. Defining
\begin{align}
 g(\del) :=  (u^\BS)^{-1}(u(\del)),\qquad \del\in    [0,1], \label{eq:matt}
\end{align}
we seek the implied volatility $\sig=g(1)$. We will show in Section \ref{sec:error.impvol}, Lemma
\ref{lem:imp_vol_u_delta}, that under suitable assumptions $u(\delta)\in
](\ee^{x}-\ee^{k})^+,\ee^x[$ for any $\del\in[0,1]$. This
guarantees that $g(\del)$ in \eqref{eq:matt} is well defined. By expanding both sides of \eqref{eq:matt} as a Taylor series in $\del$, we see that $\sigma$
admits an expansion of the form
\begin{align}
 \sig=g(1)=  \sig_0 + \sum_{n=1}^\infty  \sig_{n},\qquad \sig_{n}=  \frac{1}{n!} \partial_{\delta}^{n} g(\delta) \vert_{\delta=0}
   . \label{ae6}
\end{align}
Note that, by \eqref{eq:u.expand.again} we also have
\begin{align}\label{ae8_bis}
  u_{n}=  \frac{1}{n!}\partial_{\delta}^{n}u^{\BS}(g(\delta))\vert_{\delta=0},\qquad 1\leq n \leq
  N,
\end{align}
and by applying the Faa di Bruno's formula (Proposition \ref{prop:multivariate_faa}), one can find
the recursive representation
\begin{equation}
\sig^{(\bar{z})}_n
    =  \frac{u^{(\bar{z})}_n}{\p_\sig u^\BS\big(\sig^{(\bar{z})}_0\big)}
            - \frac{1}{n!} \sum_{h=2}^{n}
       \mathbf{B}_{n,h}\(1!\,\sig^{(\bar{z})}_{1},2!\,\sig^{(\bar{z})}_{2},\dots,(n-h+1)!\, \sig^{(\bar{z})}_{n-h+1}\)
            \frac{\p_\sig^h u^\BS\big(\sig^{(\bar{z})}_0\big)}{\p_\sig u^\BS\big(\sig^{(\bar{z})}_0\big)}, \qquad
1
    \leq n \leq N, \label{eq:sig.n_bis}
\end{equation}
where $\mathbf{B}_{n,h}$ denote the so-called Bell polynomials. It was shown in \cite{LPP2} (see
also Proposition \ref{prop.un}) that each term $\sig^{(\bar{z})}_n$ is a polynomial in the
log-moneyness $(k-x)$. Moreover, if the coefficients of the model are time-independent, then the
expansion turns out to be also polynomial in time.
\begin{definition}
\label{def:taylor.sigma} {For a Call option with $\log$-strike $k$ and maturity $T$,} we define
the  \emph{$N$-th order approximation of the implied volatility} $\sigma(t,x,y;T,k)$ as
\begin{align}
\bar{\sigma}_N(t,x,y;T,k):=\sum_{n=0}^N {\sigma^{({x,y})}_n(t,x,y;T,k)}, \label{eq:sig.approx}
\end{align}
where {$\sigma^{({x,y})}_n$} are as defined in
\eqref{eq:sig.n_bis}.
\end{definition}
We recall that similar implied volatility expansions have been developed by \cite{BenArousLaurence}, \cite{DeuschelFriz}, \cite{forde-jacquier-lee}, and \cite{gatherallocal} among others.

\section{Error estimates for prices and sensitivities}
\label{sec:error}

In this section we derive error estimates for prices and sensitivities. Let us introduce the
following
\begin{notation}
For $z_{0}=(x_0,y_0)\in \DomXY$ and $0< r\leq +\infty$,
we set
 $$ D(z_{0},r)= B(x_0,r) \times B(y_0,r), $$
with $B(x_0,r)=\{x\in\R\mid |x-x_{0}|<r\}$ and $B(y_0,r)=\{y\in\R^{d-1}\mid |y-y_{0}|<r\}$.
Moreover, for $T\in ]0,T_{0}[$, we consider the cylinders $\Ho(T,z_{0},r), \Hb(T,z_{0},r)$ and the lateral boundary $\Sigma(T,z_{0},r)$
defined by
 \begin{align}
 \Ho(T,z_{0},r):= ]0,T[ \times D(z_0,r),\qquad \Hb(T,z_{0},r):= [0,T[ \times D(z_0,r),\qquad
 \Sigma(T,z_{0},r):= [0,T[ \times \p D(z_0,r),
\end{align}
respectively.
\end{notation}
Since we work with logarithmic variables, we are going to restate Assumption \ref{assum1and} in terms of conditions on the operator $\Ac_{t}$ as defined in \eqref{operator_AA}.
We recall that $N\geq 2$ is an integer constant that is fixed throughout the paper.
\begin{assumption}\label{assum2and2}
There exist $M_0>0$, $0<r\leq+\infty$
and $z_0=(x_0,y_0)\in\DomXY$
such that
the operator $\Ac_{t}$ as in \eqref{operator_AA} coincides with $\tilde{\Ac}_{t}$ on
$\Hb(T_{0},z_{0},r)$, where $\tilde{\Ac}_{t}$ is a differential operator of the form
\begin{align}\label{operator_AA_tilde}
 \tilde{\Ac}_{t}= \frac{1}{2}\sum_{i,j=1}^{d}\tilde{a}_{ij}(t,z)\p_{z_{i}z_{j}}+\sum_{i=1}^{d}\tilde{a}_{i}(t,z)\p_{z_{i}},\qquad t\in [0,T_{0}[,\ z\in \R^{d},
\end{align}
such that, for some $M\in ]0,M_{0}]$ and $\e\in ]0,1[$, we have:
\begin{enumerate}
\item[i)] {\it Regularity and boundedness:} \
 the coefficients  $\tilde{a}_{ij},\tilde{a}_{i}\in C^{N+1}_P\big([0,T_0[\times \R^{d}\big)$, with partial derivatives up to order $N+1$ bounded by $\ce$.
\item[ii)] {\it Uniform ellipticity:}
\begin{align}\label{cond:parabolicity}
 \eps \ce|\z|^2\le  \sum_{i,j=1}^{d}\tilde{a}_{ij}(t,z)\z_{i}\z_{j}\le \ce |\z|^2,\qquad t\in\left[0,T_{0}\right[,\
 z,\z\in\mathbb{R}^d.
\end{align}
 \end{enumerate}
\end{assumption}
Note that, if Assumption \ref{assum2and2} is satisfied with $r=+\infty$, then the operator
$\Ac_{t}$ is uniformly elliptic with bounded coefficients. {The forthcoming error bounds will be
asymptotic in the limit of small $M(T-t)$; in particular, the constant $C$ appearing in the error
estimates will be dependent on $M_0$ but not on $M$.}

Assumption \ref{assum2and2} is (locally) equivalent to Assumptions \ref{assum1and}. Precisely, the
former implies the latter on the domain $D= ]e^{x_0 - r},e^{x_0 + r}[ \times B(y_0,r) $.
Therefore, when Assumptions \ref{assum1i}, \ref{assum1ii} and \ref{assum2and2} are in force, in
light of Theorem \ref{la1} there exists a local transition density $\bar{\G}$ on $D$  for the
process $(S,Y)$. We then define the \emph{logarithmic local density} $\G$ as
\begin{equation}\label{ae40}
  \G(t,x,y;T,\x,\y)= e^{\x}\, \bar{\G}\big(t,e^{x},y;T,e^{\x},\y\big),
  \end{equation}
for any $(T,\xi,\y)\in  H(T_0,z_0,r)$ and $(t,x,y)\in \Hb(T,z_0,r)$.
\begin{remark}\label{ra7}
{Clearly Lemma \ref{la5} and Theorem \ref{la1}
can be extended to $\G$ through the logarithmic change of variables.} In particular, in this section we will use
that:
\begin{itemize}
\item[(i)]
$\G(t,z;\cdot,\cdot)\in C_P^{N,1}\big(]t,T_{0}[\times D(z_0,r)\big)$ for any $(t,z)\in
\Hb(T_0,z_0,r)$;
\item[(ii)]  $\G(\cdot,\cdot;T,\z)\in {C_P^{N+2,1}\big( \Hb(T,z_0,r) \big)}$ for any {$(T,\z)\in \Ho(T_{0},z_{0},r)$}
and solves the backward Kolmogorov equation
   \begin{equation}\label{eq:backward_Kolm_log}
   \left(\p_{t}+\Ac_{t}\right)
   f=0\qquad \text{on $\Hb(T,z_0,r)$.}
   \end{equation}
{Moreover, for any $(T,\bar{z})\in \Ho(T_{0},z_{0},r)$
and $\phi\in C_b\left(D(z_0,r)\right)$, we have
\begin{equation}
\lim_{(t,z)\to(T,\bar{z})\atop t<T}\int_{D(z_0,r)}\G(t,z;T,\zeta)\phi(\zeta)d\zeta=\phi(\bar{z});
\end{equation}
}
\item[(iii)] if $u$ is the function as defined in \eqref{aeprice}, then for any $T\in ]0,T_0[$ and $k\in\R$, we have that
$u(\cdot,\cdot;T,k)\in C^{N+2,1}_P\big(\Hb(T,z_0,r)\big)\cap C\big([0,T]\times D(z_0,r)\big)$ and
solves equation \eqref{eq:backward_Kolm_log}.
\end{itemize}
\end{remark}
Next we prove sharp error estimates for the derivatives $\partial_k^m (u-\bar{u}_N)$. In
Subsection \ref{sec:price.error} we prove some global bounds in the case $r=+\infty$ and then in
Subsection \ref{sec:price.error.local} we prove analogous local bounds in the general case
$r<+\infty$.

\subsection{Error estimates for uniformly parabolic equations}\label{sec:price.error}
Throughout this section we assume Assumption \ref{assum2and2} satisfied with $r=+\infty$. Under this assumption $u$ is the unique\footnote{The solution is unique within the class of non-rapidly increasing
functions.} classical solution of the Cauchy problem \eqref{eqpde} and can be represented as
\begin{align}\label{eq:convolution_u}
u(t,z)=\int_{\mathbb{R}^d} \Gamma(t,z;T,\xi,\eta) \big(e^{\x}-e^{k}\big)^+ d \xi d\eta,\qquad
{t\in\left[0,T\right[},\ z\in\mathbb{R}^d,
\end{align}
{where $\Gamma $ is the fundamental solution of the uniformly parabolic operator { $(\partial_t +
\Ac_{t})$}.} In the following statement $\bar{u}_{N}$ is the $N$th order approximation of $u$ as
defined in \eqref{eq:u.approx}.
\begin{theorem}\label{th:error_estimates_taylor}
Let Assumptions \ref{assum1i}, \ref{assum1ii} and \ref{assum2and2} hold with $r=+\infty$. Then,
for any {$m,q\in\N_0$ with $m+2q\leq N$}, we have
\begin{equation}\label{eq:error_estimate}
 \big|{\partial^q_T} \partial^m_k \big( u- \bar{u}_{N}\big)(t,x,y;T,k) \big| \leq C e^x {M^{q}} \left(\ce({T-t})\right)^{\frac{{N-m{- 2q}+2}}{2}},
\end{equation}
for $0\leq t<T< T_0,$ $x,k\in\mathbb{R}$ and $y\in\R^{d-1}$. The constant $C$ in
\eqref{eq:error_estimate} depends only on $T_{0},\cem,\e,N$ and the dimension $d$. In particular,
$C$ is independent of $M$.
\end{theorem}
The proof of Theorem \ref{th:error_estimates_taylor}, which is postponed to Appendix
\ref{appth44}, is based on the following classical Gaussian estimates (see, for instance Chapter 1
in \cite{friedman-parabolic}, Corollary 5.5 in \cite{pascucci-parametrix} and
\cite{pascuccibook}).
\begin{lemma}\label{lem:gaussian_estimates}
Let $\G=\G(t,z;T,\zeta)$ be the fundamental solution of $(\Ac_{t}+\p_{t})$. Then, for any $c>1$,
{$q\in\N_0$} and $\b,\g\in\mathbb{N}_{0}^{d}$ with {$|\b|+2 q\le N$}, we have
\begin{align}\label{Gaua}
 \big|(z-\z)^{\g}{\partial_T^q D^{\b}_{\z}}\G(t,z;T,\zeta)\big|\le C {M^{q}} \left(M(T-t)\right)^{\frac{{|\g|-|\b| -2 q}}{2}}\Gbs^{
 }{\left(c\ce(T-t),z-\zeta\right) },\qquad 0\leq t<T \leq T_{0},\ z,\zeta\in\R^{d},
\end{align}
where $\Gbs$ is the $d$-dimensional standard Gaussian function
\begin{equation}\label{eq:normal_density_d}
 \Gbs(t,z) = (2\pi t)^{-\frac{d}{2}}\exp\left(  - \frac{|z|^2}{2 t} \right),\qquad  t\in\R_{>0},\ z\in\R^d,
\end{equation}
and {$C$ is a positive constant that depends only on $c,T_{0},\cem,\varepsilon,N$ and the
dimension $d$.}
\end{lemma}

\subsection{Error estimates for locally parabolic equations}\label{sec:price.error.local}

We now relax the global parabolicity assumption of Subsection \ref{sec:price.error}, by assuming that the
pricing operator $\Ac_{t}$ is only {\it locally elliptic}: precisely, throughout this section we
impose that Assumptions \ref{assum1i}, \ref{assum1ii} and \ref{assum2and2} hold for some $r>0$. We
first state the result in the one-dimensional case.
\begin{theorem}\label{t2a}
Let $d=1$. Under Assumptions \ref{assum1i}, \ref{assum1ii} and \ref{assum2and2},
for any $\d\in]0,1[$, $T\in]0,T_{0}[$ and $m\le N$ we have
\begin{equation}\label{e2}
  \left|\partial^m_k u(t,z;T,k)-\partial^m_k \bar{u}_{N}(t,z;T,k)\right|
  \le C (M(T-t))^{\frac{{N-m+2}}{2}},\qquad (t,z)\in \Hb(T,z_{0},\d r),\ |k-x_{0}|<\d r,
  \end{equation}
where $C$ is a  positive constant that depends only
on $r,z_{0},\d,d, M_0,\e,N$ and $T_0$. In particular, $C$ is independent of $M$.
\end{theorem}
The proof of Theorem \ref{t2a} is a simpler modification of that of Theorem \ref{t2} below, and therefore will be
omitted. Theorem \ref{t2} is the main result of this section: it gives estimates for the
derivatives of the price function w.r.t. the log-strike $k$ in dimension $d\ge 2$.

\medskip For the rest of the section we fix $\hat{N}\in\N_{0}$, with $\hat{N}\le N$, and consider
$d\ge 2$.  By our general assumptions (see, in particular, Remark \ref{ra7}) we have that, for any
$T\in ]0,T_0[$, $(t,z)\in \Hb(T,z_0,r)$, {$|k-x_{0}|<r$} and $\d\in[0,1]$, the pricing function
$u$ can be represented as
\begin{equation}\label{ae20}
 u(t,z;T,k)=I_{1,\d}(t,z;T,k)+I_{2,\d}(t,z;T,k),
 \end{equation}
where
\begin{align}
 I_{1,\d}(t,z;T,k)&=\int_{D(z_0,\d r)}\left(e^{\x}-e^{k}\right)^+\G(t,z;T,\x,\y)d \x d \y,\\
 I_{2,\d}(t,z;T,k)&=\int_{\R^d\setminus D(z_0,\d
 r)}\left(e^{\x}-e^{k}\right)^+ p(t,z;T,d\x,d\y),\label{ae89}
\end{align}
and $p$ denotes the transition {distribution} 
of the process $(\log S,Y)$. We
note explicitly that, even if $\log S$ takes value in $[-\infty,+\infty[$ (due to the possibility
for $S$ to reach $0$),
 we can exclude $\{-\infty\}\times\R^{d-1}$ from the domain of integration
of $I_{2,\d}$ because the Call payoff function
is null for $\x\le k$.

{Formula \eqref{ae20} is useful to study the regularity properties of $u$ w.r.t. $k$ {and $T$}. In
fact, by (i) of Remark \ref{ra7}, $I_{1,\d}$ is {twice differentiable in $k$, with $\partial^2_k
I_{1,\d}(t,z;\cdot,\cdot)\in C^{N}_P\big(]t,T_{0}[\times D(z_0,r)\big)$,} and we have
\begin{align}\label{ae90}
 {\p_{T}^{q}}\p_{k}^{m} I_{1,\d}(t,z;T,k) &=U_{{1,q,m,\d}}(t,z;T,k)+U_{{2,q,m,\d}}(t,z;T,k),\\
\intertext{where}
 U_{{1,q,m,\d}}(t,z;T,k)&=e^k \int_k^{x_0+\d r} \int_{|\y-y_0|<\d r} {\p_{T}^{q}} \G(t,z;T,\xi,\y) d \xi  d  \y,\\
 U_{{2,q,m,\d}}(t,z;T,k)&= e^{k} \sum_{j=1}^{m-1} \binom{m-1}{j} \int_{|\y-y_0|<\d r} {\p_{T}^{q}} \partial_{k}^{j-1}  \G(t,z;T,k,\y) d \y,
\end{align}
for $(t,z)\in \Hb(T,z_0,r)$ and $k\in B(x_0,\d r)$. However, the assumptions imposed in Section
\ref{sec:model} are not sufficient to ensure the existence of the derivatives $
{\p_{T}^{q}}\partial^m_k I_{2,\d}$ (and consequently of $ {\p_{T}^{q}}\partial^m_k u$). Indeed, a
formal computation gives
\begin{align}\label{ae91}
 {\p_{T}^{q}} \p_{k}^{m} I_{2,\d}(t,z;T,k) &=U_{{3,q,m,\d}}(t,z;T,k)+U_{{4,q,m,\d}}(t,z;T,k),\\
\intertext{where}
 U_{{3,q,m,\d}}(t,z;T,k)&= {\p_{T}^{q}}e^k \int_{[x_0+\d r,+\infty[ \times \DomY}  p(t,z;T,d\xi,d\eta),\\ \label{eq:ste132}
 U_{{4,q,m,\d}}(t,z;T,k)&= {\p_{T}^{q}}\partial_{k}^{m} \int_{]k,x_0+\d r[ \times (\DomY\setminus B(y_{0},\d r))
 } p(t,z;T,d\xi,d\eta) \left(e^{\x}-e^{k}\right).
\end{align}
Now, it is clear that {$U_{3,q,m,\d}$} depends smoothly on $k$. On the contrary, the existence and
boundedness properties of the derivatives {$U_{4,q,m,\d}$} depend on the tails of the distribution
and cannot be deduced from the general assumptions of Section \ref{sec:model} because of the {\it
local nature} of such assumptions. Notice that this problem only arises when $d\ge 2$ and
therefore, in order to prove results in the most general setting, we need to impose the following
additional
\begin{assumption}\label{exder}
{For any 
$(t,z)\in \Hb(T_0,z_0,r)$, the function $u(t,z;\cdot,\cdot)\in C^{\hat{N}}_P\big(]t,T_{0}[\times D(z_0,r)\big)$.}
Moreover,
in the case $\hat{N}\geq 2$, there exist $\d\in]0,1[$ and
some positive constants $\Cdt $ and $\Cdb $ such that
\begin{align}\label{ae28}
 &\left|{\p_{T}^{q}\p_{k}^{m}}\Gamma(t,z;T,k,\y)\right|\le \Cdt,\qquad {2q + m\le \hat{N}},
 \intertext{for any $(T,k,\y)\in \Ho(T_{0},z_0,\d^2 r)$, $(t,z)\in \Hb(T,z_0,r)\setminus
\Hb(T,z_0,{\d} r)$, and}
 \label{se24}
 &{\left| U_{3,q,m,\d^2}(t,z;T,k)\right| +} \left| U_{{4,q,m,\d^2}}(t,z;T,k)\right| \le \Cdb,\qquad {2q+m}\le \hat{N},
\end{align}
for any $ (T,k)\in ]0, T_{0}[\times B(x_0,\d^2 r)$ and $(t,z)\in \Hb(T,z_{0},\d^3 r)$.
\end{assumption}
\begin{remark}If $\log S_T$ (or, equivalently, $S_{T}$) has a {\it marginal} local density
${\G}_{S}(t,z;T,k)$ such that
  $$ \p_{T}^{q}\p_{k}^{m}{\G}_{S}(t,z;\cdot,\cdot)\in C\big(]t,T_{0}[\times B(x_0,r)\big),\qquad 2q+m\leq \hat{N},$$
then the first part of Assumption \ref{exder} is satisfied: in fact, $u(t,z;\cdot,\cdot)\in
C^{\hat{N}}_P\big(]t,T_{0}[\times B(x_0,r)\big)$ because it can be represented as
\begin{equation}
 u(t,z;T,k) = \int_{k}^{\bar{k}} {\G}_{S}(t,z;T,\xi) (e^{\xi} - e^{k})d\xi +
\int_{[\bar{k},+\infty[ } {p}_{S}(t,z;T,d\xi) (e^{\xi} - e^{k}),
\end{equation}
for some $\bar{k}>k$, where ${p}_{S}$ denotes the marginal transition probability of $\log
S$.
This is the case, for instance, of the Heston model where $S_{T}$ has a smooth marginal density (see Remark \ref{rem:Heston_density}).

The need for conditions \eqref{ae28} and \eqref{se24} will be clarified in the proofs of
Lemma \ref{lemm:ste_dens} and Theorem \ref{t2}, respectively. Condition \eqref{ae28} is
intuitively easy to understand: roughly speaking, it states that the derivatives of the local
density $\G(t,z;T,\z)$ are locally bounded, away from the pole, all the way up to $t=T$.
This looks like a sensible condition, given the boundedness hypothesis for the
diffusion coefficients on the whole cylinder. By opposite, condition \eqref{se24} might seem a little bit cryptic at a first glance; however, in
most cases of interest such hypothesis turns out to be substantially simplified.
For instance, in many financial models such as the Heston model, the local density ${\Gamma}$ is defined on the whole strip $B(x_{0},r)\times \DomY$ (see Remark \ref{rem:Heston_density}), i.e. we have
\begin{equation}\label{eq:ste131}
 p(t,z;T,H)=\int_{H}\G(t,z;T,\z) d \z, \qquad H\in\Bc\big(B(x_{0},r)\times \DomY\big).
\end{equation}
In this case, condition \eqref{se24} is automatically satisfied for {$q=0$ and} $m=0,1$, whereas
for $2\leq {m + 2 q}\leq \hat{N}$ it reduces to
\begin{equation}
{\left| \int_{[x_0+\d r,+\infty[ \times \DomY}  \p_{T}^{q} \G(t,z;T,\zeta) d \zeta \right|+}
\left|\int_{|\y-y_{0}|>{\d^2}r} {\partial_{T}^{q} \partial_{k}^{(m-2)\vee 0} }
\G(t,z;T,k,\eta)d\eta \right| \le \Cdb,
\end{equation}
for any $(T,k)\in ]0, T_{0}[\times B(x_0,\d^2 r)$, $(t,z)\in \Hb(T,z_{0},\d^3 r)$.
\end{remark}

We are now ready to state the main result of this section.

\begin{theorem}\label{t2}
Let $d\geq 2$, and let Assumptions \ref{assum1i}, \ref{assum1ii}, \ref{assum2and2} and \ref{exder}
be in force. Then, for any {$m,q\in\N_0$ with $m+2q\leq \hat{N}$} and $T\in ]0,T_{0}[$, we have
\begin{equation}\label{e2b}
  \left|{\partial^q_T}\partial^m_k \big(u- \bar{u}_{N}\big)(t,z;T,k)\right|
  \le C{M^q}\left(M(T-t)\right)^{\frac{{N-m-2q+2}}{2}},\qquad (t,z)\in \Hb(T,z_{0},\d^4 r),\
  |k-x_{0}|<\d^{4}r,
 \end{equation}
where $\d\in]0,1[$ is as in Assumption \ref{exder}, and the positive constant $C$ depends only on $r,z_{0},d, M_0,\e,N,T_0$ and, only if $\hat{N}\geq 2$, also on $\delta$ and the constants $\Cdt $ and $\Cdb $ in \eqref{ae28} and \eqref{se24}. In particular, $C$ is independent of $M$.
\end{theorem}

\begin{lemma}\label{lemma:zerolimit}
{Let $D_{0}$ be a domain of $\R^n$ and
  $$h(\cdot,\cdot;T,\th):\overline{\Ho(T,z_{0},r)}\longrightarrow \R,\qquad (T,\th)\in ]0,T_0[\times D_0,$$
such that: }
\begin{enumerate}
\item[i)] {for any $(t,z)\in [0,T_0[\times \overline{D(z_{0},r)}$, the function $h(t,z;\cdot,\cdot)\in C^{p}\big(]t,T_0[\times D_{0}\big)$}
with derivatives {$\partial^q_T D^{\beta}_{\theta} h(t,z;T,\th)$ locally bounded in $(T,\th)$},
uniformly w.r.t. {$(t,z)\in  [0,T[ \times \big(\, \overline{D(z_{0},r)}\setminus D(z_{0},\r_{0}
r)\big)$ for a certain $\r_{0}\in ]0,1[$};
\item[ii)] for any {$(T,\th)\in\, ]0,T_0[\times D_{0}$} the function
$h(\cdot,\cdot;{T,\th})\in C^{1,2}\left(\Hb(T,z_{0},r)\right)\cap
C\big(\overline{H(T,z_{0},r)}\big)$ and verifies
\begin{equation}
\begin{cases}
 \left(\p_t + \tilde{\Ac}_{t}\right)h(t,z;T,\th)=  0,\qquad  &(t,z)\in  \Hb(T,z_{0},r),\\
 h(T,z;T,\th)=0,&  z\in D(z_{0},r).
\label{eqpde_tilde_ste}
\end{cases}
\end{equation}
\end{enumerate}
Then for any multi-index $\beta\in \N_0^{n}$ and {any $q\in\N_0$ with $q + |\beta|\leq p$}, we
have
\begin{equation}\label{eq:zerolimit}
 \lim_{(t,z)\to (T,\bar{z})\atop {t<T} } {\partial_T^q} D^{\beta}_\th h(t,z;T,\th)= 0, \qquad \bar{z}\in D(z_{0},r),\ {(T,\th)\in ]0,T_0[\times D_{0}}.
\end{equation}
\end{lemma}
\begin{proof} By induction on $q$ we prove \eqref{eq:zerolimit} and that, {for any
$\r\in[\r_{0},1[$,} we have
\begin{equation}\label{eq:ste404}
 \p_{T}^{q}D^{\beta}_\th h(t,z;{T}, \th)=\int_{t}^{T} \int_{\partial D(z_{0},\r r)} P_{\r r}(t,z;s,\z) \p_{T}^{q}D^{\beta}_\th h(s,\z;{T},\th)  d \z d s, \qquad (t,z)\in \Ho(T,z_{0},\r r),
\end{equation}
where $P_{\r r}$ denotes the {Poisson kernel} of the uniformly parabolic operator $\big(\p_t +
\tilde{\Ac}_{t}\big)$ on $\Ho(T,z_{0},\r r)$.

For $q=0$, differentiating the representation formula
\begin{equation}\label{eq:ste404b}
  h(t,z;{T}, \th)=\int_{t}^{T} \int_{\partial D(z_{0},\r r)} P_{\r r}(t,z;s,\z) h(s,\z;{T},\th)  d \z d s, \qquad (t,z)\in \Ho(T,z_{0},\r r),
\end{equation}
and using the terminal condition in \eqref{eqpde_tilde_ste}, we obtain
\begin{equation}
  \big|D^{\beta}_\th h(t,z;T,\theta)\big| \leq {\big\|D^{\beta}_\th h(\cdot,\cdot;T,\theta)
  \big\|_{L^{\infty}(\Sigma(T,z_{0},\r r) ) } }
  \int_{t}^{T} \int_{\partial D(z_{0},\r r)} P_{\r r}(t,z;s,\z)  d \xi d s, \qquad (t,z)\in H(T,z_{0},\r r),
\end{equation}
which in turn implies \eqref{eq:zerolimit} with $q=0$.

Next, we assume \eqref{eq:zerolimit} and \eqref{eq:ste404} true for $q$: by differentiating
\eqref{eq:ste404} we get
\begin{align}
 \p_{T}^{q+1}D^{\beta}_\th h(t,z;{T}, \th) = &\int_{\partial D(z_{0},\r r)} P_{\r
 r}(t,z;T,\z) \p_{T}^{q}D^{\beta}_\th h(T,\z;{T},\th)  d \z
 \\ &+\int_{t}^{T} \int_{\partial D(z_{0},\r r)} P_{\r
 r}(t,z;s,\z) \p_{T}^{q+1}D^{\beta}_\th h(s,\z;{T},\th)  d \z d s=
\intertext{(by \eqref{eq:zerolimit})}
 =&\int_{t}^{T} \int_{\partial D(z_{0},\r r)} P_{\r
 r}(t,z;s,\z) \p_{T}^{q+1}D^{\beta}_\th h(s,\z;{T},\th)  d \z d s, \qquad (t,z)\in \Ho(T,z_{0},\r
 r).
\end{align}
Then, for $(t,z)\in H(T,z_{0},\r r)$ we have
\begin{equation}
  \big|\p_{T}^{q+1}D^{\beta}_\th h(t,z;T,\theta)\big| \leq {\big\|\p_{T}^{q+1}D^{\beta}_\th h(\cdot,\cdot;T,\theta)
  \big\|_{L^{\infty}(\Sigma(T,z_{0},\r r) ) } }
  \int_{t}^{T} \int_{\partial D(z_{0},\r r)} P_{\r r}(t,z;s,\z)  d \xi d s,
\end{equation}
which concludes the proof. 
%
%
%
%
%
\end{proof}

The following lemma is preparatory for the proof of Theorem \ref{t2}, but it may also have an
independent interest: it shows that the difference between $\G$ and $\tilde{\G}$, and of their
derivatives, decays exponentially on $\Ho(T,z_0,r)$ as $t$ approaches $T$.
\begin{lemma}\label{lemm:ste_dens} Let $\hat{N}\geq 2$ and let $\tilde{\G}$ be the fundamental solution of the uniformly
parabolic operator $\big(\p_t + \tilde{\Ac}_{t}\big)$. Then, under the assumptions of Theorem
\ref{t2}, for any {$m,q\in\N_0$ with $m+2q\leq \hat{N}$} we have
\begin{equation}\label{eq:ste43}
 \left|{\partial_T^q {\partial_{k}^{m}}}\big(\G-\tilde{\G}\big)(t,z;T,k,\eta)\right|
 \leq C e^{-\frac{1}{C\sqrt{M(T-t)}}}, \qquad (T,k,\y)\in \Ho(T_{0},z_0,\d^2 r),\ (t,z)\in \Hb(T,z_{0},\d^2 r),
 \end{equation}
where $C$ is a positive constant that depends only on $z_0,\d,N,d, M_0,\e,T_0$, and on
$\Cdt,\Cdb $ in \eqref{ae28} and \eqref{se24}.
\end{lemma}
\begin{proof}
\emph{Step 1.}
Fix
$(T,k,\y)\in  \Ho(T_{0},z_{0},\d^2 r)$ and consider the function
  $${w_{q,m}}(t,z):={\partial_T^q \partial_{k}^{m}}\big(\G-\tilde{\G}\big)(t,z;T,k,\y),\qquad (t,z)\in \Hb(T,z_{0}, r).$$
We prove that
\begin{equation}\label{ae21}
  \begin{cases}
    \big(\p_t +  \tilde{\Ac}_{t}\big){w_{q,m}}=0, & \text{on } \Hb(T,z_{0}, r), \\
   \lim\limits_{(t,z)\to (T,\bar{z})\atop t<T} {w_{q,m}}(t,z)=  0, & \bar{z}\in D(z_{0},r).
  \end{cases}
\end{equation}
The first equation in \eqref{ae21} follows from the fact that $\Ac_{t}$ and $\tilde{\Ac}_{t}$
coincide on $\Hb(T_{0},z_{0},r)$.
To prove the second one, we set
\begin{equation}
 h(t,z;k):=\int_{D(z_{0},r)} \left(\G(t,z;T,\z)-\tilde{\G}(t,z;T,\z)\right) \psi(\z-(k,\y))d \z,\qquad
 (t,z)\in  \Hb(T,z_{0},r),
\end{equation}
where
\begin{equation}
\psi(z):= \prod_{i=1}^{d}\z_i^{+}, \qquad \z=(\z_1,\dots,\z_d)\in\R^d.
\end{equation}
Notice that $h(\cdot,\cdot;T,k,\y)$ satisfies
\begin{equation}
 \begin{cases}
 \big(\p_t + \tilde{\Ac}_{t}\big)h(t,z;k)=  0,\qquad  &(t,z)\in \Hb(T,z_{0},r),
 \\
 h(t,z;k)=  0,& z\in D(z_{0},r).
\end{cases}
\end{equation}
Moreover,
we have
  $$\partial_{k}^{2} \partial_{\y_2}^2 \cdots \partial_{\y_d}^2 h(t,z;k)=\G(t,z;T,k,\y)-\tilde{\G}(t,z;T,k,\y),$$
and therefore also
\begin{equation}
{\partial_T^q \partial_{k}^{2+m} } \partial_{\y_2}^2 \cdots \partial_{\y_d}^2
h(t,z;k)={w_{q,m}}(t,z).
\end{equation}
Hence, by applying Lemma \ref{lemma:zerolimit} to $h$
we obtain the limit in \eqref{ae21}.

\medskip\noindent
\emph{Step 2.} It suffices to prove the thesis for $T-t$ suitably small and positive.
In \cite[Theorem 3.1]{PP_compte_rendu} we proved that
there exist $\t>0$ and a non-negative
function $v$ such that
\begin{equation}\label{e5}
\begin{cases}
   \big(\p_{t}+\tilde{\Ac}_{t}\big) v(t,z)=0,\qquad &  (t,z)\in [T-\t,T[\times D(z_{0}, r),\\
  v(t,z)\ge 1,&  (t,z)\in [T-\t,T[\times \p D(z_{0}, r),
\end{cases}
\end{equation}
and
\begin{equation}\label{e6}
  0< v(t,z)\le C e^{-\frac{r^{2}}{C\sqrt{M(T-t)}}},\qquad  (t,z)\in [T-\t,T[\times D(z_0,\d^2 r),
\end{equation}
where the positive constant $C$ depends only on $\d,M_{0},\e,T_{0},z_{0}$ and $d$.
Now, by \eqref{e5}, \eqref{e6}, and by the limit in \eqref{ae21} together with the bound \eqref{ae28}, one has
\begin{equation}
\liminf_{(t,z)\to (\bar{t},\bar{z})\atop (t,z)\in [T-\tau,T[\times D(z_0,r)} \big(\Cdt v -
{w_{q,m}}\big)(t,z) \geq 0, \qquad (\bar{t},\bar{z})\in  \big( \{T\}\times D(z_0,r)
\big)\cup\big([T-\tau,T[\times \partial D(z_0,r)\big).
\end{equation}
Therefore, the maximum principle yields
 $$
 \left|{w_{q,m}}(t,z)\right| \leq \Cdt
 v(t,z),\qquad (t,z)\in [T-\t,T[\times D(z_0,r),
 $$
and eventually, \eqref{eq:ste43} stems from \eqref{e6}.
\end{proof}
\begin{proof}[of Theorem \ref{t2}]
We only prove the statement for $2\leq m\leq \hat{N}$, being the other cases simpler. Throughout
the proof, we denote by $C$ every positive constant that depends at most on $r,z_{0},\d,d,
M_0,\e,N,T_0$ and on $\Cdt,\Cdb $ in \eqref{ae28} and \eqref{se24}.

\medskip\noindent
\emph{Step 1.}
We fix $T\in ]0,T_0[$ and prove that
\begin{equation}\label{ae22}
  \left|{w_{q,m}}(t,z;T,k)\right| \leq C,\qquad (t,z)\in \Hb(T,z_{0},\d^3 r),\ k\in B(x_0,\d^3 r),
\end{equation}
where ${w_{q,m}}:={\partial_T^q } \partial_k^m ( u- \tilde{u})$ and
\begin{equation}\label{eq:ste50}
 \tilde{u}(t,z;T,k):=\int_{k}^{\infty} \int_{\R^{d-1}} \tilde{\G}(t,z;T,\xi,\y)
 \left( e^{\xi}-e^{k}\right) d \xi d \y, \qquad (t,z)\in [0,T[\times \R^{d}.
\end{equation}
Differentiating formula \eqref{ae20} and recalling \eqref{ae90} and \eqref{ae91}, we get
\begin{align}\label{eq:ste_estim_dk_u}
 {\partial_T^q }\p_{k}^{m} u(t,z;T,k) &=\sum_{i=1}^{4}(-1)^{i}U_{{i,q,m,\d}}(t,z;T,k).
\end{align}
Analogously, differentiating \eqref{eq:ste50} we obtain
\begin{equation}\label{ae25}
 {\partial_T^q } \partial_k^m \tilde{u}(t,z;T,k) = -e^k \int_k^{\infty} \int_{\R^{d-1}} {\partial_T^q } \tilde{\G}(t,z;T,\xi,\y) d \xi  d \y +  \sum_{j=1}^{m-1} \binom{m-1}{j} e^{k}
 \int_{\R^{d-1}} {\partial_T^q }\partial_{k}^{j-1}  \tilde{\G}(t,z;T,k,\y) d \y.
\end{equation}
Thus we have
 \begin{align}
 \left|{w_{q,m}}(t,z;T,k)\right| & \leq
 C\bigg( 1+U_{{4,q,m,\d^2}}(t,z;T,k)+\sum_{j=1}^{m-1}\left(J_{{1,q,j,\d^2}}+J_{{2,q,j,\d^2}}\right)(t,z;T,k)\bigg)\\
 & \leq
 C\bigg( 1+\sum_{j=1}^{m-1}\left(J_{{1,q,j,\d^2}}+J_{{2,q,j,\d^2}}\right)(t,z;T,k)\bigg)&& (\text{by \eqref{se24}})
 \end{align}
for any $ k\in B(x_0,\d^2 r)$ and $(t,z)\in \Hb(T,z_{0},\d^3 r)$,
where
\begin{align}
 J_{{1,q,j,\d^2}}(t,z;T,k)&=\int_{|\y-y_{0}|<\d^2 r}\left|{\partial_T^q }\partial_{k}^{j-1} (\G - \tilde{\G})(t,z;T,k,\y)\right| d \y,\\
 J_{{2,q,j,\d^2}}(t,z;T,k)&=\int_{
 |\y-y_{0}|\ge \d^2 r}\left|{\partial_T^q }\partial_{k}^{j-1}\tilde{\G}(t,z;T,k,\y)\right| d  \y.
\end{align}
Now, by applying Lemma \ref{lemm:ste_dens} and standard Gaussian estimates on the functions
$J_{{1,q,j,\d^2}}$ and $J_{{2,q,j,\d^2}}$ respectively, we obtain that {the latter} 
are bounded by a constant $C$ for any $ k\in B(x_0,\d^2 r)$ and $(t,z)\in \Hb(T,z_{0},\d^3 r)$. This proves
\eqref{ae22}.

\medskip\noindent
\emph{Step 2.} Fix now $(T,k)\in ]0,T_{0}]\times B(x_0,\d^3)$. Clearly,
$\tilde{u}(\cdot,\cdot;T,k)$ in \eqref{eq:ste50} is a classical solution to the Cauchy problem
\begin{equation}
 \begin{cases}
 \big(\p_t + \tilde{\Ac}_{t}\big)\tilde{u}(\cdot,\cdot;T,k)=  0,\qquad  &\text{on }  [0,T[\times \R^{d},
 \\
 \tilde{u}(T,x,y;T,k)=  \left(e^x - e^k \right)^{+},& (x,y) \in \R^d.
\end{cases}
\end{equation}
We set $h(t,z;k):=\left(u-\tilde{u}\right)(t,z;T,k)$ and notice that, by Remark \ref{ra7}-(iii), we have
\begin{equation}\label{ae26}
 \big(\p_t + \tilde{\Ac}_{t}\big)h(\cdot,\cdot;k)= 0,\qquad \text{on } \Hb(T,z_{0},r),
\end{equation}
because $\Ac_{t}$ and $\tilde{\Ac}_{t}$ coincide on $\Hb(T_{0},z_{0},r)$; moreover, we have
\begin{equation}\label{ae26bis}
 h(T,z;k)= 0,\qquad z\in D(z_{0},r).
\end{equation}
Now, by estimate \eqref{ae22} the derivatives ${\partial_T^q}\p_{k}^{m}h={w_{q,m}}$ are bounded on
$\Sigma(T,z_{0},\d^{3}r)$ for $k\in B(x_0,\d^3)$. Then, from Lemma \ref{lemma:zerolimit} applied
to $h$ on $\Hb(T,z_{0},\d^{3} r)$, we infer
\begin{equation}\label{ae26ter}
 \lim\limits_{(t,z)\to (T,\bar{z})\atop t<T} {w_{q,m}}(t,z;T,k)=  0,\qquad \bar{z}\in D(z_{0},\d^{3} r).
\end{equation}
By differentiating \eqref{ae26}, we also have $\big(\p_t +
\tilde{\Ac}_{t}\big){w_{q,m}}(\cdot,\cdot;T,k)=0$ on $\Hb(T,z_{0}, \d^{2}r)$. Thus we can use the
same argument used in Part 2 of the proof of Lemma \ref{lemm:ste_dens}: precisely, we consider the
function $v$ satisfying \eqref{e5}-\eqref{e6} and, by the maximum principle, \eqref{ae26ter} and
\eqref{ae22} we infer
\begin{align}
 \left|{w_{q,m}}(t,z;T,k)\right| \leq \big\| {w_{q,m}}(\cdot,\cdot;T,k)\big\|_{L^{\infty}(\Sigma(T,z_0,\d^{3} r))}
 e^{-\frac{r^{2}}{C\sqrt{M(T-t)}}},\qquad (t,z) \in \Hb(T,z_{0},\d^{4}r).
\end{align}
Eventually, by the triangular inequality we get $$ \left|\p_{k}^{m}(u-\bar{u}_{N})\right|\le
\left|{w_{q,m}}\right|+\left|\p_{k}^{m}(\tilde{u}-\bar{u}_{N})\right| \le
  C  e^{-\frac{r^{2}}{C\sqrt{M(T-t)}}}+ \left|\p_{k}^{m}(\tilde{u}-\bar{u}_{N})\right|,
  \qquad \text{ on } \Hb(T,z_{0},\d^4 r),
$$ and the statement follows from the asymptotic estimate of Theorem
\ref{th:error_estimates_taylor} applied to the uniformly parabolic operator
$\big(\p_{t}+\tilde{\Ac}_{t}\big)$.
\end{proof}

\section{Error estimates and Taylor formula of the implied volatility}\label{sec:error.impvol}
In this section we establish error estimates for the $N$-th order implied volatility approximation
$\bar{\sigma}_{N}(t,x,y;T,k)$ in Definition \ref{def:taylor.sigma} and for its derivatives w.r.t.
$k$ and $T$. Such bounds are proved under the assumptions of Subsection
\ref{sec:price.error.local} and are valid in the \emph{parabolic} domain $|x-k|\leq
\lam\sqrt{M(T-t)}$, for {any $\lam>0$ and suitably} small time-to-maturity $(T-t)$, with $M$ being
the local-ellipticity constant in Assumption \ref{assum2and2}. We recall that $N,\hat{N}\in\N_0$
are fixed throughout the paper and such that $N\geq 2$ and $\hat{N}\leq N$. Moreover
$z_0=(x_0,y_0)\in\R\times\R^{d-1}$ is the center of the cylinder in Assumptions \ref{assum2and2}
and \ref{exder}.
\begin{theorem}\label{th:IV_error}
Let $d= 1$ ($d\geq 2$) and let the assumptions of Theorem \ref{t2a} (Theorem \ref{t2}) be in
force. Then, for any $\lam>0$ and $m,{q}\in\N_0$ with {$2q+m\leq \hat{N}$}, there exist two
positive constants $C$ and $\tau_0$ such that
\begin{equation}\label{eq:IV_estimate}
 \left| {\partial_T^q} \partial_k^m \sigma(t,x_0,y_0;T,k)-  {\partial_T^q} \partial_k^m\bar{\sigma}_{N}(t,x_0,y_0;T,k) \right| \leq C 
 {M^{q+\frac{1}{2}}}\big(M(T-t)\big)^{\frac{N-m-{2q}+1}{2}},
\end{equation}
for any $0\leq t<T<T_0$ and $k$ such that $T-t\leq \tau_0$ and
$|x_0-k|\leq \lam \sqrt{M(T-t)}$. The constants $C$ and $\tau_0$
depend only on $r,z_0,d, M_0,\e,N,T_0,\lam$ and, if both $d,\hat{N}\geq 2$, also on $\delta$ and the constants $\Cdt $ and $\Cdb $ in \eqref{ae28} and \eqref{se24}. In particular, $C$ and $\tau_0$ are independent of $M$.
\end{theorem}
Before proving Theorem \ref{th:IV_error}, we show the following remarkable corollary which is the
main result of the paper.
\begin{corollary}
{Let the assumptions of Theorem \ref{th:IV_error} hold and, for simplicity, assume $N=\hat{N}$.}
Then for any ${q},m\in\N_0$ with $2{q}+m\leq N$, the two limits
\begin{align}\label{eq:limit_one}
 {\partial_T^q} \partial_k^m  \bar{\sigma}_N(t,x_0,y_0;t,x_0)&:= \lim_{(T,k)\to (t,x_0)\atop |x_0-k|\leq \lam \sqrt{T-t}} {\partial_T^q} \partial_k^m  \bar{\sigma}_N(t,x_0,y_0;T,k),\\ \label{eq:limit_two}
 {\partial_T^q} \partial_k^m  \sigma(t,x_0,y_0;t,x_0)&:= \lim_{(T,k)\to (t,x_0)\atop |x_0-k|\leq \lam \sqrt{T-t}} {\partial_T^q} \partial_k^m  \sigma(t,x_0,y_0;T,k),
\end{align}
{exist, are finite and} coincide for any $\lam>0$ and $t\in[0,T_0[$. {Consequently, we have the}
following {parabolic} $N$-th order Taylor expansion:
\begin{equation}\label{eq:Tay_form_intro}
 {\s(t,x_{0},y_{0};T,k)} = \sum_{2 {q}+m\leq N} \frac{{(T-t)^q} {(k-x_{0})}^m}{{q!}m!}\, {\partial_T^q} \partial_k^m  \bar{\sigma}_{N}(t,x_0,y_0;t,x_0) + R_N(t,x_0,y_0,T,k),
\end{equation}
with $$
 R_N(t,x_0,y_0,T,k)=\text{\rm o}\left({|T-t|}^{\frac{N}{2}}+|k-x_0|^N\right),\qquad \text{as }(T,k)\to(t,x_0)\text{ with } |x_0-k|\leq \lam \sqrt{T-t}.
$$
\end{corollary}
\begin{proof}
By Theorem \ref{th:IV_error}, we have
\begin{equation}\label{eq:IV_limit_bis}
 \lim_{(T,k)\to (t,x_0)\atop |x_0-k|\leq \lam \sqrt{T-t}} {\partial_T^q} \partial_k^m  \big(\sigma-\bar{\sigma}_N\big)(t,x_0,y_0;T,k)  = 0,\qquad
 t\in[0,T_0[,\quad \lam>0,
\end{equation}
for any ${q},m\in\N_0$ with $2{q}+m\leq N$. Therefore, the limit in \eqref{eq:limit_one}
{converges} if and only if the limit \eqref{eq:limit_two} {converges} and {in that case they
coincide.} {Now, by the representation formulas in Theorem \ref{th:un_general_repres} and
Proposition \ref{prop.un}, $\bar{\sigma}_N(t,x_0,y_0;\cdot,\cdot)\in C_{P}^{N}\left([0,T_0[\times
\R\right)$ and thus the limit in \eqref{eq:limit_two} converges. }
\end{proof}
\begin{remark}
The derivatives appearing in the Taylor formula \eqref{eq:Tay_form_intro} can be computed
explicitly (possibly with the aid of a symbolic computation software) by means of the
representation formulas of
Theorem \ref{th:un_general_repres} and Proposition \ref{prop.un}. 
\end{remark}
\begin{remark}
A direct computation shows that, at order $N=0$, formula \eqref{eq:Tay_form_intro} is consistent
with the well-known results by \cite{berestycki2002asymptotics} and
\cite{berestycki-busca-florent}. Furthermore, again by direct computation, one can check that in
the special case $d=1$, formula \eqref{eq:Tay_form_intro} with {$q=0$} and $m=1$ is consistent
with the well-known practitioners' \emph{$1/2$ slope rule}, according to which the at-the-money
slope of the implied volatility is one half the slope of the local volatility function.
\end{remark}

The rest of the section is devoted to the proof of Theorem \ref{th:IV_error}. Hereafter $\lam> 0$
is fixed and we assume the hypotheses of Theorem \ref{th:IV_error} to be in force. In particular,
{the center $z_0=(x_0,y_0)$} of the cylinder $H(T_0,z_0,r)$ in Assumptions \ref{assum2and2} and
\ref{exder} is fixed from now on.
\begin{notation}\label{notation:ste}
If not explicitly stated, $C$ and $\tau_0$ will always denote two positive constants dependent at
most on $\lam$, on $r,z_0,d, M_0,\e,N,T_0,\d$ appearing in Assumptions \ref{assum1i},
\ref{assum1ii},  and, only if both $\hat{N},d\geq 2$, also on $\Cdt,\Cdb$ in \eqref{ae28} and
\eqref{se24}.
Note that, in particular, {\it neither $C$ nor $\tau_0$ do depend on $M$}.
\end{notation}

{The proof of Theorem \ref{th:IV_error} is based on some preliminary results.
\begin{lemma}\label{landre10} For any positive constants $c,\bar{\s},\lam,\m$ with $\m<1$, there exists a positive
$\bar{\t}$ {only dependent on $c,\bar{\s},\lam,\m$}, such that
\begin{align}\label{andre6}
  u^{\BS}(\m \s;\t,x,k)+c e^{x}\s^2\t \le u^{\BS}(\s;\t,x,k) ,
\end{align}
for any $\t\in[0,{\bar{\t}}]$, $ \s\leq \bar{\s}$ and $|x-k|\le \lam\s \sqrt{\t}$.
\end{lemma}
\proof We recall the following expression for the Black\&Scholes price (see, for instance,
\cite{roper2009relationship}):
\begin{align}\label{andre1}
  u^{\BS}(\s;\t,x,k)&=\left(e^{x}-e^{k}\right)^{+}+
  e^{x}\sqrt{\frac{\t}{2\pi}}\int_{0}^{\s}e^{-\frac{1}{2}\left(\frac{x-k}{w\sqrt{\t}}+\frac{w\sqrt{\t}}{2}\right)^{2}}dw.
  \end{align}
Then we have
\begin{align}
 u^{\BS}(\s;\t,x,k)-u^{\BS}(\m\s;\t,x,k)
    &=
    e^{x}\sqrt{\frac{\t}{2\pi}}\int_{\m\s}^{\s}e^{-\frac{1}{2}\left(\frac{x-k}{w\sqrt{\t}}+\frac{w\sqrt{\t}}{2}\right)^{2}}dw\ge
\intertext{(by using $|x-k|\le \lam \s \sqrt{\t}$ and $\s\leq \bar{\s}$)}
    &\ge
        e^{x}\sqrt{\frac{\t}{2\pi}}e^{-\frac{1}{2}\left(\frac{\lam}{\m}+\frac{\bar{\s}\sqrt{\t}}{2}\right)^{2}}\s(1-\m)\ge c e^{x}\s^2\t,
\end{align}
for any $\t\in[0,{\bar{\t}}]$ where $\bar{\t}$ is positive and suitably small constant, depending
only on $c,\lam,\bar{\s}$ and $\m$.
\endproof
}

\begin{notation}\label{notat1}
Sometimes, in order to simplify the notation, we will use the shortcuts
\begin{align}
 &{\FF(\s,k,T)}:=u^{\BS}(\s;T-t,\xfixed,k),\qquad &\s>0,\quad k\in\R,\quad {T\ge t,}\\
 &{\GG(u,k,T)}:=\left(u^{\BS}(\cdot;T-t,\xfixed,k)\right)^{-1}(u)\qquad &u\in ](e^{\xfixed}-e^{k})^{+},e^{\xfixed}[,\quad
 k\in\R,\quad {T\ge t,}
\end{align}
for the Black\&Scholes price and its inverse function with respect to the volatility variable. To
ease notations, for any function $F$ of {three} variables $z_{1},z_{2},z_3$, we also set
$\p_{i}F=\frac{\p F}{\p z_{i}}$, $i=1,2{,3}$. Derivatives of compositions of $\FF$ and $\GG$ will
be expressed according this notation: for example, first order derivatives are given by
\begin{align}
 \label{aae8}
 \frac{d}{dk}\FF\left({\GG(u,k,T)},k{,T}\right)&=\left(\p_{1}\FF\right)\left({\GG(u,k,T)},k{,T}\right)\cdot \p_{2}{\GG(u,k,T)}+\left(\p_{2}\FF\right)\left({\GG(u,k,T)},k{,T}\right),\\
 \frac{d}{d T}\FF\left({\GG(u,k,T)},k{,T}\right)&=\left(\p_{1}\FF\right)\left({\GG(u,k,T)},k{,T}\right)\cdot {\p_{3} \GG(u,k,T)}+\left({\p_{3}}\FF\right)\left({\GG(u,k,T)},k{,T}\right) .
\end{align}
For any $\delta\in [0,1]$, we introduce the functions
\begin{align}\label{eq:def_u_delta}
 u(\delta,k{,T})\equiv u(\delta;t,\xfixed,\yfixed,T,k) & :=
 u^{\BS}\left(\s^{(\xfixed,\yfixed)}_{0}(t,T);T-t,\xfixed,k\right)+R(\delta;t,\xfixed,\yfixed,T,k),\\ \label{eq:def_u_delta_bis}
 R(\delta,k{,T})\equiv R(\delta;t,\xfixed,\yfixed,T,k) &:=\sum_{n=1}^N \delta^n u^{(\xfixed,\yfixed)}_{n}(t,\xfixed,\yfixed;T,k) + \delta^{N+1} \left(u-\bar{u}_N\right)(t,\xfixed,\yfixed;T,k),
\end{align}
Recall that $\s^{(\xfixed,\yfixed)}_{0}(t,T)$ and
$u^{(\xfixed,\yfixed)}_{n}(t,\xfixed,\yfixed;T,k)$ are defined for any $0\leq t<T\leq T_0$ and
$k\in\R$, as indicated by \eqref{e10} and \eqref{eq:v.n.pide} respectively. Consequently, by
Theorem \ref{t2} and by Corollary \ref{lem:ste10}, Eq. \eqref{eq:ste36c}, there exist $C$ and
$\tau_0$ as in Notation \ref{notation:ste} such that
\begin{align}\label{aae1}
  \left|R(\delta,k{,T})\right|&\le Ce^{\xfixed}M\left(T-t\right), \qquad
\intertext{and, for any ${q,}m,h\in\N_0$ and $j\in\N$, with ${q+}m+h>0$, $h,j\leq N+1$ and
$m{+2q}\leq \hat{N}$, }\label{ste:equ11}
  \left|{\partial_T^q}\partial_k^m \left(\left(\partial_{\delta}^h u(\delta,k{,T})\right)^j\right)\right|
  &\leq Ce^{\xfixed}{M^{q}} (M(T-t))^{\frac{j(h+1)-m{-2q}}{2}},
\end{align}
for any $0\leq t<T< T_0$ and $k$ such that $T-t\leq \tau_0$ and
$|x_0-k|\leq \lam \sqrt{M(T-t)}$.
\end{notation}

{\begin{lemma}\label{lem:imp_vol_u_delta}
There exists a positive $\t_{0}$ as in Notation \ref{notation:ste}
such that
\begin{align}
u^{\BS}\big( \sqrt{\eps{{M}}};{T-t},\xfixed,k\big) \leq u(\del,k{,T})
     \leq u^{\BS}\big(\sqrt{4M};{T-t},\xfixed,k\big) , \label{eq:bounds}
\end{align}
or equivalently
\begin{align}\label{ae9}
 \sqrt{\eps{M}}\leq \big(u^{\BS}\big)^{-1}(u(\delta,k{,T});{T-t},\xfixed,k)\leq \sqrt{4M},
 \end{align}
for any $\delta\in[0,1]$, $0\leq t<T< T_0$ and $k\in\R$
such that $T-t\leq \tau_0$ and $|\xfixed-k|\leq \lam \sqrt{M(T-t)}$.
\end{lemma}
\begin{proof}
Since $u(\delta,k{,T})-
u^{\BS}\left(\s^{(\xfixed,\yfixed)}_{0}(t,T);T-t,\xfixed,k\right)=R(\delta,k{,T})$, from estimate
\eqref{aae1} we infer
\begin{equation}\label{aae2}
  u^{\BS}\left(\s^{(\xfixed,\yfixed)}_{0}(t,T);T-t,\xfixed,k\right)-Ce^{\xfixed}M\left(T-t\right)\le u(\delta,k{,T}) \le u^{\BS}\left(\s^{(\xfixed,\yfixed)}_{0}(t,T);T-t,\xfixed,k\right)
  +Ce^{\xfixed}M\left(T-t\right),
\end{equation}
with $C$ as in Notation \ref{notation:ste}.
Now recall that, by Assumption \ref{assum2and2} along with definition \eqref{e10}, we have
 $$\sqrt{2\eps M}\le \s^{(\xfixed,\yfixed)}_{0}(t,T)\le\sqrt{2M}\le \sqrt{2M_{0}}$$
and therefore, for any fixed $\l>0$, the thesis follows by combining \eqref{aae2} with estimate \eqref{andre6} with $\m=\frac{1}{2}$. \end{proof}
}

\begin{remark}
In light of Lemma \ref{lem:imp_vol_u_delta}, the function $\GG\left(u(\delta,k{,T}),k{,T}\right)$
is well defined for any {$\delta\in[0,1]$, $0\leq t<T< T_0$ and $k\in\R$ such that $T-t\leq
\tau_0$} and $|\xfixed-k|\leq \lam \sqrt{M(T-t)}$.
\end{remark}

\begin{lemma}\label{l1_bis}
For any ${q,}m,n\in\N_0$, there exist $C,\tau_0>0$ as in Notation \ref{notation:ste} such that
\begin{align}\label{ae12_bis}
 \left| \left(\partial^n_{1}\p_{2}^{m}{\partial^q_{3}} \GG\right)\left(u(\delta,k{,T}),k{,T}\right) \right| \leq C {M^{q+\frac{1}{2}}}\left(M(T-t)\right)^{-\frac{m{+2q}+n}{2}} e^{-n k},
  \end{align}
for any $\delta\in[0,1]$, $0\leq t<T< T_0$ and $k\in\R$ such that $T-t\leq \tau_0$ and
$|\xfixed-k|\leq \lam \sqrt{M(T-t)}$. Here $C$ also depends on $m${, $q$} and $n$.
\end{lemma}
\proof See Appendix \ref{apperrb}.
\endproof

\begin{lemma}\label{l1}
For any ${q,}m,n\in\N_0$ with ${2q+}m\leq \hat{N}$, there exist $C,\tau_0>0$ as in Notation
\ref{notation:ste} such that
\begin{align}\label{ae12}
 \left| 
 {\frac{d^{q+m}}{dT^q\, dk^m} 
 }\left(\partial^n_{1} \GG\right)\left(u(\delta,k{,T}),k{,T}\right) \right| \leq C 
 {M^{q+\frac{1}{2}}}(M(T-t))^{-\frac{m{+2q}+n}{2}}
  e^{-n k}
  \end{align}
for any $\delta\in[0,1]$, $0\leq t<T< T_0$ and $k\in\R$
such that $T-t\leq \tau_0$ and $|\xfixed-k|\leq \lam \sqrt{M(T-t)}$. Here the constant $C$ also depends on $n$.
\end{lemma}
\proof See Appendix \ref{apperrb}.
\endproof

We are now ready to prove Theorem \ref{th:IV_error}.
\begin{proof}[of Theorem \ref{th:IV_error}]
We set
  $$G(\delta,k{,T})=\GG(u(\d,k{,T}),k{,T})$$
with $\GG={\GG(u,k,T)}$ and $u=u(\d,k{,T})$ defined in Notation \ref{notat1} and
\eqref{eq:def_u_delta} respectively. By definition we have
\begin{align}\label{eq:sigma_g1}
\sigma(k{,T})=g(1,k{,T}),
\end{align}
where $\s(k{,T}):=\sigma(t,\xfixed,\yfixed,k{,T})$ is the exact implied volatility. Moreover, {for
$\bar{\sigma}_{N}(k{,T}):=\bar{\sigma}_{N}(t,\xfixed,\yfixed;k,T)$ as defined in
\eqref{eq:sig.approx}}, we have
\begin{align}\label{eq:sigmabar_g0tay}
 \bar{\sigma}_{N}(k{,T})=
 \sum_{n=0}^{N}\sigma^{(\xfixed,\yfixed)}_n(t,\xfixed,\yfixed;k,T)
 =\sum_{n=0}^N \frac{1}{n!}\partial^n_{\delta}g(\delta,k{,T})\big|_{\delta=0},
\end{align}
as, by \eqref{eq:def_u_delta} and \eqref{ae6},
$g(\delta,k{,T})|_{\delta=0}=\sigma^{(\xfixed,\yfixed)}_0(t,T),$ and
$\partial^n_{\delta}g(\delta,k{,T})\big|_{\delta=0}=\sigma^{(\xfixed,\yfixed)}_n(t,\xfixed,\yfixed;k,T)$
for $1\leq n\leq N$. Now, by \eqref{eq:sigma_g1}-\eqref{eq:sigmabar_g0tay}, there exists
$\bar{\delta}\in [0,1]$ such that
\begin{align}
 \sigma(k{,T})-\bar{\sigma}_{N}(k{,T})&=
 \frac{1}{(N+1)!}\partial^{N+1}_{\delta}g(\bar{\delta},k{,T}) \\
&\hspace{-65pt}=\frac{1}{(N+1)!}\sum_{h=1}^{N+1}
 \left(\partial^h_{1} \GG\right)\left(u(\bar{\delta},k{,T}),k{,T}\right)\cdot \mathbf{B}_{N+1,h}\left(
 \partial_{\delta}u(\bar{\delta},k{,T}),\partial^2_{\delta}u(\bar{\delta},k{,T}),
 \dots,\partial_{\delta}^{N-h+2}u(\bar{\delta},k{,T})  \right) ,
\end{align}
where the last equality stems from the Fa\`a di Bruno's formula \eqref{eq:Faa_di_Bruno_appendix}.
Now, differentiating  both the left and the right-hand side $m$ {and $q$} times w.r.t. $k$ {and
$T$ respectively,} we get
\begin{align}
\hspace{-30pt} \left|{\partial_T^q} \partial_k^m \sigma(k{,T})-{\partial_T^q}\partial_k^m
\bar{\sigma}_{N}(k{,T})\right| \leq C\sum_{h=1}^{N+1}& {\sum_{l=0}^{q}} \sum_{j=0}^{m}
 \left| \frac{\dd^{{q-l+}m-j} }{{d T^{q-l}} d k^{m-j} }\left(\partial^h_{1}
 \GG\right)\left(u(\bar{\delta},k{,T}),k{,T}\right) \right| \\
 &\cdot \left| \frac{\dd^{{l+}j}
 }{{d T^{l}} d k^{j} }  \mathbf{B}_{N+1,h}\left(  \partial_{\delta}u(\bar{\delta},k{,T}),
 \dots,\partial_{\delta}^{N-h+2}u(\bar{\delta},k{,T})  \right)  \right|.  \label{eq:diff_sigma_sigma_bar}
\end{align}
Again by Fa\`a di Bruno's formula, we have
\begin{align}
 &\left|\frac{\dd^{{l+}j}
 }{{d T^{l}} d k^{j} }  \mathbf{B}_{N+1,h}\left( \partial_{\delta}u(\bar{\delta},k{,T}), \dots,\partial_{\delta}^{N-h+2}u(\bar{\delta},k{,T})  \right)\right| \\
 &\qquad\leq C  \sum_{\substack{ j_1,\dots,j_{N-h+2} \\ i_1+\cdots + i_{N-h+2}=j \\ { l_1+\cdots + l_{N-h+2}=l}}}
 \left|{\partial_T^{l_1}}\partial_k^{i_1}\left(\partial_{\delta}u(\bar{\delta},k{,T})\right)^{j_1}\right|
 \cdots  \left|{\p_T^{l_{N-h+2}}}\p_k^{i_{N-h+2}}\left(\partial^{N-h+2}_{\delta} u(\bar{\delta},k{,T})\right)^{j_{N-h+2}}\right|
\intertext{(by \eqref{ste:equ11})}
 &\qquad\le C \sum_{\substack{{j_1,\dots,j_{N-h+2}}}} e^{(j_1+\cdots+j_{N-h+2})\xfixed} {M^l} (M(T-t))^{-\frac{{j+2l}}{2}+\frac{j_1+\cdots + j_{N-h+2}}{2}+\frac{j_1+2 j_2+\cdots +(N-h+2) j_{N-h+2}}{2}}
\intertext{(by both the identities in \eqref{eq:relation_indexes_bell})}
 &\qquad= C \sum_{\substack{ {j_1,\dots,j_{N-h+2}} }} e^{h\xfixed}  (M(T-t))^{\frac{-j + h + N + 1}{2}}
 =C e^{h\xfixed} {M^l}  (M(T-t))^{\frac{-j {-2l}+ h + N + 1}{2}}. \label{eq:very_final}
\end{align}
Combining Lemma \ref{l1} and \eqref{eq:very_final} with \eqref{eq:diff_sigma_sigma_bar}, we obtain
$$\big|\partial_k^m \sigma(k{,T})-\partial_k^m \bar{\sigma}_{N}(k{,T})\big| \le C{M^{q+\frac{1}{2}}}\big(M(T-t)\big)^{\frac{N+1-m {-2q}}{2}}\sum_{h=1}^{N+1} { e^{h(\xfixed-k)}} 
  .$$
The statement then follows from the assumption {$|\xfixed-k|\leq \l\sqrt{M(T-t)}{\leq \lam T_0}$}.
\end{proof}

\appendix

%
%

\section{Proof of Theorem \ref{th:error_estimates_taylor}}\label{appth44}
First observe that, for any $z,\bar{z}\in\mathbb{R}^d$, $t<T$ and {$m\le N$, we have}
\begin{equation}\label{eq:ste30}
 \partial^m_k u(t,z;T,k)-\partial^m_k\bar{u}_{N}^{(\bar{z})}(t,z;T,k) = \int_t^T \int\limits_{\mathbb{R}^d} \G(t,z;s,\z) \sum_{n=0}^N
 \left({\Ac_{s} - \bar{\Ac}^{(\bar{z})}_{s,n}}\right) \partial^m_k u^{(\bar{z})}_{N-n}(s,\z;T,k)d \z d s,
\end{equation}
where
\begin{align}
 \bar{\Ac}^{(\bar{z})}_{t,n} = \sum_{i=0}^n  \Ac^{(\bar{z})}_{t,i}.
\end{align}
In fact, when $m=0$ the identity \eqref{eq:ste30} reduces to Lemma 6.23 in \cite{LPP4}. The
general case easily follows by applying the operator $\partial^m_k$ to \eqref{eq:ste30} with $m=0$
and then shifting $\partial^m_k$ onto $u^{(\bar{z})}_{N-n}$. {For clarity, we split the proof in
two separate steps.

\medskip\noindent {\bf [Step 1: case $q=0$ and $0\leq m\leq N$]}\\}
Let $$\mathbb{T}^{a_{\alpha}(s,\cdot)}_{z,n}(\zeta):=\sum\limits_{|\beta|\leq n}
\frac{D^{\beta}a_{\alpha}(s,z)}{\beta!}(\zeta-z)^{\beta}$$ be the $n$-th order Taylor polynomial
of the function $\z\mapsto a_{\alpha}(s,\z)$, centered at $z$. Setting $\bar{z}=z$ and by
definition of $(\Ac_{t,i})_{0\leq i \leq N}$, from \eqref{eq:ste30} we obtain
\begin{align}
 \partial^m_k u(t,z;T,k) - \partial^m_k\bar{u}_{N}(t,z;T,k)=
 \sum_{0\le n\le N\atop |\alpha|\leq 2} I_{n,\a}
\end{align}
where
\begin{align}
 I_{n,\a}&=\int_t^T \int\limits_{\mathbb{R}^d}
  \Gamma(t,z;s,\zeta)\left( a_{\alpha}(s,\zeta) -
  \mathbb{T}^{a_{\alpha}(s,\cdot)}_{z,n}(\zeta)
  \right)D_{\zeta}^{\alpha}
  \partial_k^m u^{(z)}_{N-n}(s,\z;T,k) d \zeta d s
  \intertext{(by Corollary \ref{lem:ste10})}
  & = \sum_{
  |\gamma|\leq N-n \atop 1\leq j \leq 3(N-n)}\int_t^T \int\limits_{\mathbb{R}^d}
  \Gamma(t,z;s,\zeta)\left( a_{\alpha}(s,\zeta) -
  \mathbb{T}^{a_{\alpha}(s,\cdot)}_{z,n}(\zeta)
  \right) (\zeta-z)^{\gamma}\cdot\\
  &\hspace{150pt}\cdot {f^{(N-n,0,m,\alpha)}_{\gamma,j}}(z;s,T) \partial_{\zeta_1}^{j+m+\alpha_1} u^{(z)}_0(s,\zeta;T,k)\, d \zeta d
  s
  \intertext{{(integrating by parts $m$ times
  })}
 & = \sum_{
 |\gamma|\leq N-n \atop 1\leq j \leq 3(N-n)}\int_t^T \int\limits_{\mathbb{R}^d} {(-1)^{m}} R^{\alpha,\gamma,m}_{n,1}R^{\alpha,\gamma,m,j}_{n,2}\, d \zeta d s, \label{eq:ste47}
\end{align}
with
\begin{equation}\label{eq:ste46}
\begin{split}
{R^{\alpha,\gamma,m}_{n,1} }& ={\partial^{m}_{\zeta_1}} \left(  \Gamma(t,z;s,\zeta)\big(
a_{\alpha}(s,\zeta) -
  \mathbb{T}^{a_{\alpha}(s,\cdot)}_{z,n}(\zeta)
  \big) (\zeta-z)^{\gamma}\right),  \\
{R^{\alpha,\gamma,m,j}_{n,2} }& ={f^{(N-n,0,m,\alpha)}_{\gamma,j}}(z;s,T)
{\partial_{\zeta_1}^{j+\alpha_1}} u^{(z)}_0(s,\zeta;T,k).
\end{split}
\end{equation}
Note that $R_{n,1}$ is well defined because $a_{\alpha}(s,\cdot)\in C^{N+1}(\R^d)$,  by
hypothesis, and $m\leq N$.
Now, 
{on the one hand,} by repeatedly applying the Leibniz rule, the mean value theorem and Lemma
\ref{lem:gaussian_estimates} with $c=2$, we obtain
\begin{align}\label{eq:ste49}
 \left|{R^{\alpha,\gamma,m}_{n,1}}\right|&\leq CM(M(s-t))^{ \frac{{n-m+|\gamma|+1}}{2}}\Gbs \left(2 M(s-t), \zeta -
 z\right).
\end{align}
On the other hand, by \eqref{eq:ste36b} and by Lemma \ref{ste:l2}, we have
\begin{align}\label{eq:ste48}
 \left|{R^{\alpha,\gamma,m,j}_{n,2} }\right| & \leq C e^{\zeta_1} (M(T-s))^{\frac{{N-n-|\g|-\alpha_{1}+1}}{2}}
 \leq C e^{\zeta_1} (M(T-s))^{\frac{{N-n-|\g|-1}}{2}} && \text{(since $\alpha_1\leq 2$)}.
\end{align}
To conclude, it is enough to combine estimates \eqref{eq:ste48} and \eqref{eq:ste49} with identity
\eqref{eq:ste47}. In particular, by using $$
 \int_{\R^d} \Gbs\left(2 M(s-t), \zeta - z\right)  e^{\zeta_1} d \zeta = e^{z_1 + M(s-t)/2 },$$
we get
\begin{align}
 |I_{n,\a}| & \leq C e^{z_1} M^{\frac{N-m+2}{2}} \int_{t}^T (s-t)^{{ \frac{{n-m+|\gamma|+1}}{2}}} (T-s)^{\frac{{N-n-|\g|-1}}{2}} d s \leq
 C e^{z_1} (M(T-t))^{\frac{N-m+2}{2}},
\end{align}
where we used the identity $$\int_{t}^{T}(T-s)^{n}(s-t)^{j}\, d s=\frac{\Gamma_E (j+1) \Gamma_E
(n+1)}{\Gamma_E (j+n+2)}(T-t)^{j+n+1},$$ with $\Gamma_E$ representing the Euler Gamma function.

\medskip\noindent {\bf [Step 2: case $0<  m +2 q\leq N$]}
\\
We first prove that,
for any $\bar{m},\bar{q}\in\N_0$ with $\bar{m}+2\bar{q}\leq N-2$, one has
\begin{align}
&\lim_{s\to T^-}\int\limits_{\mathbb{R}^d} \G(t,z;s,\z) \sum_{n=0}^{N}
 \left(\Ac_{s} - \bar{\Ac}^{(z)}_{s,n}\right) \partial^{\bar{q}}_T\partial^{\bar{m}}_k u^{(z)}_{N-n}(s,\z;T,k)d \z \\
 &= \bigg(  \frac{a_{11}(T,{z})}{2}  \bigg)^{\bar{q}}e^k \int\limits_{\mathbb{R}^{d-1}} \big( \partial^2_{k} + \partial_{k} \big)^{\bar{q}} \big(1+\partial_{k}\big)^{\bar{m}}\left( \G(t,z;T,k,\eta) \big( a_{11}(T,k,\eta) -  \mathbb{T}^{a_{11}(T,\cdot)}_{z,N}(k,\eta)   \big)\right)  d \eta.\label{eq:ste_limit_bis} 
\end{align}
Set
  $$ I_{n}(t,z): = \sum_{|\a|\leq 2}\  \int\limits_{\mathbb{R}^d}\Gamma(t,z;s,\zeta)
  \left( a_{\alpha}(s,\zeta) -
  \mathbb{T}^{a_{\alpha}(s,\cdot)}_{z,n}(\zeta)\right)D_{\zeta}^{\alpha}
  \partial_T^{\bar{q}}\partial_k^{\bar{m}} u^{(z)}_{N-n}(s,\z;T,k) d \zeta, \qquad 0\leq n \leq N.
$$ Now, by applying \eqref{eq:ste37} and integrating by parts $\bar{m}+2 \bar{q}+2$ times w.r.t.
$\z_1$ (this is possible because $a_{\alpha}(s,\cdot)\in C^{N+1}(\R^d)$), for $n\leq N-1$ we get
\begin{align}
 I_{n}(t,z)
  &= (-1)^{\bar{m}+2 \bar{q}+2} \sum_{|\a|\leq 2} \sum_{
 |\gamma|\leq N-n \atop 1\leq j \leq 3(N-n)} \int\limits_{\mathbb{R}^d} \partial^{\bar{m}+2 \bar{q}+2}_{\zeta_1}
\left(  \big( a_{\alpha}(s,\zeta) -
  \mathbb{T}^{a_{\alpha}(s,\cdot)}_{z,n}(\zeta)
  \big)\Gamma(t,z;s,\zeta) (\zeta-z)^{\gamma}\right) R^{\alpha,\gamma,\bar{q},\bar{m},j}_{n} \, d \zeta,
\end{align}
with
\begin{equation}\label{eq:ste46_bis}
 R^{\alpha,\gamma,\bar{q},\bar{m},j}_{n}  = f^{(N-n,\bar{q},\bar{m},\alpha)}_{\gamma,j}(z;s,T)
 \partial_{\zeta_1}^{j+\alpha_1-2}
u^{(z)}_0(s,\zeta;T,k),
\end{equation}
and $f^{(N-n,\bar{q},\bar{m},\alpha)}_{\gamma,j}$ as in Corollary \ref{lem:ste10}.
Moreover, by \eqref{eq:ste36b} and by Lemma \ref{ste:l2} we obtain $$
 \big|R^{\alpha,\gamma,\bar{q},\bar{m},j}_{n}\big| \leq C  M^{\bar{q}}e^{\z_{1}} \sqrt{M(T-s)},
$$ and thus
\begin{equation}\label{eq:ste401}
\lim_{s\to T^-} I_n(t,z) = 0,\qquad 0\leq n\leq N-1,\ t<T,\ z\in\R^{d}.
\end{equation}
On the other hand, by \eqref{eq:ste33} and \eqref{eq:ste33_bis}, we have 
\begin{align}
I_N(t,z)&:=
\int\limits_{\mathbb{R}^d} \G(t,z;s,\z)
 \left(\Ac_{s} - \bar{\Ac}^{(z)}_{s,N}\right) \partial^{\bar{q}}_T\partial^{\bar{m}}_k u^{(z)}_{0}(s,\z;T,k)d \z \\
 & =\Big(  \frac{a_{11}(T,{z})}{2}  \Big)^{\bar{q}} 
 \int\limits_{\mathbb{R}^d} \G(t,z;s,\z)
\big( a_{11}(s,\zeta) -
  \mathbb{T}^{a_{11}(s,\cdot)}_{z,N}(\zeta)
  \big)\big( \partial^2_{\z_1} - \partial_{\z_1} \big)^{\bar{q}+1} \big(1-\partial_{\z_1}\big)^{\bar{m}} u^{(z)}_{0}(s,\z;T,k)d \z
  \intertext{(integrating by parts
  )}
&=  \Big(  \frac{a_{11}(T,{z})}{2}  \Big)^{\bar{q}} 
\int\limits_{\mathbb{R}^d} \big( \partial^2_{\z_1} + \partial_{\z_1} \big)^{\bar{q}}
\big(1+\partial_{\z_1}\big)^{\bar{m}}\left( \G(t,z;s,\z) \big( a_{11}(s,\zeta) -
  \mathbb{T}^{a_{11}(s,\cdot)}_{z,N}(\zeta)
  \big)\right)\\
&  \hspace{110pt} \cdot \big( \partial^2_{\z_1} - \partial_{\z_1} \big) u^{(z)}_{0}(s,\z;T,k)d \z.
 \end{align}
From \eqref{e10} and \eqref{eq:ste_400} we have 
 $$
 \big( \partial^2_{\z_1} - \partial_{\z_1} \big) u^{(z)}_{0}(s,\z;T,k) =e^k \Gamma_0\left( \int_s^T a_{11}(r,z)d r, \zeta_1-\frac{\int_s^T a_{11}(r,z)d r}{2} - k
 \right),
 $$
where $\G_{0}$ denotes the Gaussian density in \eqref{eq:normal_density_d} with $d=1$.
Noting that
 $$
 \Gamma_0\left( \int_s^T a_{11}(r,z)d r, \zeta_1-\frac{\int_s^T a_{11}(r,z)d r}{2} - k \right) \longrightarrow \d_k,\qquad \text{as $s\to T^-$},
 $$
we obtain
 \begin{equation}\label{eq:ste402}
 \lim_{s\to T^-} I_N(t,z) =  \Big(  \frac{a_{11}(T,{z})}{2}  \Big)^{\bar{q}}e^k \int\limits_{\mathbb{R}^{d-1}} \big( \partial^2_{k} + \partial_{k} \big)^{\bar{q}} \big(1+\partial_{k}\big)^{\bar{m}}\left( \G(t,z;T,k,\eta) \big( a_{11}(T,k,\eta) -  \mathbb{T}^{a_{11}(T,\cdot)}_{z,N}(k,\eta)   \big)\right)  d \eta.
 \end{equation}
Finally, \eqref{eq:ste401} and \eqref{eq:ste402} yield \eqref{eq:ste_limit_bis}.

We now prove \eqref{eq:error_estimate}. By repeatedly applying the Leibniz rule on
\eqref{eq:ste30} and \eqref{eq:ste_limit_bis}, we get
\begin{equation}
{ \partial^q_T} \partial^m_k \big(u- \bar{u}_{N}\big)(t,x,y;T,k) = \int_t^T
\int\limits_{\mathbb{R}^d} \G(t,z;s,\z) \sum_{n=0}^N \left({ \Ac_{s} -
\bar{\Ac}^{(\bar{z})}_{s,n}}\right) \partial^q_T \partial^m_k u^{(\bar{z})}_{N-n}(s,\z;T,k)d \z d
s +\sum_{i=0}^{q-1} J_i,
\end{equation}
with
\begin{equation}
J_i =\partial_T^{q-1-i}\Bigg( \bigg(  \frac{a_{11}(T,{z})}{2}  \bigg)^{i}e^k
\int\limits_{\mathbb{R}^{d-1}} \big( \partial^2_{k} + \partial_{k} \big)^{i}
\big(1+\partial_{k}\big)^{m}\left( \G(t,z;T,k,\eta) \big( a_{11}(T,k,\eta) -
\mathbb{T}^{a_{11}(T,\cdot)}_{z,N}(k,\eta)   \big)\right)  d \eta \Bigg).
\end{equation}
Now, by proceeding as in Step 1, it is easy to show that
\begin{equation}
\left| \int_t^T \int\limits_{\mathbb{R}^d} \G(t,z;s,\z) \sum_{n=0}^N \left({ \Ac_{s} -
\bar{\Ac}^{(\bar{z})}_{s,n}}\right) \partial^q_T \partial^m_k u^{(\bar{z})}_{N-n}(s,\z;T,k)d \z d
s \right| \leq C e^x { M^{q}} \left(\ce({T-t})\right)^{\frac{{N-m{ - 2q}+2}}{2}}.
\end{equation}
Analogously, by repeatedly applying Leibniz rule along with Faa di Bruno's Formula (Proposition
\ref{prop:multivariate_faa}) and Lemma \ref{lem:gaussian_estimates}, and by using that
 $$ e^k
 \int_{\R^{d-1}} \Gbs\big(2 M(T-t),  x - k, y - \eta \big)  d \eta = \frac{e^{k}}{\sqrt{4\pi
 M(T-t)}}e^{-\frac{(k-x)^{2}}{4 M(T-t)}} 
 \leq \frac{C e^{x}}{\sqrt{M(T-t)}}, $$
with $\G_{0}$ as in \eqref{eq:normal_density_d}, one can also show
\begin{equation}
\left| J_i \right| \leq C e^x { M^{q}} \left(\ce({T-t})\right)^{\frac{{N-m{ - 2q}+2}}{2}},\qquad
0\leq i\leq q-1,
\end{equation}
which concludes the proof.

\section{Proof of Lemmas \ref{l1_bis} and \ref{l1}}\label{apperrb}
\begin{proof}[of Lemma \ref{l1_bis}]
The case $n=m=0$ has been already proved in \eqref{ae9}. To prove the general case, we proceed by
induction on $m$ and $n$.

\medskip\noindent
{\bf [Step 1: case $m{=q}=0$]}. \\ By \eqref{ae4} and by using $|\xfixed-k|\leq \lam
\sqrt{M(T-t)}$, we have
\begin{align}\label{ae10_bis}
 \partial_\s {\FF(\s,k,T)}
  &\geq \frac{e^k \sqrt{T-t} }{\sqrt{2 \pi }} \exp\left({- \frac{\lam^2 M}{2\s^2}-\frac{\sigma ^2 (T-t)}{8}  - \frac{ \lam \sqrt{M(T-t)}}{2}}\right)\\ &\geq \frac{e^k \sqrt{T-t} }{\sqrt{2 \pi }}
 \exp\left({- \frac{\lam^2 M}{2\s^2}-\frac{\sigma ^2 T_0}{8}  - \frac{ \lam \sqrt{M_{0}T_0}}{2}}\right),
\end{align}
which, by \eqref{ae9}, implies
\begin{equation}\label{aae3}
 \left(\p_{1}\FF\right)\left(\GG(u(\delta,k{,T}),k{,T})\right)\geq \frac{e^k \sqrt{T-t} }{\sqrt{2 \pi }} \exp\left({ -
\frac{\lam^2 }{2\eps}-\frac{M_{0} T_0}{2} - \frac{ \lam \sqrt{M_{0} T_0}}{2}}\right) .
\end{equation}
Therefore, we obtain
\begin{align}\label{ae10}
 0<\left(\p_{1}\GG\right)\left(u(\delta,k{,T}),k{,T}\right)=\frac{1}{\left(\partial_1\FF\right)\big(\GG(u(\delta,k{,T}),k{,T})\big)}\leq \frac{C}{e^{k} \sqrt{T-t}},
\end{align}
which is \eqref{ae12_bis} for $m=0$ and $n=1$.

We now fix $\bar{n}\in \N$, assume \eqref{ae12_bis} to hold true for any $n\in\N_{0}$ with $n\leq
\bar{n}$ and prove it true for $\bar{n}+1$. Differentiating the identity
$u=\FF({\GG(u,k,T)},k{,T})$ and applying the univariate version of Fa\`a di Bruno's formula (see
Appendix \ref{append:faa_bell}, Eq. \eqref{eq:Faa_di_Bruno_appendix}), we obtain
\begin{align}
 \partial^{\bar{n}+1}_1{\GG(u,k,T)}=- \sum\limits_{h=2}^{\bar{n}+1} \frac{(\partial^h_{1}\FF)\left({\GG(u,k,T)},k{,T}\right)}
 {(\partial_{1}\FF)\left({\GG(u,k,T)},k{,T}\right)} \mathbf{B}_{\bar{n}+1,h}\left(
 \partial_{1} {\GG(u,k,T)},\dots,\partial^{\bar{n}-h+2}_{1} {\GG(u,k,T)}\right).
\end{align}
Now, by \eqref{aae3}, Lemma \ref{ste:l3} and recalling the estimate of Lemma
\ref{lem:imp_vol_u_delta} for $u=u(\d,k{,T})$, we get
  $$\left|\frac{\left(\partial^h_{1}\FF\right)\left(\GG(u(\d,k{,T}),k{,T}),k{,T}\right)}{\left(\partial_{1}\FF\right)\left(\GG(u(\d,k{,T}),k{,T}),k{,T}\right)}\right|\le C M^{-\frac{h-1}{2}}.$$
Moreover, for any $h=2,\dots,\bar{n}+1$, we have
\begin{align}
 &\left|\mathbf{B}_{\bar{n}+1,h}\left(\partial_{1}
 {\GG(u,k,T)},\dots,\partial^{\bar{n}-h+2}_{1}{\GG(u,k,T)}\right)|_{u=u(\d,k)}\right|\le
\intertext{(by \eqref{eq:Bell_polyn_appendix} in Appendix \ref{append:faa_bell})}
 &\leq C\sum_{j_1,\dots,j_{\bar{n}-h+2}}\left| \left(\partial_{1} \GG\right)(u(\d,k{,T}),k{,T})\right|^{j_1} \cdots
  \left|\left(\partial^{\bar{n}-h+2}_{1} \GG\right)(u(\d,k{,T}),k{,T})\right|^{j_{\bar{n}-h+2}}\le
\intertext{(by inductive hypothesis)}
 &\leq C  \sum_{j_1,\dots,j_{\bar{n}-h+2}} \sqrt{M}\left(e^k\sqrt{M(T-t)}\right)^{-j_1} \cdots \sqrt{M}\left(e^k\sqrt{M(T-t)}\right)^{-(\bar{n}-h+2)
 j_{\bar{n}-h+2}}\\
 &\leq
 C M^{\frac{h}{2}}\left(e^k\sqrt{M(T-t)}\right)^{-\bar{n}-1},
\end{align}
where the last inequality follows from the identities \eqref{eq:relation_indexes_bell} in Appendix
\ref{append:faa_bell}. This concludes the proof of \eqref{ae12_bis} with $m=0$.

\medskip\noindent {\bf [Step 2: case {$q=0$}
] }\\ {We proceed by induction on $m$. The sub-case $m=0$ has already been proved in Step 1.} Now
fix $\bar{m}\in\N$, assume \eqref{ae12_bis} to hold for any $n,m\in\N_{0}$, $m\leq \bar{m}$ and
prove it true for $m=\bar{m}+1$ and $n\in\N_{0}$. First note that differentiating w.r.t. $k$ the
identity
 \begin{equation}\label{eq:ste407}
 \s=\GG({\FF(\s,k,T)},k{,T}),\qquad \s>0,
 \end{equation}
we get
\begin{equation}\label{ea1}
  \left(\p_{2}\GG\right)\left({\FF(\s,k,T)},k{,T}\right)=-\left(\p_{1}\GG\right)\left({\FF(\s,k,T)},k{,T}\right)\cdot \p_{2}{\FF(\s,k,T)},
\end{equation}
or equivalently, setting $u={\FF(\s,k,T)}$ that is $\s={\GG(u,k,T)}$,
\begin{equation}\label{ea2}
  \p_{2}\GG\left(u,k{,T}\right)=-\p_{1}\GG\left(u,k{,T}\right)\cdot
  \left(\p_{2}\FF\right)({\GG(u,k,T)},k{,T}),\qquad u\in ](e^{\xfixed}-e^{k})^+,e^\xfixed[.
\end{equation}
Fix $n\in\N_0$: differentiating \eqref{ea2}, $n$ times w.r.t. $u$ and $\bar{m}$ times w.r.t. $k$,
we get
\begin{align}
 \partial_1^n \partial_2^{\bar{m}+1} {\GG(u,k,T)} &=- \frac{d^{n+\bar{m}}}{d u^{n}d k^{\bar{m}}}
 \left(\p_{1}\GG\left(u,k{,T}\right)\cdot \left(\p_{2}\FF\right)({\GG(u,k,T)},k{,T})\right)\\
 & = - \sum_{i=0}^{n}\sum_{j=0}^{\bar{m}} \binom{n}{i}\binom{\bar{m}}{j} \left(\partial_{1}^{n+1-i}\partial_{2}^{\bar{m}-j}  {\GG(u,k,T)}\right)
 \cdot \frac{d^{i+j}}{d u^{i}d k^{j}}\left(\p_{2}\FF\right)({\GG(u,k,T)},k{,T}).\hspace{-10pt}\\ \label{equ:ste7}
\end{align}
Now, by inductive hypothesis, for any $i,j,n\in\N_{0}$ with $i\le n$ and $j\le \bar{m}$, we have
\begin{equation}\label{equ:ste6}
 \left|\left(\partial_{1}^{n+1-i}\partial_{2}^{\bar{m}-j}\GG\right)(u(\delta,k{,T}),k{,T})\right| \leq C \sqrt{M}\left( M(T-t)\right)^{-\frac{n+1-i+\bar{m}-j}{2}}
  e^{-(n+1-i) k}.
  \end{equation}
The proof will be concluded once we show that
\begin{equation}\label{aae5}
 \left|\frac{d^{i+j}}{d u^{i}d k^{j}}\left(\p_{2}\FF\right)({\GG(u,k,T)},k{,T})\big|_{u=u(\delta,k{,T})}\right| \leq  C \left(M(T-t)\right)^{-\frac{i+j}{2}}e^{-(i-1)
 k}.
\end{equation}
Indeed \eqref{aae5}, combined with \eqref{equ:ste6} and \eqref{equ:ste7}, yields \eqref{ae12_bis}
for $\bar{m}+1$.

\medskip\noindent More generally, we prove that for any $i,j,\g_{1},\g_{2}{, \g_3}\in\N_{0}$ with $\g_{1}+\g_{2}{+\g_3}>0$ and $j\le \bar{m}$ (here
$\bar{m}$ is fixed in the inductive hypothesis at the beginning of Step 2), we have
\begin{equation}\label{aae6}
 \left|\frac{d^{i+j}}{d u^{i}d k^{j}}\left(\p_{1}^{\g_{1}}\p_{2}^{\g_{2}}{\p_3^{\g_{3}} }\FF\right)({\GG(u,k,T)},k{,T})\big|_{u=u(\delta,k{,T})}\right|
 \leq  C M^{{\g_3}-\frac{\g_{1}}{2}}\left(M(T-t)\right)^{\frac{1-i-j-\g_{2}{-2 \g_3}}{2}}e^{(1-i) k},
\end{equation}
We prove \eqref{aae6} by using another inductive argument on $j$.\\ {\bf [Step 2-a): case $j=0$]}.
\\
By the univariate version of the Fa\`a di Bruno's formula (see Appendix \ref{append:faa_bell}, Eq.
\eqref{eq:Faa_di_Bruno_appendix}), for any $i,\g_{1},\g_{2}\in\N_{0}$ we have
\begin{equation}\label{equ:ste2}
\begin{split}
 \frac{d^{i}}{d u^{i}} \left(\p_{1}^{\g_{1}}\p_{2}^{\g_{2}}{\p_3^{\g_{3}} } \FF\right) (\GG\left(u,k{,T}\right),k{,T}) =& \sum_{h=1}^{i}
 \left(\partial_1^{h+\g_{1}}\partial_{2}^{\g_{2}} {\p_3^{\g_{3}} } \FF\right) (\GG\left(u,k{,T}\right),k{,T})\\
 &\cdot\mathbf{B}_{i,h}\left( \p_1 {\GG(u,k,T)},\p^2_1 {\GG(u,k,T)},\dots, \p^{i-h+1}_1 {\GG(u,k,T)} \right).
\end{split}
\end{equation}
By Lemmas 
\ref{ste:l3} and \ref{lem:imp_vol_u_delta}, using that $\g_{1}+\g_{2}{+\g_3}>0$, we have
\begin{equation}\label{equ:ste1}
 \left|\left(\partial_1^{h+\g_{1}}\partial_{2}^{\g_{2}} {\p_3^{\g_{3}} } \FF\right) (\GG\left(u,k{,T}\right),k{,T})|_{u=u(\delta,k{,T})}\right|
 \leq C e^{k} { M^{{\g_3}-\frac{h+\g_{1}}{2}}}(M(T-t))^{\frac{1-\g_{2}{-2\g_3}}{2}}.
\end{equation}
Moreover, by \eqref{ae12_bis} with $m=0$ (already proved in Step 1) and by the relations
\eqref{eq:relation_indexes_bell} we have
\begin{equation}
 \left|\mathbf{B}_{i,h}\left( \p_1 {\GG(u,k,T)},\p^2_1 {\GG(u,k,T)},\dots, \p^{i-h+1}_1 {\GG(u,k,T)} \right)|_{u=u(\delta,k{,T})}\right|
 \leq C {M}^{\frac{h}{2}} (M(T-t))^{-\frac{i}{2}} e^{-i k},
\end{equation}
which, combined with \eqref{equ:ste1} and \eqref{equ:ste2}, proves \eqref{aae6} for $j=0$ and any
$i,\g_{1},\g_{2}\in\N_{0}$ with $\g_{1}+\g_{2}{+\g_3}>0$.

\medskip\noindent {\bf [Step 2-b): case $1\le j\le \bar{m}$]}\\
Fix $j_{0}\in\N$ with $j_{0}\leq \bar{m}-1$: we assume \eqref{aae6} to hold for any
$i,\g_{1},\g_{2}{,\g_3}\in\N_{0}$ with $\g_{1}+\g_{2}{+\g_3}>0$ and $0\le j\leq j_{0}$ and prove
it true for $i,\g_{1},\g_{2}{,\g_3}\in\N_{0}$ with $\g_{1}+\g_{2}{+\g_3}>0$ and $j=j_{0}+1$. We
have
\begin{equation}\label{equ:ste5}
\begin{split}
 \frac{d^{i+j_{0}+1}}{d u^{i}d k^{j_{0}+1}} &\left(\p_{1}^{\g_{1}}\p_{2}^{\g_{2}}{\p_{3}^{\g_{3}}} \FF\right) (\GG\left(u,k{,T}\right),k{,T})\\
 &= \frac{d^{i+j_{0}}}{d u^{i}d k^{j_{0}}} \bigg(\left(\p_{1}^{1+\g_{1}}\p_{2}^{\g_{2}}{\p_{3}^{\g_{3}}} \FF\right) \left(\GG\left(u,k{,T}\right),k{,T}\right)\cdot \p_2 {\GG(u,k,T)}\\
&\hspace{60pt} +\left(\p_{1}^{\g_{1}}\p_{2}^{1+\g_{2}}{\p_{3}^{\g_{3}}}
\FF\right)(\GG\left(u,k{,T}\right),k{,T})\bigg)\\
 &=\sum_{h=0}^{i} \sum_{q=0}^{j_{0}}\binom{i}{h}\binom{j_{0}}{q}
 \left(\frac{d^{h+q}}{d u^{h}d k^{q}}\left(\p_{1}^{1+\g_{1}}\p_{2}^{\g_{2}}{\p_{3}^{\g_{3}}} \FF\right) (\GG\left(u,k{,T}\right),k{,T})\right)
 \cdot \p_1^{i-h} \p^{j_{0}-q+1}_2 {\GG(u,k,T)}\\
 &\quad +\frac{d^{i+j_{0}}}{d u^{i}d k^{j_{0}}} \left(\p_{1}^{\g_{1}}\p_{2}^{1+\g_{2}}{\p_{3}^{\g_{3}}} \FF\right)(\GG\left(u,k{,T}\right),k{,T}).
\end{split}
\end{equation}
By inductive hypothesis we have
\begin{equation}
 \left|\frac{d^{h+q}}{d u^{h}d k^{q}}\left(\p_{1}^{1+\g_{1}}\p_{2}^{\g_{2}} {\p_{3}^{\g_{3}}} \FF\right) (\GG\left(u,k{,T}\right),k{,T})\big|_{u=u(\delta,k{,T})}\right|
 \leq C {M}^{{\g_3}-\frac{\gamma_1+1}{2}}(M(T-t))^{-\frac{h+q+\gamma_2{+2\g_3}-1}{2}}e^{-{(h-1)} k},
 \end{equation}
 and
  \begin{equation}
 \left|\frac{d^{i+j_{0}}}{d u^{i}d k^{j_{0}}} \left(\p_{1}^{\g_{1}}\p_{2}^{1+\g_{2}} {\p_{3}^{\g_{3}}} \FF\right)(\GG\left(u,k{,T}\right),k{,T})\big|_{u=u(\delta,k{,T})}\right|
 \leq C {M}^{{\g_3}-\frac{\gamma_1}{2}}(M(T-t))^{-\frac{i+j_{0}+\gamma_2{+2\g_3}}{2}}e^{-(i-1) k}.
\end{equation}
Now we recall that we are assuming, by inductive hypothesis, that \eqref{ae12_bis} holds for any
$n\in\N_{0}$ and $m\le \bar{m}$: thus, since $j_{0}-q+1\le \bar{m}$ by assumption, we get
 $$\left|\p_1^{i-h} \p^{j_{0}-q+1}_2  {\GG(u,k,T)}|_{u=u(\delta,k{,T})}\right| \leq C {M}^{\frac{1}{2}}(M(T-t))^{-\frac{i-h+j_{0}-q+1}{2}}e^{-(i-h) k}.$$
The last three estimates combined with \eqref{equ:ste5} yield \eqref{aae6} for $j=j_{0}+1$.

{\medskip\noindent {\bf [Step 3: case $q\in\N$
] }\\ It is analogous to Step 2. For simplicity, we only prove the case $q=1$. By identity
\eqref{eq:ste407} we get
\begin{equation}\label{ea1_bis}
  \left(\p_{3}\GG\right)\left({\FF(\s,k,T)},k{,T}\right)=-\left(\p_{1}\GG\right)\left({\FF(\s,k,T)},k{,T}\right)\cdot \p_{3}{\FF(\s,k,T)},
\end{equation}
or equivalently, setting $u={\FF(\s,k,T)}$ that is $\s={\GG(u,k,T)}$,
\begin{equation}\label{ea2_bis}
  \p_{3}\GG\left(u,k{,T}\right)=-\p_{1}\GG\left(u,k{,T}\right)\cdot
  \left(\p_{3}\FF\right)({\GG(u,k,T)},k{,T}),\qquad u\in ](e^{\xfixed}-e^{k})^+,e^\xfixed[.
\end{equation}
Fix $n,m\in\N_0$: differentiating \eqref{ea2_bis}, $n$ and $m$ times w.r.t. $u$ and $k$
respectively, and once w.r.t. $T$, we get
\begin{align}
 \partial_1^n \partial_2^{m}\partial_3 {\GG(u,k,T)} &=- \frac{d^{n+{m}}}{d u^{n}d k^{{m}}}
 \left(\p_{1}\GG\left(u,k{,T}\right)\cdot \left(\p_{3}\FF\right)({\GG(u,k,T)},k{,T})\right)\\
 & = - \sum_{i=0}^{n}\sum_{j=0}^{{m}} \binom{n}{i}\binom{{m}}{j} \left(\partial_{1}^{n+1-i}\partial_{2}^{{m}-j}  {\GG(u,k,T)}\right)
 \cdot \frac{d^{i+j}}{d u^{i}d k^{j}}\left(\p_{3}\FF\right)({\GG(u,k,T)},k{,T}).\hspace{-10pt}\\ \label{equ:ste7_bis}
\end{align}
Now, by \eqref{ae12_bis} with $q=0$, for any $i,j,n\in\N_{0}$ with $i\le n$ and $j\le {m}$, we
have
\begin{equation}\label{equ:ste6b}
 \left|\left(\partial_{1}^{n+1-i}\partial_{2}^{{m}-j}\GG\right)(u(\delta,k{,T}),k{,T})\right| \leq C {M}^{\frac{1}{2}}\left( M(T-t)\right)^{-\frac{n+1-i+{m}-j}{2}}
  e^{-(n+1-i) k},
  \end{equation}
whereas, by \eqref{aae6}, we obtain
\begin{equation}\label{aae5b}
 \left|\frac{d^{i+j}}{d u^{i}d k^{j}}\left(\p_{3}\FF\right)({\GG(u,k,T)},k{,T})\big|_{u=u(\delta,k{,T})}\right| \leq  C M \big(M(T-t)\big)^{-\frac{i+j+1}{2}}e^{-(i-1)
 k}.
\end{equation}
Eventually, \eqref{equ:ste6b} and \eqref{aae5b} combined with \eqref{equ:ste7_bis} prove
\eqref{ae12_bis} for $q=1$. }
\end{proof}
\begin{remark}
The inductive argument of the previous proof shows that estimate \eqref{aae6} is valid for any
$i,j,\g_{1},\g_{2}{,\g_3}\in\N_{0}$, with $\g_{1}+\g_{2}{+\g_3}>0$ and {$\delta\in[0,1]$, $0\leq
t<T< T_0$ and $k\in\R$ such that $T-t\le \t_{0}$} and $|\xfixed-k|\leq \lam \sqrt{M(T-t)}$. In
this case, the constant $C$ in \eqref{aae6} also depends on $i,j,\g_{1},\g_2$ and ${\g_3}$.
\end{remark}

\begin{proof}[of Lemma \eqref{l1}]{For simplicity, we split the proof in two separate steps.

\medskip\noindent {\bf [Step 1: case $q=0$
] }\\ } By the bivariate version of Fa\`a di Bruno's formula (see Appendix \ref{append:faa_bell},
Proposition \ref{prop:multivariate_faa}), we obtain {\begin{align}
 &\frac{d^m }{d k^m}\left(\partial^n_{1} \GG\right)\left(u(\delta,k{,T}),k{,T}\right) \\
 &=\sum_{h=1}^{m} \left(\nabla^{h} \partial_1^n \GG\right) \left(u(\delta,k{,T}),k{,T}\right)\ast \mathbf{B}_{m,h}
 \left(\binom{\p_{k} u(\delta,k{,T}) }{1},\binom{\p_{k}^{2} u(\delta,k{,T})}{0},
 \dots,\binom{\p_{k}^{m-h+1} u(\delta,k{,T})}{0}\right)=
\intertext{(by exploiting the first relation in
\eqref{eq:relation_indexes_bell})}\label{eq:estim_ste_52}
 & = \sum_{h=1}^{m}\sum_{j_{1}=0}^{h}g_{h,j_{1}}(\d,k{,T})
 \left(\nabla^{j_1}\partial_1^{n+h-j_1} \GG\right)\left(u(\delta,k{,T}),k{,T}\right)\ast \binom{\p_{k} u(\delta,k{,T})}{1}^{j_1}
\end{align}
where ``$\ast$'' denotes the tensorial scalar product (see \eqref{eq:tens_scal_prod}) and
\begin{equation}\label{aae7}
 g_{h,j_{1}}(\d,k{,T})=\sum_{j_2,\dots,j_{m-h+1}} c^{m,h}_{j_1,\dots,j_{m-h+1}}\prod_{i=2}^{m-h+1}\left(\p_{k}^{i} u(\delta,k{,T})\right)^{j_{i}}
\end{equation}
for some constants $c^{m,h}_{j_1,\dots,j_{m-h+1}}$ and the sum in \eqref{aae7} is taken over all
sequences $j_2,\dots, j_{m-h+1}$ of non-negative integers verifying the identities in
\eqref{eq:relation_indexes_bell}.} Now, by estimate \eqref{ste:equ11} and by the relations
\eqref{eq:relation_indexes_bell}, we obtain
\begin{equation}\label{eq:estim_ste_51}
 \left|g_{h,j_{1}}(\d,k{,T})\right|  \leq C e^{(h-j_1)\xfixed} (M(T-t))^{-\frac{m-h}{2}}.
\end{equation}
Moreover we have
\begin{align}
 &\left|\left(\nabla^{j_1}\partial_1^{n+h-j_1} \GG\right)\left(u(\delta,k{,T}),k{,T}\right)\ast \binom{\p_{k}
 u(\delta,k{,T})}{1}^{j_1}\right|\\
 &\ \le
 C\sum_{q=0}^{j_1} \left|\left(\partial_1^{n+h-q}\partial_2^{q}\GG\right)
 \left(u(\delta,k{,T}),k{,T}\right)\right|\left|\left(\p_{k} u(\delta,k{,T})\right)^{j_1-q}\right|
\end{align}
and therefore, by Lemma \ref{l1_bis} and estimate \eqref{ste:equ11}, we get
\begin{equation}\label{eq:estim_ste_50}
 \left|\left(\nabla^{j_1}\partial_1^{n+h-j_1} \GG\right)\left(u(\delta,k{,T}),k{,T}\right)\ast \binom{\p_{k} u(\delta,k{,T})}{1}^{j_1} \right|
 \leq C e^{-(n+h-q) k + (j_1-q) \xfixed }\sqrt{M}(M(T-t))^{-\frac{n+h}{2}} .
\end{equation}
Eventually, \eqref{ae12} follows by combining \eqref{eq:estim_ste_51}-\eqref{eq:estim_ste_50} with
\eqref{eq:estim_ste_52} and by observing that
\begin{equation}
e^{(h-q)(\xfixed-k)}\leq e^{m|\xfixed-k|}\leq e^{m\lambda \sqrt{M(T-t)}},
\end{equation}
since {$|\xfixed-k|\leq \l\sqrt{M(T-t)}$}.

{\medskip\noindent {\bf [Step 2: case $q\in \N$
] }\\ It is analogous to Step 1. For simplicity, we only prove the case $q=1$. Leibniz rule yields
\begin{align}
\frac{d^m }{d k^m}&\frac{d }{d T}\left(\partial^n_{1}
\GG\right)\left(u(\delta,k{,T}),k{,T}\right)\\ &=  \frac{d^m }{d k^m}\bigg( \big(\p_T
u(\delta,k{,T})\big)  \big(\partial^{n+1}_{1} \GG\big)\left(u(\delta,k{,T}),k{,T}\right)
+\left(\partial^n_{1} \partial_3 \GG\right)\left(u(\delta,k{,T}),k{,T}\right) \bigg)\\ & =
\sum_{i=0}^{m} \binom{m}{i} \big(\p^{m-i}_k \p_T u(\delta,k{,T})\big) \frac{d^i }{d k^i}
\big(\partial^{n+1}_{1} \GG\big)\left(u(\delta,k{,T}),k{,T}\right)+\frac{d^m }{d
k^m}\left(\partial^n_{1}
\partial_3 \GG\right)\left(u(\delta,k{,T}),k{,T}\right).
\\ \label{eq:ste410}
\end{align}
By \eqref{ae12} with $q=0$, by \eqref{ste:equ11}, and by using that $|x_0-k|\leq \lam (T-t)$, we
get
\begin{equation}\label{eq:ste409}
\left|   \big(\p^{m-i}_k \p_T  u(\delta,k{,T})\big) \frac{d^i }{d k^i} \big(\partial^{n+1}_{1}
\GG\big)\left(u(\delta,k{,T}),k{,T}\right)   \right| \leq C  {
M^{1+\frac{1}{2}}}(M(T-t))^{-\frac{m{ +2}+n}{2}}
  e^{-n k} .
\end{equation}
On the other hand, by proceeding exactly as in Step 1, one can show
\begin{equation}
\left|  \frac{d^m }{d k^m}\left(\partial^n_{1} \partial_3
\GG\right)\left(u(\delta,k{,T}),k{,T}\right) \right| \leq C {
M^{1+\frac{1}{2}}}(M(T-t))^{-\frac{m{ +2}+n}{2}}
  e^{-n k} ,
\end{equation}
which, combined with \eqref{eq:ste409} and \eqref{eq:ste410}, proves \eqref{ae12} for $q=1$. }
\end{proof}

\section{Short-time/small-noise estimates in the Black\&Scholes model}\label{app:BS_greeks}
We collect here the short-time estimates for the sensitivities with respect to $\sigma$, $x$ and
$k$ of the Black\&Scholes function $u^{\BS}(\sigma)=u^{\BS}(\sigma;\t,x,k)$, needed to prove the
results of Section \ref{sec:error.impvol}. In this appendix $\G_{0}$ denotes the Gaussian density
in \eqref{eq:normal_density_d} with $d=1$. 
\begin{lemma}\label{lemm:homogeneous_gaussian_estimate}
For any $n\in \N_0$ and $c>1$ we have
\begin{equation}\label{eq:homogeneous_gaussian_estimate}
 \left( \frac{|x|}{\sqrt{t}} \right)^n \Gbs(t,x)\leq  \sqrt{c}\left(\frac{c n}{(c-1)\sqrt{e}}\right)^{\frac{n}{2}} \Gbs (c t, x),\qquad t\in \R_{>0},\ x\in\R.
\end{equation}
\end{lemma}
\begin{proof}
Set $z=\frac{|x|}{\sqrt{t}}$. For any $c>1$ we have
\begin{equation}
 \left( \frac{|x|}{\sqrt{t}} \right)^n \Gbs(t,x) = \frac{z^n}{\sqrt{2\pi t}} \exp\left( -\frac{z^{2}}{2} \right) =  \sqrt{c}\,g(z)\Gbs (c t, x) ,
\end{equation}
with
\begin{equation}
 g(z) = z^n \exp\left( -\frac{z^2}{2} \left(1-\frac{1}{c} \right)  \right),\qquad z\ge 0.
\end{equation}
The statement now follows by observing that $g$ attains a global maximum at $z_{n}=\sqrt{\frac{c
n}{c-1}}$ and that
 $$ g(z_{n}) = e^{-\frac{n}{2}} \left(\frac{c n}{c-1}\right)^{n/2}. $$
\end{proof}
{
\begin{lemma}\label{lemm:homogeneous_gaussian_estimate_deriv}
For any $n\in \N_0$ and $c>1$ we have
\begin{equation}\label{eq:homogeneous_gaussian_estimate_deriv}
 |\p_x^n \Gbs(t,x)|\leq  C\,t^{-\frac{n}{2}}\Gbs (c t, x),\qquad t\in \R_{>0},\ x\in\R,
\end{equation}
where $C$ is a positive constant only dependent on $n$ and $c$.
\end{lemma}
\begin{proof}
Then, by definition \eqref{eq:normal_density_d} we have
\begin{equation}
\p_x^n \Gbs(t,x) = t^{-\frac{n}{2}} \mathbf{H}_n\left(\frac{x}{\sqrt{2t}}\right) \Gbs(t,x),
\end{equation}
and thus the statement easily stems from Lemma \ref{lemm:homogeneous_gaussian_estimate}.
\end{proof}
} In what follows we will make use of the representation of the Black\&Scholes price in term of
the Gaussian density $\Gbs$ in \eqref{eq:normal_density_d}, i.e.
\begin{equation}\label{eq:BS_convolution}
 u^{\BS}(\sigma)=u^{\BS}(\sigma;\t,x,k) = \int_{k}^{+\infty} \Gbs\left(\s^2 \t, x-\frac{\s^2\t}{2}-y\right)\left(e^{y} - e^{k}\right) d y,
\end{equation}
and of the family of Hermite polynomials defined as
\begin{equation}\label{eq:def_hermite}
\mathbf{H}_n(x):=e^{x^2}\partial_x^n e^{-x^2},\qquad n\in\N_0.
\end{equation}

\begin{lemma}\label{ste:l2}
For any $m,n\in \mathbb{N}_0$ and $M>0$
we have
\begin{align}\label{eq:ste01}
 \left|\partial^n_{x}\p_{k}^{m} \ub(\s;\t,x,k)\right|& \le C{e^{x}}
 \left(\s\sqrt{\t}\right)^{{(1-m-n)\wedge 0}}
 , \qquad x,k \in \R, \quad 0<\s\sqrt{\t} \leq M,
\end{align}
where $a\wedge b=\min\{a,b\}$ and $C$ is a positive constant only dependent on $m,n$ and $M$.
\end{lemma}
\begin{proof}
Throughout this proof we will denote by $C$ any generic constant that depends at most on $m,n$ and
$M$. We first prove the statement for $m=0$. If also $n=0$ then the thesis easily follows by
writing $\ub$ as an expectation. If $n\geq 1$ then by \eqref{eq:BS_convolution} we have
\begin{align}
 \partial^n_x \ub(\s;\t,x,k) &= \int_{k}^{+\infty} \partial^n_x\Gbs\left(\s^2 \t, x-\frac{\s^2\t}{2}-y\right) \left(e^{y} - e^{k}\right) d y=
\intertext{(since $\partial_x \Gbs = - \partial_y \Gbs$ and integrating by parts)}
 & =  \int_{k}^{\infty} \partial^{n-1}_x  \Gbs\left(\s^2 \t, x-\frac{\s^2\t}{2}-y\right)  e^{y} d
 y.\label{eq:ste_400}
\end{align}
Thus, by the Gaussian estimate \eqref{eq:homogeneous_gaussian_estimate_deriv} with $c=2$ we obtain
\begin{align}
 \left|\partial^n_x \ub(\s;\t,x,k)\right|  &\le C\left(\sigma \sqrt{\tau}\right)^{-n+1}  \int_{\R} \Gbs\left(2\s^2 \t, x-\frac{\s^2\t}{2}-y\right)  e^{y} d
 y= {C}\,e^{x+\frac{\s^{2} \tau}{2}} \left(\s\sqrt{\t}\right)^{-n+1}
\end{align}
which proves the statement for $m=0$. The case $m\geq 1$ now trivially stems from the identity
\begin{equation}\label{eq:ste33}
\partial_k \ub(\s;\t,x,k) =  \ub(\s;\t,x,k) -\partial_x \ub(\s;\t,x,k),
\end{equation}
along with \eqref{eq:ste01} with $m=0$.
\end{proof}

\begin{proposition}\label{p1} { Fix $(t,T,k,\sig)$ and let $\z=\frac{x-k-\frac{\s^{2}\t}{2}}{\s\sqrt{2\t}}$ and $\tau=T-t$.  Then for any $n \geq 2$ we have
\begin{align}
 \frac{\p_\sig^n u^{\BS}(\sig)}{\p_\sig u^{\BS}(\sig)}
    &=  \sum_{q=0}^{\left\lfloor n/2 \right\rfloor}\sum_{p=0}^{n-q-1}
            c_{n,n-2q}\sig^{n-2q-1} \tau^{n-q-1} \binom{n-q-1}{p}
            \(\frac{1}{\sig\sqrt{2\tau}}\)^{p+n-q-1} \mathbf{H}_{p+n-q-1}(\zeta), \label{eq:2}
\end{align}}
where the coefficients $(c_{n,n-2k})$ are defined recursively by
\begin{align}
c_{n,n}
    &=1 , &
    &\text{and}&
c_{n,n-2q}
    &=(n-2q+1) c_{n-1,n-2q+1} + c_{n-1,n-2q-1}, &
q &\in \{ 1, 2, \cdots , \left\lfloor n/2 \right\rfloor \}.
\end{align}
\end{proposition}
\begin{proof}
See Proposition 3.5 in \cite{LPP2}.
\end{proof}

\begin{lemma}\label{ste:l3}
For any 
$m{,q,n}\in \mathbb{N}_0$ {with $m+q+n>0$} we have
\begin{align}\label{eq:ste02}
 \left|\partial^n_{\s}{\partial^q_{\t}}\p_{k}^{m} \ub(\s;\t,x,k)\right|&\le Ce^{k} { \s^{-n{+2q}}} \left(\s\sqrt{\t}\right)^{1-m{-2q}}
 ,\qquad x,k \in \R, \quad {0<\s\sqrt{\t} \leq M,}
\end{align}
where $C$ is a positive constant only dependent on {$m$, $q$, $n$ and $M$. If $q=0$, then $C$ is
independent of $M$.}.
\end{lemma}
\begin{proof}We split the proof in three steps.

{\medskip\noindent {\bf [Step 1: case $q=n=0$]}. \\} Here we will denote by $C$ any generic
constant that depends at most on {$m$}. For any $m\in\N$, by \eqref{eq:BS_convolution} we have
\begin{align}
\p_{k}^{m} \ub(\s;\t,x,k) &
= \p_k^{m-1} \left( e^k \int_k^{\infty} \Gbs\left(\s^2 \t, y-x+\frac{\s^2\t}{2}\right) d y \right)
\\ & = \sum_{i=0}^{m-1} \binom{m-1}{i} e^{k} \partial_{k}^i  \int_k^{\infty} \Gbs\left(\s^2 \t,
y-x+\frac{\s^2\t}{2}\right) d y.\label{eq:ste_estim_dk_BS}
\end{align}
Now, we have $\int_k^{\infty} \Gbs\left(\s^2 \t, y-x+\frac{\s^2\t}{2}\right) d y \in ]0,1[$ and,
for $i\geq 1$, we have
\begin{align}
\partial_{k}^i  \int_k^{\infty} \Gbs\left(\s^2 \t, y-x+\frac{\s^2\t}{2}\right) d y & = -\partial_{k}^{i-1} \Gbs\left(\s^2 \t,
k-x+\frac{\s^2\t}{2}\right).
\end{align}
Thus by applying the Gaussian estimate \eqref{eq:homogeneous_gaussian_estimate_deriv} with $c=2$,
we obtain
\begin{align}
\left|  \partial_{k}^i  \int_k^{\infty} \Gbs\left(\s^2 \t, y-x+\frac{\s^2\t}{2}\right) d y \right|
\leq C \left(\sigma \sqrt{2\tau}\right)^{-i+1}    \Gbs\left(2\s^2 \t, k-x+\frac{\s^2\t}{2}\right)
\leq C \left(\sigma \sqrt{\tau}\right)^{-i},
\end{align}
which, combined with \eqref{eq:ste_estim_dk_BS}, proves \eqref{eq:ste02}.

{\medskip\noindent {\bf [Step 2: case $q=0$, $n\geq 1$]}. \\} Here we will denote by $C$ any
generic constant that depends at most on {$m$ and $n$}. A direct computation shows
\begin{equation}\label{ae4}
 \partial_{\sigma}u^{\BS}(\sigma;\t,x,k)
  =e^k \sigma \t  \Gbs\left(\s^2 \t, x-k-\frac{\s^2\t}{2}\right)=e^k \sqrt{\t}\,\Gbs\left(1, \zeta\right),
\end{equation}
with $\z=\frac{x-k-\frac{\s^{2}\t}{2}}{\s\sqrt{2\t}}$. Therefore we have
\begin{align}\label{eq:estim_vega}
  0< \partial_{\sigma}u^{\BS}(\sigma;\t,x,k)\leq \frac{e^k \sqrt{\t} }{\sqrt{2 \pi
  }}, \qquad x,k\in\R,\ \s,\t\in\R_{>0},
\end{align}
which proves \eqref{eq:ste02} for $n=1$ and $m=0$. { Notice that
\begin{equation}\label{a10}
  \left|\p_{k}^{m}\Gbs\left(1, \zeta\right)\right|=\frac{1}{\left(\s\sqrt{2\t}\right)^{m}}\left|\p_{\zeta}^{m}\Gbs\left(1,
  \zeta\right)\right|\le C\left(\s\sqrt{\t}\right)^{-m},\qquad m\in\N_{0},
\end{equation}
where the last inequality follows from \eqref{eq:homogeneous_gaussian_estimate_deriv}.} Then, by
differentiating \eqref{ae4}, it is straightforward to show that
\begin{equation}\label{ae5}
  \left|\p_{\s}\p_{k}^{m}\ub(\s;\t,x,k)\right|\le C
  e^{k}\sqrt{\t}\left(\s\sqrt{\t}\right)^{-m},\qquad   m\in\N_0.
\end{equation}
For $n\geq 2$, by combining Proposition \ref{p1} with \eqref{ae4}, we have {
\begin{equation}\label{ae7}
\begin{split}
 \p_\sig^n \ub(\s;\t,x,k) =\,&  e^k \sqrt{\t}\sum_{q=0}^{\left\lfloor n/2 \right\rfloor}\sum_{p=0}^{n-q-1}
            c_{n,n-2q}\sig^{n-2q-1} \tau^{n-q-1} \binom{n-q-1}{p}\cdot\\
            &\cdot\(\frac{1}{\sig\sqrt{2\tau}}\)^{p+n-q-1} \Gbs\left(1, \zeta\right)\mathbf{H}_{p+n-q-1}(\zeta).
\end{split}
\end{equation}
Now notice that
\begin{align}\label{ae8}
  \left|\p_{k}^{m} \left(\Gbs\left(1, \zeta\right) \mathbf{H}_{p}(\zeta)\right)\right|&=\left|\p_{k}^{m} \p^{p}_{\zeta}\Gbs\left(1, \zeta\right) \right|
  =\frac{1}{\left(\s\sqrt{\t}\right)^{m}}\left|\p^{m+p}_{\zeta}\Gbs\left(1, \zeta\right) \right|
  \le C \left(\s\sqrt{\t}\right)^{-m}.
\end{align}
}
Then the thesis follows by differentiating formula \eqref{ae7} and using \eqref{ae8}.

{\medskip\noindent {\bf [Step 3: case $q\geq 1$]}. \\ Here we will denote by $C$ any generic
constant that depends at most on {$m$, $q$, $n$ and $M$}. By applying the identity
\begin{equation}
\partial_{\tau}u^{\BS}(\sigma;\t,x,k) = \frac{\s^2}{2}\big( \partial^2_x - \partial^2_x \big) u^{\BS}(\sigma;\t,x,k) = \frac{\s^2}{2}\big( \partial^2_k - \partial^2_k \big) u^{\BS}(\sigma;\t,x,k)
\end{equation}
we get
\begin{equation}
\partial^n_{\s}{\partial^q_{\t}}\p_{k}^{m} \ub(\s;\t,x,k) = \partial_k^m \big( \partial^2_k - \partial_k \big)^q\, \partial_{\s}^n \bigg(  \bigg(\frac{\s^2}{2} \bigg)^q \ub(\s;\t,x,k)  \bigg).
\end{equation}
The statement now follows by applying Faa di Bruno's formula (Proposition \ref{prop:multivariate_faa}) along with \eqref{eq:ste02} for $q=0$.
}

\end{proof}

\section{Explicit representation for the volatility expansion}\label{app:un_represent_theorem}
Here we recall an explicit representation formula for the $n$-th order correcting terms $u_n$ and
$\s_n$ appearing in the price expansion \eqref{eq:v.expand} and the implied volatility expansion
\eqref{ae6}, respectively. The following result is a particular case of \cite[Theorem 3.2]{LPP4}.
\begin{theorem}\label{th:un_general_repres}
Let $N\in\N$, $\bar{z}\in\R^d$ and assume that $D^{\beta}_z a_{\alpha}(\cdot,\bar{z})\in
L^{\infty}([0,T])$ for any $1\leq|\alpha|\leq 2$ and $|\beta|\leq N$. Then, for any $1\leq n\leq
N$, the function $u_n$ in \eqref{eq:v.n.pide} is given by
\begin{align}\label{eq:un}
u^{(\bar{z})}_n(t,z)
    &=  \Lc^{(\bar{z})}_n(t,T,z) u^{(\bar{z})}_0(t,z),\qquad t\in[0,T[,\ z\in\mathbb{R}^d.
\end{align}
In \eqref{eq:un}, $\Lc^{(\bar{z})}_n(t,T,z)$ denotes the differential operator acting on the
$z$-variable and defined as
\begin{align}\label{eq:def_Ln}
 \Lc_n^{(\bar{z})}(t,T,z):=  \sum_{h=1}^n \int_{t}^T d s_1 \int_{s_1}^T d s_2 \cdots \int_{s_{h-1}}^T d s_h
      \sum_{i\in I_{n,h}}\Gc^{(\bar{z})}_{i_{1}}(t,s_1,z) \cdots \Gc^{(\bar{z})}_{i_{h}}(t,s_h,z) ,
\end{align}
where\footnote{For instance, for $n=3$ we have $I_{3,3}=\{(1,1,1)\}$, $I_{3,2}=\{(1,2),(2,1)\}$
and $I_{3,1}=\{(3)\}$. }
\begin{align}\label{eq:def_Ln_bis}
 I_{n,h}
    =  \{i=(i_{1},\dots,i_{h})\in\mathbb{N}^{h} \mid i_{1}+\dots+i_{h}=n\} ,\qquad 1 \le h \le n ,
\end{align}
and the operator $\Gc^{(\bar{z})}_{n}(t,s,z)$ is defined as
\begin{align}\label{def_Gn}
\Gc^{(\bar{z})}_{n}(t,s,z)
    &:= \Ac^{(\bar{z})}_n\big(s,z-\bar{z}+\mv^{(\bar{z})}(t,s)+ \Cv^{(\bar{z})}(t,s)\nabla_{z}\big),
\end{align}
with $\mv^{(\bar{z})}(t,s)$ and $\Cv^{(\bar{z})}(t,s)$ being, respectively, the vector and the
matrix whose components are given by
\begin{align}
 \mv^{(\bar{z})}_i(t,s) = \int_t^s a_i (r,\bar{z}) d r,\qquad  \Cv^{(\bar{z})}_{ij}(t,s) =  \int_t^s a_{ij} (r,\bar{z}) d r,\qquad  i,j=1,\dots,d.
\end{align}
\end{theorem}
\begin{corollary}\label{lem:ste10}
{Let $N\in\N_0$, and let Assumption \ref{assum2and2} be in force. Then, f}or any {$n,m,q \in \N_0$
with $n,2q\leq N$, 
and for any multi-index $\a \in\mathbb{N}_0^{d}$, we have
\begin{equation}\label{eq:ste37}
{\partial^q_T} \partial^m_k D^{\a}_{z} u^{(\bar{z})}_n (t,z;T,k) = \sum_{0\leq |\gamma|\leq n
\atop {1\leq j \leq 3n}} { f^{ (n,q,m,\a)}_{\gamma,j}}(\bar{z};t,T) (z-\bar{z})^{\gamma}
 \partial_{z_1}^{{j+m+2q+\a_1}} 
 u^{(\bar{z})}_0(t,z;T,k),
\end{equation}
with
\begin{equation}\label{eq:ste36b}
 \left|f^{{(n,q,m,\a)}}_{\gamma,j}(\bar{z};t,T)\right|  \leq C {M^{q}}(M(T-t))^{\frac{{n-|\gamma|+j}}{2}},
\end{equation}
for any $0\leq t<T< {T_0}$, {$z,\bar{z}\in D(z_0,r)$} and $k\in\R$. Consequently, we have
\begin{equation}\label{eq:ste36cbis}
 \left|{\partial^q_T}\partial^m_k u^{(z)}_0 (t,z;T,k)\right|  \leq C e^{x}{M^{q}}(M(T-t))^{\frac{(1-m{ -2 q})\wedge 0}{2}}.
\end{equation}
and, for $n\geq 1$,
\begin{equation}\label{eq:ste36c}
 \left|{\partial^q_T}\partial^m_k u^{(z)}_n (t,z;T,k)\right|  \leq C e^{x}{M^{q}}(M(T-t))^{\frac{n+1-m{ -2 q}}{2}}.
\end{equation}
In \eqref{eq:ste36b}, \eqref{eq:ste36cbis} and \eqref{eq:ste36c}, $C$ is a positive constant only
dependent on $\e,\cem, T_{0}, N, |\a|$ and $m$.
}\end{corollary}
\proof Using the explicit formulas \eqref{eq:un}-\eqref{eq:def_Ln} and noting that
$u^{(\bar{z})}_0(t,z;T,k)$ does not depend on $z_{2},\dots,z_{d}$, it is straightforward to prove
that
\begin{equation}\label{eq:ste35}
 u^{(\bar{z})}_n (t,z;T,k) = \sum_{|\gamma|\leq n \atop 0\leq j \leq 3n} {f}^{(n)}_{\gamma,j}(\bar{z};t,T) (z-\bar{z})^{\gamma}\partial_{z_1}^{j} u^{(\bar{z})}_0(t,z;T,k),
\end{equation}
with 
\begin{equation}\label{eq:ste36}
 \big|{\partial^i_T}{f}^{(n)}_{\gamma,j}(\bar{z};t,T)\big|  \leq C {M^{i}}(M(T-t))^{\frac{{n-|\gamma|+j-2i}}{2}},\qquad 0\leq 2i\leq
 N.
\end{equation}
The general statement now follows from \eqref{eq:ste35}-\eqref{eq:ste36} along with the identities
\eqref{eq:ste33} and
\begin{equation}\label{eq:ste33_bis}
\partial_T u^{(\bar{z})}_0(t,z;T,k) = \frac{a_{11}(T,\bar{z})}{2}\big(\partial^2_{z_1}-\partial_{z_1}\big) u^{(\bar{z})}_0(t,z;T,k).
\end{equation}
Estimate \eqref{eq:ste36cbis} follows from Lemma \ref{ste:l2}. By combining
\eqref{eq:ste37} with \eqref{eq:ste01} eventually we get estimate \eqref{eq:ste36c}.
\endproof

Furthermore, we recall the following result \cite[Proposition 3.6]{LPP2}.
\begin{proposition}
\label{prop.un}
For every $n \in \mathbb{N}$ and $\bar{z}\in\R^d$, the ratio ${u^{(\bar{z})}_n}/{\partial_\sig
u^\BS\big(\sig^{(\bar{z})}_0\big)}$ in \eqref{eq:sig.n_bis} is a finite sum of the form
\begin{align}
 \frac{u^{(\bar{z})}_n}{\partial_\sig u^\BS\big(\sig^{(\bar{z})}_0\big)}
    =  \sum_m \({\sig^{(\bar{z})}_0\sqrt{2(T-t)}}\)^{-m} \chi^{(\bar{z})}_{m,n}\, \mathbf{H}_{m}\(\zeta\) ,\qquad   \zeta =\frac{x-k-\frac{1}{2}\sig_0^2 (T-t)}{\sig_0\sqrt{2(T-t)}}  \label{eq:form}
\end{align}
for any $t<T$, $z=(x,y)\in \R^d$ and $k\in\R$, where the coefficients $\chi^{(\bar{z})}_{m,n}=\chi^{(\bar{z})}_{m,n}(t,z;T,k)$ are explicit functions, polynomial in the log-moneyness $(k-x)$. Here, $\mathbf{H}_m$ represents the $m$-th order Hermite polynomial defined in
\eqref{eq:def_hermite}.
\end{proposition}

%
%

\section{Multivariate Fa\`a di Bruno's formula and Bell polynomials}\label{append:faa_bell}
In this section we recall a multivariate version of the well-known Fa\`a di Bruno's formula (see
\citet*{Riordan} and \citet*{Johnson}) and more precisely, its Bell polynomial version.

{For greater convenience, we recall some elements of tensorial calculus. For any given $n,h\in
\N$, we denote by $\tens$ a \emph{rank-$h$ tensor on $\R^{n}$}, i.e. an array
$\tens=(\tens_{i})_{i\in \{1,\dots, n\}^h}$, with $\tens_{i}\in\R$. Moreover, by definition a
\emph{rank-$0$ tensor} is a real number, independently of the dimension $n$.}

Let us now fix the dimension $n\in \N$. For any couple of tensors $\tens$, $\tensbis$ of rank
$h_1$ and $h_2$ respectively, we define the \emph{tensorial product} $\tens \otimes \tensbis$ as
the rank-$(h_1+h_2)$ tensor given by
\begin{equation}\label{eq:tens_product}
 \tens \otimes \tensbis_{i_1,\dots,i_{h_1},i_{h_1+1},\dots,i_{h_1+h_2}} =  \tensbis_{i_1,\dots,i_{h_1}} \tens_{i_1,\dots,i_{h_2}} ,\qquad i\in\{1,\dots,n\}^{h_1+h_2}.
\end{equation}
We also set
$\tens^{0}=1$, $\tens^{1}=\tens$ and
 $$
 \tens^{i} := \tens\underbrace{\otimes \tens \otimes \cdots \otimes}_{(i-1) \text{ times}} \tens, \qquad i\ge2. $$
Furthermore, if $\tens$ and
$\tensbis$ have the same rank $h$, we define the \emph{tensorial scalar product}
$\tens\ast\tensbis$ as the rank-$0$ tensor given by
\begin{equation}\label{eq:tens_scal_prod}
 \tens\ast\tensbis = \sum_{i\in\{1,\dots,n\}^h} \tens_i \tensbis_i.
\end{equation}
We say that a rank-$h$ tensor $\tens$ is \emph{symmetric} if $\tens_i =\tens_{\nu(i)}$ for any
$i\in\{1,\dots, n\}^h$ and for any permutation $\nu$ of the indexes $(i_1,\dots,i_h)$.

Consider now a polynomial $p$ in the variables $x=(x_1,\dots,x_j)$, homogeneous of degree $h$, of
the form
{\begin{equation}\label{eq:scalar_prod_pol}
 p(x)=\sum_{\beta\in\N_0^j\atop |\beta|= h}b_{\beta} x_{1}^{\beta_{1}} \cdots x_j^{\beta_j}.
\end{equation}
} For any rank-$h$ symmetric tensor $\tens$ and any family of rank-$1$ tensors
$\{\tensbis_1,\dots,\tensbis_j\}$, it is well defined the
scalar{
\begin{equation}
 \tens\ast p(\tensbis_1,\dots,\tensbis_j) =\tens\ast
 \sum_{\beta\in\N_0^j\atop |\beta|= h}b_{\beta} \tensbis_1^{\beta_1}\otimes \cdots \otimes\tensbis_j^{\beta_j}.
\end{equation}
}Note that, the tensor $p(\tensbis_1,\dots,\tensbis_j)$ is not well-defined on its own because
\emph{the tensorial product \eqref{eq:tens_product} is not commutative}. Nevertheless, by assuming
$\tens$ to be symmetric, the scalar product \eqref{eq:scalar_prod_pol} is well-defined as it does
not depend on the specific order of the tensorial products inside the sum.

We are ready to state the following
\begin{proposition}[Multivariate Fa\`a di Bruno's formula]\label{prop:multivariate_faa}
Let $G:\R\to\R^n$ and $F:\R^n \to \R$ be two smooth functions. Then, for any $m\in\N$ we have
{\begin{align}\label{eq:Faa_di_Bruno_appendix}
 \frac{\dd^m}{d x^m}F(G(x))=\sum_{h=1}^m \left(\nabla^{h}F\right)(G(x))\ast \mathbf{B}_{m,h}\left( \frac{\dd}{d x}G(x),\frac{\dd^2}{d
x^2}G(x),\dots,\frac{\dd^{m-h+1}}{d x^{m-h+1}}G(x) \right),
\end{align}
}where $\nabla^{h}F$ is the rank-$h$ tensor with dimension $n$ of the $h$-th order partial
derivatives of $F$, i.e.
\begin{equation}
 \nabla^{h}F_{i} = \partial_{i_1}\cdots \partial_{i_h} F, \qquad i\in\{1,\dots,n\}^h,
\end{equation}
and $\mathbf{B}_{m,h}$ is the family of the Bell polynomials defined as
\begin{equation}\label{eq:Bell_polyn_appendix}
\mathbf{B}_{m,h}(z)=\sum_{j_1, j_2,\dots, j_{m-h+1}} \frac{m!}{j_1! j_2! \cdots j_{m-h+1}!}\left(
\frac{z_1}{1!} \right)^{j_1} \left( \frac{z_2}{2!} \right)^{j_2}\cdots \left(
\frac{z_{m-h+1}}{(m-h+1)!} \right)^{j_{m-h+1}}, \qquad 1\leq h \leq m,
\end{equation}
where the sum is taken over all sequences $j_1, j_2,\dots, j_{m-h+1}$ of non-negative integers
such that
\begin{align}\label{eq:relation_indexes_bell}
 j_1+j_2+\cdots + j_{m-h+1}=h\quad\text{and}\quad  j_1+2 j_2+\cdots +(m-h+1) j_{m-h+1}=m.
\end{align}
\end{proposition}

%
%

\begin{footnotesize}
\bibliographystyle{chicago}
\bibliography{BibTeX-Final}
\end{footnotesize}

\end{document}